\documentclass[12pt]{article}

\usepackage{amsmath,amssymb,amsthm}
\usepackage[colorlinks=true,allcolors=black,bookmarksopen,bookmarksdepth=3]{hyperref}
\usepackage[margin=1in]{geometry}
\usepackage{tikz-cd}

\usepackage{enumitem}
\setitemize{label=$\bullet$,topsep=1pt,itemsep=1pt,partopsep=0pt,parsep=0pt}

\newtheorem{theorem}{Theorem}[section]
\newtheorem{lemma}[theorem]{Lemma}
\newtheorem{proposition}[theorem]{Proposition}
\newtheorem{corollary}[theorem]{Corollary}

\theoremstyle{definition}
\newtheorem{definition}[theorem]{Definition}

\theoremstyle{remark}
\newtheorem{remark}[theorem]{Remark}

\newtheorem{example}[theorem]{Example}
\newtheorem{conjecture}[theorem]{Conjecture}

\numberwithin{equation}{section}

\newcommand{\0}{\mathbf 0}
\newcommand{\1}{\mathbf 1}
\newcommand{\bb}{\mathbf b}

\newcommand{\ZZ}{\mathbf Z}
\newcommand{\CC}{\mathbf C}
\newcommand{\RR}{\mathbf R}
\newcommand{\HH}{\mathbf H}
\newcommand{\EE}{\mathbf E}
\newcommand{\BB}{\mathbf B}
\newcommand{\bO}{\mathbf{bO}}
\newcommand{\BO}{\mathbf{BO}}
\newcommand{\mO}{\mathbf{mO}}
\newcommand{\MO}{\mathbf{MO}}

\newcommand{\SSS}{\mathbf{S}}

\newcommand{\cP}{\mathcal P}
\newcommand{\F}{\mathcal F}

\newcommand{\fS}{\mathfrak S}

\newcommand{\im}{\operatorname{im}}
\newcommand{\id}{\operatorname{id}}

\newcommand{\Eq}{\operatorname{Eq}}
\newcommand{\Coeq}{\operatorname{Coeq}}
\newcommand{\Emb}{\operatorname{Emb}}
\newcommand{\Hom}{\operatorname{Hom}}
\newcommand{\Inj}{\operatorname{Inj}}
\newcommand{\End}{\operatorname{End}}
\newcommand{\Aut}{\operatorname{Aut}}
\newcommand{\Nbd}{\operatorname{Nbd}}
\newcommand{\Maps}{\operatorname{Maps}}
\newcommand{\Mapstilde}{\widetilde{\operatorname{Maps}}}
\newcommand{\RepMapstilde}{\widetilde{\operatorname{RepMaps}}}
\newcommand{\Gr}{\operatorname{Gr}}
\newcommand{\GL}{\operatorname{GL}}
\newcommand{\Iso}{\operatorname{Iso}}

\newcommand{\colim}{\operatornamewithlimits{colim}}

\newcommand{\cyl}{\operatorname{cyl}}
\newcommand{\hq}{\mathord{/\!/}}
\newcommand{\Str}{\operatorname{Str}}

\newcommand{\Top}{\mathrm{Top}}
\newcommand{\kTop}{\mathrm{kTop}}
\newcommand{\Set}{\mathrm{Set}}
\newcommand{\PSh}{\mathrm{PSh}}
\newcommand{\Ab}{\mathrm{Ab}}
\newcommand{\Grpd}{\mathrm{Grpd}}
\newcommand{\FinGrp}{\mathrm{FinGrp}}
\newcommand{\InjFinGrp}{\mathrm{InjFinGrp}}
\newcommand{\Spc}{\mathrm{Spc}}
\newcommand{\OrbSpc}{\mathrm{OrbSpc}}
\newcommand{\OrbSpcPair}{\mathrm{OrbSpcPair}}
\newcommand{\SpcPair}{\mathrm{SpcPair}}
\newcommand{\RepOrbSpcPair}{\mathrm{RepOrbSpcPair}}
\newcommand{\RepOrbSpc}{\mathrm{RepOrbSpc}}
\newcommand{\OrthSpc}{\mathrm{OrthSpc}}
\newcommand{\GloSpc}{\mathrm{GloSpc}}
\newcommand{\Sp}{\mathrm{Sp}}
\newcommand{\OrbSp}{\mathrm{OrbSp}}
\newcommand{\RepOrbSp}{\mathrm{RepOrbSp}}
\newcommand{\OrthSp}{\mathrm{OrthSp}}
\newcommand{\GloSp}{\mathrm{GloSp}}

\newcommand{\rep}{\mathrm{rep}}
\newcommand{\Rep}{\mathrm{Rep}}

\newcommand{\C}{\mathrm C}
\newcommand{\op}{\mathrm{op}}

\newcommand{\LL}{\mathbb L}
\newcommand{\OO}{\mathbb O}

\newcommand{\Shv}{\mathrm{Shv}}
\newcommand{\Vect}{\mathrm{Vect}}
\newcommand{\Ind}{\operatorname{Ind}}
\newcommand{\Pro}{\operatorname{Pro}}
\newcommand{\Ho}{\operatorname{Ho}}

\newcommand{\der}{\mathrm{der}}
\newcommand{\iso}{\mathrm{iso}}

\newcommand{\st}{\mathrm{st}}
\newcommand{\cst}{\mathrm{cst}}
\newcommand{\fr}{\mathrm{fr}}

\newcommand{\acts}{\curvearrowright}
\newcommand{\righttwoarrows}{\mathrel{\vcenter{\mathsurround0pt{\ialign{##\crcr\noalign{\nointerlineskip}$\rightarrow$\crcr\noalign{\nointerlineskip}$\rightarrow$\crcr}}}}}
\newcommand{\rightthreearrows}{\mathrel{\vcenter{\mathsurround0pt{\ialign{##\crcr\noalign{\nointerlineskip}$\rightarrow$\crcr\noalign{\nointerlineskip}$\rightarrow$\crcr\noalign{\nointerlineskip}$\rightarrow$\crcr}}}}}
\newcommand{\lefttwoarrows}{\mathrel{\vcenter{\mathsurround0pt{\ialign{##\crcr\noalign{\nointerlineskip}$\leftarrow$\crcr\noalign{\nointerlineskip}$\leftarrow$\crcr}}}}}
\newcommand{\leftthreearrows}{\mathrel{\vcenter{\mathsurround0pt{\ialign{##\crcr\noalign{\nointerlineskip}$\leftarrow$\crcr\noalign{\nointerlineskip}$\leftarrow$\crcr\noalign{\nointerlineskip}$\leftarrow$\crcr}}}}}

\begin{document}

\title{Orbifold bordism and duality for finite orbispectra}

\author{John Pardon\thanks{This research was conducted during the period the author served as a Clay Research Fellow and was partially supported by a Packard Fellowship and by the National Science Foundation under the Alan T.\ Waterman Award, Grant No.\ 1747553.}}

\date{August 16, 2021}

\maketitle

\begin{abstract}
We construct the stable (representable) homotopy category of \emph{finite orbispectra}, whose objects are formal desuspensions of finite orbi-CW-pairs by vector bundles and whose morphisms are stable homotopy classes of (representable) relative maps.
The stable representable homotopy category of finite orbispectra admits a contravariant involution extending Spanier--Whitehead duality.
This duality relates homotopical cobordism theories (cohomology theories on finite orbispectra) represented by global Thom spectra to geometric (derived) orbifold bordism groups (homology theories on finite orbispectra).
This isomorphism extends the classical Pontryagin--Thom isomorphism and its known equivariant generalizations.
\end{abstract}

\section{Introduction}

The classical Pontryagin--Thom isomorphism \cite{pontryagin,pontryagintranslation,thom} equates manifold bordism groups $\Omega_*(X)$ with corresponding stable homotopy groups $[S,X\wedge MO]$ for spaces $X$.
When $X$ is a $G$-space ($G$ a compact Lie group), equivariant versions of this isomorphism are well-studied \cite{wasserman,tomdieck,connerfloyd,brockerhook,schwedeglobal}.

A main result of this paper is to construct the Pontryagin--Thom isomorphism in the homotopy theory of \emph{orbispaces} \cite{haefligerorbiespaces}.
The basic objects of this homotopy theory are \emph{orbi-CW-complexes}, which are built like CW-complexes by attaching cells of the form $(D^k,\partial D^k)\times\BB G$ for finite groups $G$ along \emph{representable} maps \cite{gepnerhenriques}.
(The more general setting in which one allows compact Lie groups in place of finite groups is unfortunately beyond the scope of this paper, most significantly due to the failure of `enough vector bundles' in this context.)
The most familiar instance of orbispaces in topology is probably orbifolds; moduli spaces of solutions to elliptic partial differential equations, as they appear in low-dimensional and symplectic topology, are also best regarded as orbispaces, and they provide some of the motivation for our present investigation.

The Pontryagin--Thom isomorphism relates `geometric bordism theories' with `homotopical cobordism theories' for orbispaces $X$.
In our setting, the relevant geometric bordism theories $\Omega_*(X)$ are given by bordism classes of (possibly `derived') orbifolds with a representable map to $X$ (and possibly with some sort of tangential structure).
The homotopical cobordism theories relevant for us are those associated to the global Thom spectra defined by Schwede \cite{schwedeglobal}.
These theories (on both the geometric side and the homotopical side) come in two flavors; on the geometric side, these correspond to the adjectives `ordinary' and `derived'.
The difference between ordinary and derived bordism measures the failure of equivariant transversality.

The Pontryagin--Thom isomorphism between geometric bordism and homotopical cobordism passes through the category of \emph{finite representable orbispectra} and a contravariant `duality' involution on this category.
The construction of this category and of its involution are our remaining main results.
They both rely crucially on the fact, proven in \cite{orbibundle}, that compact orbispaces admit `enough vector bundles' (the assertion that a given compact orbi-CW-complex $X$ admits enough vector bundles is equivalent to the assertion that $X$ is homotopy equivalent to a compact effective orbifold with boundary; \emph{effective} means that in the local models $\RR^n/G$ or $(\RR^{n-1}\times\RR_{\geq 0})/G$, the homomorphism $G\to\GL_n(\RR)$ is \emph{injective}).
Enough vector bundles also underlies much of the other reasoning in this paper, including the definition of derived orbifold bordism groups, the extension of geometric bordism theories to orbispectra, and the relation between orbi-CW-complexes and the global homotopy theory from \cite{schwedeglobal}.

Before stating our main results more formally, we give a concrete example to motivate the more abstract discussion which follows.

\begin{example}\label{orbibordismex}
We describe a stable homotopy theoretic realization of the bordism group of closed orbifolds, which for reasons which will become apparent shortly, we denote by $\Omega_*(R(*))$.
This group has been studied by Druschel \cite{druschelI,druschelII,druschelIII}, \'Angel \cite{angelss,angelbdry}, and Sarkar \cite{sarkar}.

The first main point of the Pontryagin--Thom construction for manifolds is to note that every manifold $M$ admits a homotopically unique embedding into $\RR^N$ as $N\to\infty$, in the sense that the space of embeddings $M\hookrightarrow\RR^N$ becomes highly connected in the limit $N\to\infty$.
We therefore seek a corresponding sequence of orbifolds $X_N$ with the property that every orbifold $M$ admits a homotopically unique embedding into $X_N$ in the limit $N\to\infty$.
In this pursuit, it is helpful to separate the two key properties of $\RR^N$ which give rise to the unique embedding property for manifolds: it is contractible (so everything has a homotopically unique map to it) and high-dimensional (so the locus of maps which are not embeddings has arbitrarily high codimension as $N\to\infty$).

Now if we are seeking an embedding of orbifolds $M\hookrightarrow X_N$, we should first note that an embedding is necessarily \emph{representable}, so we should not seek $X_N$ which are contractible, rather we should seek $X_N$ with the property that the space of \emph{representable} maps to $X_N$ is contractible (for every domain orbispace).
This universal property defines a unique homotopy type which we denote by
\begin{equation}
R(*):=\bigsqcup_{G_0\hookrightarrow\cdots\hookrightarrow G_p}\BB G_0\times\Delta^p\Bigm/{\sim}
\end{equation}
where the right side is modelled on the nerve of the $2$-category of finite groups, injective maps, and conjugations.
It is straightforward to check that $R(*)$ has the desired property: it is enough (by an obstruction theory argument) to show that the space of representable maps $\BB G\to R(*)$ is contractible for every finite group $G$, and this space is
\begin{equation}
\bigsqcup_{G_0\hookrightarrow\cdots\hookrightarrow G_p}\RepMapstilde(\BB G,\BB G_0)\times\Delta^p\Bigm/{\sim}\quad=\quad\bigsqcup_{G\hookrightarrow G_0\hookrightarrow\cdots\hookrightarrow G_p}\Delta^p\Bigm/{\sim}
\end{equation}
which is contractible as it is the nerve of a category with an initial object (the under-category of $G$ in the $2$-category of finite groups, injective maps, and conjugations).
Thus, in particular, every compact orbifold $M$ admits a homotopically unique representable map $M\to R(*)$.

Next, we should realize $R(*)$ as a high-dimensional orbifold so as to ensure that the locus of representable maps $M\to R(*)$ which fail to be an embedding has arbitrarily large codimension inside the space of all maps (in fact, to guarantee this, we need more than just that the dimension of $R(*)$ is large, rather we need that when its tangent bundle is decomposed into isotypic pieces with respect to the isotropy group actions, every isotypic piece has high dimension).
Filter $R(*)$ by finite subcomplexes, and use enough vector bundles \cite{orbibundle} to realize each as a compact effective orbifold with boundary; moreover use enough vector bundles again to replace each with the total space of the unit disk bundle of a vector bundle over it whose isotypic pieces are all high-dimensional.
We thus get a sequence of compact orbifolds with boundary and smooth embeddings $X_0\hookrightarrow X_1\hookrightarrow X_2\hookrightarrow\cdots$, such that for every closed orbifold $M$, the direct limit over $i\to\infty$ of the space of embeddings $M\hookrightarrow X_i$ is contractible.

There is now an obvious Pontryagin--Thom collapse map giving, for any smooth suborbifold of $X_i$ of dimension $d$, an element of $\mO^{-d}((X_i,\partial X_i)^{-TX_i})$, where $\mO$ is the global spectrum defined by Schwede \cite{schwedeglobal} (we define the category of orbispectra which includes expressions such as $(X_i,\partial X_i)^{-TX_i}$ as objects, and we show that global spectra define cohomology theories on orbispectra).
The homotopically unique embedding property of the sequence $X_0\hookrightarrow X_1\hookrightarrow X_2\hookrightarrow\cdots$ thus gives us a map $\Omega_*(R(*))\to\varinjlim_{i\to\infty}\mO^{-*}((X_i,\partial X_i)^{-TX_i})$.
Now Theorem \ref{duality} defines an involution $D$ on the category of orbispectra which sends $X_i$ to $(X_i,\partial X_i)^{-TX_i}$, so we may formulate the Pontryagin--Thom map more intrinsically as
\begin{equation}
\Omega_*(R(*))\to\mO^{-*}(D(R(*))),
\end{equation}
where to be completely precise we should remark that $D$ is defined only on the category of finite orbispectra, so $D(R(*))$ is really an inverse system of orbispectra, to which applying the contravariant functor $\mO^{-*}$ yields a directed system of graded abelian groups, and the right side above refers to its direct limit.
Theorem \ref{ptfinal} states that this Pontryagin--Thom map is an isomorphism (for any orbispectrum in place of $R(*)$).
We thus conclude that the group of closed orbifolds modulo bordism is $\mO^{-*}(D(R(*)))$.
\end{example}

\subsection{Categories of orbispaces}

We approach the homotopy theory of orbispaces from the point of view of orbi-CW-complexes; these are built like CW-complexes from cells $(D^k,\partial D^k)\times\BB G$ for integers $k\geq 0$ and finite groups $G$, which are attached along \emph{representable} maps (this is a slight adaptation of a definition given by Gepner--Henriques \cite{gepnerhenriques}).
We denote by $\Spc$ the category of CW-complexes and homotopy classes of maps, and we denote by $\OrbSpc$ (resp.\ $\RepOrbSpc$) the category of orbi-CW-complexes and homotopy classes of all (resp.\ representable) maps.
There are thus functors
\begin{equation}\label{cwcats}
\Spc\to\RepOrbSpc\to\OrbSpc
\end{equation}
with $\Spc$ being a full subcategory of both $\RepOrbSpc$ and $\OrbSpc$.
It was pointed out already in Gepner--Henriques \cite{gepnerhenriques} that the distinction between representable vs all maps leads to two distinct theories, both of which can legitimately be called `the homotopy theory of orbispaces'.

The functor $\Spc\to\OrbSpc$ has a left adjoint $X\mapsto\left|X\right|$ (the coarse space of $X$) and a right adjoint $X\mapsto\tilde X$ (the classifying space of $X$).
The functor $\RepOrbSpc\to\OrbSpc$ also has a right adjoint which we denote by $X\mapsto R(X)$.
The orbi-CW-complex $R(*)$ plays a recurring role in our discussion; it is the terminal object of $\RepOrbSpc$, and it is what Rezk \cite{rezkcohesion} calls the `normal subgroup classifier' $\mathcal N$.

\begin{remark}
We certainly expect, but do not pursue here, $\infty$-categorical refinements of all of our constructions.
This expectation is reflected in our notation: although all of the categories under consideration in this paper are homotopy categories, we do not include the prefix $\Ho$ in their notation.
\end{remark}

Of importance are also the categories of \emph{relative orbi-CW-complexes} $\RepOrbSpc_*$ and $\OrbSpc_*$ which are analogues of the category $\Spc_*$ of pointed CW-complexes.
It should be noted, however, that $\RepOrbSpc_*$ and $\OrbSpc_*$ are not the homotopy categories of pointed orbi-CW-complexes, rather their objects are orbi-CW-pairs $(X,A)$ (meaning $X$ is an orbi-CW-complex and $A\subseteq X$ is a subcomplex) with a nontrivial notion of morphism.
The essential reason this slightly complicated definition is needed is that for an orbi-CW-pair $(X,A)$, there is no good way to collapse $A$ to a point and form a quotient orbi-CW-complex $X/A$.
We have functors
\begin{equation}\label{pointedcwcats}
\Spc_*\to\RepOrbSpc_*\to\OrbSpc_*
\end{equation}
again with $\Spc_*$ being a full subcategory of the latter two, and there is a natural map from \eqref{cwcats} to \eqref{pointedcwcats} given by adjoining a disjoint basepoint.

The categories $\RepOrbSpc$ and $\RepOrbSpc_*$ are a natural setting for homotopy theory.
The category $\RepOrbSpc_*$ has a natural notion of a cofiber sequence $X\to Y\to Z$, and every morphism $X\to Y$ in $\RepOrbSpc_*$ extends to a half-infinite `Puppe' sequence $X\to Y\to Z\to\Sigma X\to\Sigma Y\to\Sigma Z\to\Sigma^2X\to\cdots$ in which every consecutive triple is a cofiber sequence.

The natural functor $\RepOrbSpc\to\PSh(\Rep\{\BB G\})$ (at the level of $\infty$-categories or model categories) is an equivalence by Gepner--Henriques \cite{gepnerhenriques}, where $\Rep\{\BB G\}\subseteq\RepOrbSpc$ denotes the full subcategory spanned by the objects $\BB G$ for finite groups $G$.
This means that $\RepOrbSpc$ is the free cocompletion of its full subcategory $\Rep\{\BB G\}$.
We conjecture that $\RepOrbSpc$ is the category of representable fibrations over $R(*)$ (here $R$ is the right adjoint to $\RepOrbSpc\to\OrbSpc$) with `reasonable' fibers (note that the data of a representable fibration over $R(*)$ is at least intuitively comparable to the data of a presheaf on $\Rep\{\BB G\}$).
Let us also remark that both these descriptions of $\RepOrbSpc$ (and the corresponding descriptions of $\RepOrbSpc_*$) are manifestly natural settings for doing homotopy theory, whereas proving this for $\RepOrbSpc_*$ as we define it requires a somewhat explicit argument.
On the other hand, it is somewhat less apparent from these descriptions what the full subcategory of finite orbi-CW-complexes $\RepOrbSpc^f\subseteq\RepOrbSpc$ is.

The categories $\OrbSpc$ and $\OrbSpc_*$ do not seem to be a natural setting for homotopy theory (for example, there are morphisms in $\OrbSpc_*$ which do not have a cofiber in any reasonable sense).
Rather, $\OrbSpc$ (similarly for $\OrbSpc_*$) is a full subcategory of the larger category, say denoted by $\overline\OrbSpc$, obtained by gluing cells $(D^k,\partial D^k)\times\BB G$ along all (not necessarily representable) maps, as constructed by Gepner--Henriques \cite{gepnerhenriques} (note that this takes us outside the realm of stacks admitting \'etale atlases).
Gepner--Henriques \cite{gepnerhenriques} further showed that $\overline\OrbSpc$ is equivalent (again at the level of $\infty$-categories or model categories) to $\PSh(\{\BB G\})$.
This latter category $\PSh(\{\BB G\})$ was shown by Schwede \cite{schwedeequivalence} to be equivalent to the global homotopy category $\GloSpc$ defined in \cite{schwedeglobal} (with respect to the `global family' of all finite groups); see also K\"orschgen \cite{korschgen} and Juran \cite{juran}.
We will not explain in detail (nor use) the precise relationship between $\OrbSpc$ and $\overline\OrbSpc=\PSh(\{\BB G\})=\GloSpc$, rather we describe just the little bit that we need.

\subsection{Geometric bordism theories}

We consider various flavors of geometric bordism groups, all of which are sequences of functors
\begin{equation}\label{orbihomology}
Z_i:\RepOrbSpc_*\to\Ab
\end{equation}
satisfying $Z_i(\Sigma X)=Z_{i+1}(X)$ and $\bigoplus_\alpha Z_i(X_\alpha)\xrightarrow\sim Z_i(\bigsqcup_\alpha X_\alpha)$ and which sends cofiber sequences to exact sequences (such a functor might be called a \emph{homology theory} for orbispaces).

The bordism group $\Omega_*(X)$ is the set of closed orbifolds with a representable map to $X$, modulo bordism (graded by dimension); this is an abelian group under disjoint union.
To define $\Omega_*(X,A)$ for a pair $(X,A)$, we consider compact orbifolds-with-boundary $M$ and representable maps of pairs $(M,\partial M)\to(X,A)$.
Note that our notation is consistent with the usual meaning of $\Omega_*(X)$ for spaces $X$ (namely bordism classes of closed manifolds mapping to $X$) since an orbifold with a representable map to a space is necessarily a manifold.
Moreover, $\Omega_*(X/G)$ is $G$-equivariant bordism for $G$-spaces $X$ (i.e.\ bordism of $G$-manifolds mapping equivariantly to $X$).

Note that there is no additional generality to be gained by considering arbitrary (not necessarily representable) maps here, since a map to $X$ is the same as a representable map to $R(X)$ (where $R:\OrbSpc\to\RepOrbSpc$ is the right adjoint to $\RepOrbSpc\to\OrbSpc$), so bordism of orbifolds with an arbitrary map to $X$ is given by $\Omega_*(R(X))$.
For example, the group of bordism classes of closed orbifolds is $\Omega_*(R(*))$.
Filtering $R(*)$ by subcomplexes gives a spectral sequence converging to $\Omega_*(R(*))$; see \'Angel \cite{angelss} for a similar spectral sequence.

There are also derived bordism groups $\Omega_*^\der(X)$, whose elements are represented by `derived orbifold charts' $(D,E,s)$ consisting of an orbifold $D$, a vector bundle $E$ over $D$, and a section $s:D\to E$ whose zero set is compact, together with a representable map $D\to X$ (grading by `virtual dimension' $\dim D-\dim E$).
These are considered modulo restriction (removing from $D$ a closed subset disjoint from $s^{-1}(0)$), stabilization (replacing $D$ with the total space of a vector bundle $V$ over $D$, replacing $E$ with $E\oplus V$, and replacing $s$ with $s\oplus\id_V$), and bordism.

The tautological map $\Omega_*\to\Omega_*^\der$ is not generally an isomorphism; in fact $\Omega_*^\der$ is often nonzero in negative degrees $*<0$ (see Example \ref{dertonondernonzero}).
This can be viewed as a strong measurement of the fact that a vector bundle over an orbifold need not have any section which is transverse to zero.

That these derived bordism groups $\Omega_*^\der$ define a homology theory for orbispaces requires enough vector bundles.
This is related to the fact that the `proper' definition of a derived orbifold is as something with an atlas of derived orbifold charts (it would be essentially obvious that bordism of these defines a homology theory for orbispaces), and enough vector bundles implies everything has a global chart.

For any vector bundle over $X$, there are so-called `inverse Thom maps' (terminology following Schwede \cite[\S 6]{schwedeglobal})
\begin{equation}
\Omega_*(X)\to\Omega_{*+\left|V\right|}(X^V)\quad\text{and}\quad\Omega^\der_*(X)\to\Omega^\der_{*+\left|V\right|}(X^V).
\end{equation}
For derived bordism, the inverse Thom maps are isomorphisms, whereas for ordinary bordism they are isomorphisms for vector bundles with trivial isotropy representations but not in general.
In fact, similar to the situation in global homotopy theory \cite[\S 6]{schwedeglobal}, there is a precise sense in which derived bordism is the localization of bordism at the inverse Thom maps.

A remarkable result of Wasserman implies that bordism is in fact a particular instance of derived bordism with tangential structure.
Specifically, $\Omega_*$ is bordism of derived orbifolds together with a vector bundle $V$ and a stable isomorphism of vector bundles $TD-E=V-\underline\RR^k$, modulo $(V,k)\mapsto(V\oplus\underline\RR,k+1)$.
This fact fundamentally underlies the Pontryagin--Thom isomorphism for $\Omega_*$ (homotopical cobordism theories are really all derived cobordism theories with some sort of tangential structure).

We can also consider bordism of orbifolds with tangential structure.
The sort of tangential structure $\fS$ permitted (`coarsely stable' or `stable') depends on whether we are considering $\Omega_*^\fS$ or $\Omega_*^{\fS,\der}$.
We leave a precise discussion of these theories for the main body of the paper.

Geometric bordism theories may be extended to the category of orbispectra (to be discussed shortly) by twisting.
Structured derived bordism of $(X,A)^{-\xi}$ is defined as bordism of derived orbifolds over $(X,A)$ with the given structure on their tangent bundle minus $\xi$ (so to extend underived bordism to orbispectra, the key is to think of it as structured derived bordism via Wasserman).
For example, $\Omega_0^\fr((X,A)^{-\xi})$ is bordism of derived orbifolds representable over $(X,A)$ with a stable isomorphism between their tangent bundle and $\xi$.
Such twistings are the natural home for the fundamental class: given a compact orbifold with boundary $X$, it has a fundamental class $[X]\in\Omega_0^\fr((X,\partial X)^{-TX})$; orienting $TX$ with respect to some structure allows to undo the twist after pushing forward to the corresponding structured bordism group.

\subsection{Homotopical cobordism theories}

Any global spectrum \cite{schwedeglobal} defines a cohomology theory for orbispaces, namely a sequence of functors
\begin{equation}
Z^i:\OrbSpc_*\to\Ab
\end{equation}
satisfying $Z^i(\Sigma X)=Z^{i+1}(X)$ and $Z^i(\bigsqcup_\alpha X_\alpha)\xrightarrow\sim\prod_\alpha Z^i(X_\alpha)$ and which sends cofiber sequences to exact sequences.
Namely, given an orthogonal spectrum $Z:\OO\to\Top_*$, the group $Z^0(X,A)$ is the direct limit over vector bundles $E/X$ of sections of $\Omega^EZ(E)\to X$ supported away from $A$, modulo homotopy.
In fact, we may define $Z^0((X,A)^{-\xi})$ to be the direct limit of sections of $\Omega^EZ(E\oplus\xi)$, which extends $Z^*$ to the category of orbispectra (which we will meet shortly).
The viability of this definition depends on enough vector bundles (though one could formulate a more complicated definition, involving patching together choices of local vector bundles, which would not require an appeal to enough vector bundles).
We expect, but do not prove, that this definition is equivalent to that obtained from the composition $\OrbSpc\hookrightarrow\overline\OrbSpc=\PSh(\{\BB G\})=\GloSpc\xrightarrow{\Sigma^\infty}\GloSp$.

The orthogonal spectra relevant for this paper are the global Thom spectra defined by Schwede \cite[\S 6]{schwedeglobal}.
These include the global sphere spectrum $\SSS$ and the two flavors of the Thom spectrum of the infinite orthogonal group $\mO$ and $\MO$.
The associated cohomology theories are called homotopical cobordism theories.

\subsection{Categories of orbispectra}

To relate geometric bordism and homotopical cobordism requires introducing the category of \emph{orbispectra}.
We will only ever discuss \emph{finite} orbispectra, namely desuspensions of finite orbi-CW-pairs by vector bundles.
The category of `naive orbispectra' has objects of the form $\Sigma^{-n}(X,A)$, with morphisms $\Sigma^{-n}(X,A)\to\Sigma^{-m}(Y,B)$ given by the direct limit over $k$ of the space of relative morphisms $\Sigma^{k-n}(X,A)\to\Sigma^{k-m}(Y,B)$.
We are more interested in the category of `genuine orbispectra', whose objects take the form $(X,A)^{-V}$ for $V$ a vector bundle (with possibly nontrivial isotropy representations) and whose morphisms are defined by a direct limit over passing to Thom spaces of arbitrary vector bundles.

We define two homotopy categories of finite (genuine) orbispectra $\RepOrbSp^f$ and $\OrbSp^f$, again depending on whether we use representable maps or not.
They fit into a diagram
\begin{equation}
\Sp^f\to\RepOrbSp^f\to\OrbSp^f,
\end{equation}
with $\Sp^f$ (the category of finite spectra) being a full subcategory of the latter two.
The definitions of these categories use enough vector bundles (though this could probably be eliminated if one took a more abstract approach).

The categories $\Sp^f$ and $\RepOrbSp^f$ are natural settings for stable homotopy theory.
For example, every morphism in $\RepOrbSp^f$ fits into an `exact triangle' (although we do not actually prove that $\RepOrbSp^f$ is triangulated).
We conjecture that $\RepOrbSp$ (a category we do not define, but at the level of $\infty$-categories it would be $\Ind\RepOrbSp^f$) is the category of parameterized spectra over $R(*)$.

As before, $\OrbSp^f$ does not seem to be natural setting for stable homotopy theory.
It seems likely there is a functor $\OrbSp^f\to\GloSp$ (the category of global spectra \cite{schwedeglobal}), though we do not quite construct it, nor is it clear if we should expect it to be fully faithful.

\subsection{Duality}

Now a key result is to define a contravariant involution $D$ (`duality') on the category $\RepOrbSp^f$.
The construction of this functor relies crucially on enough vector bundles.

\begin{theorem}\label{duality}
The category $\RepOrbSp^f$ admits a contravariant involution $D$ preserving cofiber sequences, defined by declaring that (1) for any compact orbifold with boundary $X$ and codimension zero suborbifold with boundary $A\subseteq\partial X$, we have
\begin{equation}
D((X,A)^{-\xi})=(X,\partial X-A^\circ)^{\xi-TX}
\end{equation}
and (2) for any smooth embedding of such pairs $f:(X,A)\hookrightarrow(Y,B)$ (so $X\subseteq Y$ is a smooth suborbifold of $Y$ meeting $\partial Y$ transversely precisely in $A=X\cap B$), we have that $(Df)^{TY}$ is the obvious collapse map $(Y,\partial Y-B^\circ)\to(X,\partial X-A^\circ)^{TY/TX}$.
\end{theorem}

It is immediate from the definition that $D$ stabilizes the full subcategory $\Sp^f\subseteq\RepOrbSp^f$ and coincides on it with classical Spanier--Whitehead duality of finite spectra \cite{spanierwhiteheadduality} (the definition of $D$ is essentially identical to Atiyah's formulation \cite{atiyahduality}, just generalized to orbifolds).
However, whereas Spanier--Whitehead duality on $\Sp^f$ is characterized by the universal property of a map $X\wedge Y\to S^0$ being the same as a map $X\to DY$, we do not know a universal property characterization of the involution $D$ on $\RepOrbSp^f$.
There is at least a natural map from maps $X\to Y$ to maps $X\wedge DY\to R(*)$, but it is not an isomorphism and we do not know any sense in which it characterizes $D$; the essential reason for this is that $\wedge$ does not play well with representability (see Example \ref{dualizationasbordismmap}).
The identity map $X\to X$ thus corresponds to a canonical pairing $X\wedge DX\to R(*)$ which, upon passing to classifying spaces (note that $\widetilde{R(*)}=*$), gives a pairing $\tilde X\wedge\widetilde{DX}\to S^0$, hence a comparision map
\begin{equation}
\widetilde{DX}\to D\tilde X
\end{equation}
(where we should understand that the classifying space of an object of $\RepOrbSp^f$ is an object of $\Sp=\Ind\Sp^f$, not $\Sp^f$, so $D\tilde X$ is an object of $\Pro\Sp^f$).
The results of Greenlees--Sadofsky \cite[Corollary 1.2]{greenleessadofsky} and Cheng \cite{mccheng} may be viewed as the assertion that this comparison map is $K(n)$-local for all $n$, where $K(n)$ denotes Morava $K$-theory.

Duality allows us to define, for any global spectrum $E$, an $E$-homology functor $\RepOrbSp^f\to\Ab$ by taking
\begin{equation}
E_*(X):=E^{-*}(DX).
\end{equation}
Note that whereas $E$-cohomology is a functor on $\OrbSp^f$, we only define $E$-homology as a functor on $\RepOrbSp^f$.

Recall that in ordinary stable homotopy theory, the $E$-homology of a (finite) space $X$ is defined as $[S^0,X\wedge E]=[DX,E]$.
Due to $D$ not being the monoidal dual with respect to $\wedge$, this equality no longer holds in our context, so there are \emph{a priori} two reasonable notions of $E$-homology for orbispaces.
We consider the latter definition since it is the one which is relevant for the Pontryagin--Thom isomorphism.
The former definition (implemented in the context of global homotopy theory) is proposed by Schwede \cite{schwedeglobal} and is presumably quite different.

\subsection{Pontryagin--Thom isomorphism}

We may now state the Pontryagin--Thom isomorphism relating geometric bordism and homotopical cobordism on $\RepOrbSp^f$.

\begin{theorem}\label{ptfinal}
There are natural isomorphisms of functors on $\RepOrbSp^f$:
\begin{align}
\SSS_*&{}=\Omega^\fr_*\\
\mO_*&{}=\Omega_*\\
\MO_*&{}=\Omega_*^\der
\end{align}
\end{theorem}

\begin{example}
Under the Pontryagin--Thom isomorphism, the unit $\1\in\SSS^0(X)$ is sent to the fundamental class $[X]\in\Omega_0^\fr((X,\partial X)^{-TX})$ for any compact orbifold-with-boundary $X$.
\end{example}

\begin{example}
The orbi-CW-complex $R(*)$ is not finite, but we may nevertheless define $\Omega_*(R(*))$ and $\mO_*(R(*))$ by taking the direct limit over finite subcomplexes, and we conclude they are isomorphic (compare Example \ref{orbibordismex}).
\end{example}

The Pontryagin--Thom construction also gives a description of the morphism groups in $\RepOrbSp^f$ and in $\OrbSp^f$ in terms of bordism.

\begin{theorem}\label{stablerepresentablemapsbordism}
Let $(X,A)$ and $(Y,B)$ be compact orbi-CW-pairs carrying stable vector bundles $\xi$ and $\zeta$.
The set of morphisms
\begin{equation}
D((X,A)^{-\xi})\to(Y,B)^{-\zeta}
\end{equation}
in $\OrbSp^f$ (resp.\ $\RepOrbSp^f$) is in canonical bijection with bordism classes of derived orbifolds $(C,\partial C)$ with a representable map $f:C\to X$, a (representable) map $g:C\to Y$, such that $\partial C\subseteq f^{-1}(A)\cup g^{-1}(B)$, and a stable isomorphism between $TC$ and $f^*\xi+g^*\zeta$.
\end{theorem}

Note that this result gives multiple descriptions of the same stable mapping group, since a given object of $\OrbSp^f$ or $\RepOrbSp^f$ may be expressed as $(X,A)^{-\xi}$ in many different ways; in particular, passing from $(X,A)^{-\xi}$ to $((X,A)^V)^{-V-\xi}$ via the obvious isomorphism acts via Theorem \ref{stablerepresentablemapsbordism} on bordism classes of derived orbifolds by passing to the Thom space of the pullback of $V$ (and similarly for $(Y,B)^{-\zeta}$).
Also note that, in the case of $\RepOrbSp^f$, the description of morphisms is manifestly symmetric in $(X,A)^{-\xi}$ and $(Y,B)^{-\zeta}$, as it should be given that $D$ is an involution.

\begin{example}\label{dualizationasbordismmap}
As we remarked earlier, there is a canonical pairing $W\wedge DW\to R(*)$ in $\RepOrbSp^f$ which induces a natural transformation
\begin{equation}\label{dualitynattransnotquiteiso}
\Hom(Z,W)\to\Hom(Z\wedge DW,R(*)).
\end{equation}
Let us understand it via Theorem \ref{stablerepresentablemapsbordism}.
Set $Z=D((X,A)^{-\xi})$ and $W=(Y,B)^{-\zeta}$.
The domain of \eqref{dualitynattransnotquiteiso} consists of bordism classes of derived orbifolds $(C,\partial C)$ with representable maps $f:C\to X$ and $g:C\to Y$ such that $\partial C\subseteq f^{-1}(A)\cup g^{-1}(B)$ together with a stable isomorphism between $TC$ and $f^*\xi+g^*\zeta$.
The codomain consists of bordism classes of derived orbifolds $(C,\partial C)$ with a representable map $(C,\partial C)\to(X,A)\times(Y,B)$ and an isomorphism between $TC$ and the pullback of $\xi+\zeta$ (note that there is a unique up to homotopy representable map $C\to R(*)$, so we can simply ignore this piece of data).
The map from the domain to the codomain is the evident one: send $(f,g)$ to $f\times g$.
Of course, representability of $f\times g$ is a rather different (and weaker) condition from representability of both $f$ and $g$.
\end{example}

A notable omission in Theorem \ref{stablerepresentablemapsbordism} is an interpretation of the bordism group where both $f$ and $g$ are arbitrary (not necessarily representable).

\subsection{Acknowledgements}

We thank Andr\'e Henriques for discussions about orbispaces and Stefan Schwede for discussions about global homotopy theory.
The comments from both referees were also extremely helpful.

\section{Topology of orbispaces}

\subsection{Orbispaces as topological stacks}

We briefly recall some definitions and basic properties; for further background we refer the reader to \cite[\S 3]{orbibundle} and to \cite{haefligerorbiespaces,noohifoundations,behrendintroduction,metzler,behrendnoohi}.

We work in the $2$-category $\Shv(\Top,\Grpd)$, whose objects will simply be called `stacks'.
Morphisms between stacks do not form a set, rather a groupoid, which is the meaning of the prefix `$2$-'.

The Yoneda inclusion $\Top\hookrightarrow\Shv(\Top,\Grpd)$ is continuous, and we systematically identify objects of $\Top$ with their images in $\Shv(\Top,\Grpd)$.
Such stacks are called \emph{representable}.

A morphism of stacks $X\to Y$ is called representable iff for every topological space $Z$ and every map $Z\to Y$, the fiber product $X\times_YZ$ is representable.
For any property $\cP$ of morphisms of topological spaces which is preserved under pullback, a representable map of stacks $X\to Y$ is said to have $\cP$ iff $X\times_YZ\to Z$ has $\cP$ for every topological space $Z$ and every map $Z\to Y$.
Examples of such properties include being an open inclusion, a closed inclusion, \'etale, separated, proper (which by definition implies separated), and admitting local sections.

The inclusion $\Top\subseteq\Shv(\Top,\Grpd)$ admits a left adjoint $\left|\cdot\right|:\Shv(\Top,\Grpd)\to\Top$ known as passing to the \emph{coarse space} of a stack.
For a fixed stack $X$, open (resp.\ closed) inclusions $Y\hookrightarrow X$ are in bijection with open (resp.\ closed) subsets $\left|Y\right|\subseteq\left|X\right|$.

A stack $X$ is called \emph{topological} iff there exists a representable map admitting local sections $U\to X$ from a topological space $U$; such a map is called an \emph{atlas}.
A choice of atlas $U\to X$ gives rise to a topological groupoid $U\times_XU\righttwoarrows U$ presenting $X$.
Conversely, every topological groupoid $M\righttwoarrows O$ is uniquely of this form.
The coarse space of a topological stack $X$ is the quotient of any atlas $U$ by the image of $U\times_XU\to U\times U$ (which is an equivalence relation).
For a topological group $G$ acting continuously on a topological space $V$, the stack quotient $V/G$ is by definition the topological stack presented by the action groupoid $G\times V\righttwoarrows V$.

By a `point' $p$ of a stack $X$, we mean a map $p:*\to X$, i.e.\ an object of $X(*)$ (where $*$ denotes the one point space).
The automorphism group of this object of $X(*)$ is called the \emph{isotropy group} of $p$, denoted $G_p$.
Given a point $*\to X$, the fiber product $*\times_X*$ has trivial isotropy, and its points are in bijection with $G_p$.
If $X$ is a topological stack, then $*\times_X*$ is a topological space, which thus endows $G_p$ with the structure of a topological group.

A \emph{separated orbispace} is a stack $X$ which admits an \'etale atlas $U\to X$ and whose diagonal $X\to X\times X$ is proper.
Equivalently, $X$ is a separated orbispace iff $\left|X\right|$ is Hausdorff and there exists a cover of $X$ by open substacks of the form $Y/\Gamma$ where $\Gamma$ is a finite discrete group acting continuously on a Hausdorff topological space $Y$ \cite[Proposition 3.3]{orbibundle}.
In particular, a separated orbispace has an \'etale atlas $U\to X$ for which $U$ is Hausdorff.
Henceforth we will drop the prefix `separated' from `separated orbispace' and simply write `orbispace'.

The isotropy groups of an orbispace are all finite and discrete.
A map of orbispaces is representable iff it is injective on isotropy groups \cite[Corollary 3.6]{orbibundle}; in particular, an orbispace is a space iff its isotropy groups are all trivial.

The stack quotient $V/G$ is an orbispace provided $V$ is Hausdorff, $G$ is compact Hausdorff (these imply $V/G$ has proper diagonal), and there exists a map $W\to V$ such that the resulting map $G\times W\to V$ is \'etale (this implies that $W\to V/G$ is an \'etale atlas).
In particular, $V/G$ is an orbispace for $V$ Hausdorff and $G$ finite.

An orbispace is called paracompact iff its coarse space is paracompact.

For any finite group $G$ (we equip all finite groups with the discrete topology), the stack $\BB G:=*/G$ (the stack quotient of a point $*$ by the trivial action of $G$) is an orbispace.
The quotient map $*\to\BB G$ is the universal principal $G$-bundle: for any stack $Y$, the functor from maps $Y\to\BB G$ to principal $G$-bundles over $Y$ given by pulling back $*\to\BB G$ is an equivalence of groupoids.
The groupoid of maps $\BB G\to\BB H$ is (canonically equivalent to) the groupoid $\Hom(G,H)/H$ in which an object is a group homomorphism $\varphi:G\to H$ and in which an isomorphism $\varphi\xrightarrow\sim\varphi'$ is an element $h\in H$ satisfying $\varphi=h\varphi'h^{-1}$.
The full subcategory of stacks of the form $\BB G$ for some finite group $G$ is thus equivalent to the $2$-category $\FinGrp$ of finite groups, homomorphisms, and conjugations.
We will frequently restrict consideration to representable maps, in which case the category formed by $\BB G$ is denoted $\InjFinGrp$, which is the same as $\FinGrp$ except homomorphisms are required to be injective.

\begin{lemma}\label{coarseproduct}
For orbispaces $X$ and $Y$, the product $X\times Y$ is an orbispace and the natural map $\left|X\times Y\right|\to\left|X\right|\times\left|Y\right|$ is a homeomorphism.
\end{lemma}

\begin{proof}
For \'etale atlases $U_X\to X$ and $U_Y\to Y$, the product $U_X\times U_Y\to X\times Y$ is an \'etale atlas, and the diagonal of $X\times Y$ is the product of the diagonals of $X$ and $Y$, hence is proper.
Thus $X\times Y$ is an orbispace.

The assertion that $\left|X\times Y\right|\to\left|X\right|\times\left|Y\right|$ is a homeomorphism can be checked locally on $\left|X\right|$ and $\left|Y\right|$.
It thus suffices to show that for actions of finite groups $G$ and $H$ on Hausdorff spaces $X$ and $Y$, the natural map $\left|(X\times Y)/(G\times H)\right|\to\left|X/G\right|\times\left|Y/H\right|$ is a homeomorphism.
This map is obviously a bijection.
Open subsets of the domain correspond to $(G\times H)$-invariant open subsets of $X\times Y$.
Open subsets of the target are generated by products of $G$-invariant open subsets of $X$ with $H$-invariant open subsets of $Y$.
Open subsets of the latter form are certainly of the former form (which is the obvious direction that $\left|X\times Y\right|\to\left|X\right|\times\left|Y\right|$ is continuous).
Conversely, suppose $U\subseteq X\times Y$ is a $(G\times H)$-invariant open set and let $(x,y)\in U$, and let us show that there exists a product of a $G$-invariant open subset of $X$ and an $H$-invariant open subset of $Y$ which contains $(x,y)$ and is contained in $U$.
Since $U$ is open in the product topology, it contains a neighborhood $V\times W$ of $(x,y)$ where $V\subseteq X$ and $W\subseteq Y$ are open.
Now since $U$ is $(G\times H)$-invariant, it also contains $(G\cdot V)\times(H\cdot W)$, so we are done.
\end{proof}

\begin{lemma}
If $X$ is an orbispace and $U\to X$ is an \'etale atlas with $U$ Hausdorff, then $U\to X$ is separated.
\end{lemma}

\begin{proof}
The map $U\times_XU\to U\times U$ is separated since $X\to X\times X$ is separated, and the map $U\times U\to U$ is separated since $U$ is Hausdorff.
The composition $U\times_XU\to U$ is thus separated, hence $U\to X$ is separated.
\end{proof}

\begin{lemma}\label{orbiiso}
A map of orbispaces is an isomorphism iff it induces isomorphisms on isotropy groups and induces a homeomorphism on coarse spaces.
\end{lemma}

\begin{proof}
Let $f:X\to Y$ be a map of orbispaces which induces isomorphisms on isotropy groups $G_x\xrightarrow\sim G_{f(x)}$ and induces a homeomorphism on coarse spaces $\left|f\right|:\left|X\right|\to\left|Y\right|$, and let us show that $f$ is an isomorphism.
The property of $f$ being an isomorphism is local on $\left|Y\right|$, so we may assume without loss of generality that $Y=Y'/G$ for some finite group $G$ acting continuously on a Hausdorff space $Y'$.
Since $f$ is representable, $X':=Y'\times_YX$ is a space, and $X=X'/G$.
We thus have a $G$-equivariant map $X'\to Y'$ which induces a homeomorphism $\left|X'/G\right|\xrightarrow\sim\left|Y'/G\right|$ and which induces isomorphisms on stabilizer groups.
This implies that $X'\to Y'$ is a bijection.
It suffices to show that the map $f':X'\to Y'$ is open (and hence is a homeomorphism).
What we know is that $f'$ sends $G$-invariant open subsets to open subsets.
Let $x\in X'$.
Since $Y'$ is Hausdorff, there exist open neighborhoods $U_g\subseteq Y'$ of $g\cdot f'(x)$ for all $g\in G$ such that $g\cdot U_h=U_{gh}$ and $U_g\cap U_h=\varnothing$ for $g\cdot f'(x)\ne h\cdot f'(x)$ and $U_g=U_h$ for $g\cdot f'(x)=h\cdot f'(x)$.
Now let $V\subseteq(f')^{-1}(U_1)$ be any open neighborhood of $x$.
Its image $f'(V)\subseteq Y$ is the intersection of two open sets $f'(G\cdot V)\cap U_1$, so $f'(V)$ is open.
Thus $f'$ is open.
\end{proof}

A \emph{topological orbifold} is an orbispace $X$ which is \'etale locally homeomorphic to $\RR^n$, in the sense that for some (equivalently, every) \'etale atlas $U\to X$, the space $U$ is locally homeomorphic to $\RR^n$ (it may also be required paracompact if one so desires).
In other words, $X$ is locally isomorphic to $U/\Gamma$ for $U\subseteq\RR^n$ open and $\Gamma\acts U$ acting continuously.
A topological orbifold is called \emph{locally tame} iff we may take such actions $\Gamma\acts U$ to be restrictions of linear actions on $\RR^n$.
A smooth structure on a topological orbifold $X$ is a choice of atlas $U\to X$ together with a smooth structure on $U$ such that the two smooth structures on $U\times_XU$ obtained via pullback from the smooth structure on $U$ coincide (smooth structures relative to $U\to X$ and $U'\to X$ are equivalent iff they give rise to the same pullback smooth structure on $U\times_XU'$).
Smooth orbifolds are locally isomorphic to $U/\Gamma$ for $U\subseteq\RR^n$ open and $\Gamma\acts U$ acting smoothly (equivalently, linearly).

\subsection{Vector bundles and principal bundles over orbispaces}

A (real) vector bundle over a stack $X$ is a representable map $V\to X$ along with maps $\RR\times V\to V$ and $V\times_XV\to V$ over $X$, such that the pullback to any topological space $Z\to X$ is a vector bundle over $Z$ with its fiberwise scaling and addition maps.
Similarly, for a Lie group $G$, a principal $G$-bundle over a stack $X$ is a representable map $P\to X$ along with a map $G\times P\to P$ over $X$, such that the pullback to any topological space $Z\to X$ is a principal $G$-bundle with its $G$-action.
We only ever consider finite-dimensional vector bundles, so we will usually omit the adjective `finite-dimensional' for sake of brevity.
We will also only ever consider positive definite inner products, so we will also usually omit the adjective `positive definite'.

For a vector bundle $V\to X$ and a point $p:*\to X$, the fiber $V_p:=V\times_X*$ carries a linear action of the isotropy group $G_p$.
Similarly, given a principal $G$-bundle $P\to X$ and a point $p:*\to X$, the fiber $P_p$ carries a $G_p$-action compatible with the $G$-action (so, fixing an identification of $G$-spaces $P_p=G$, this becomes a homomorphism $G_p\to G$).

\begin{lemma}\label{mappingcts}
For any topological space $X$, the tautological bijection between (setwise) maps $X\times\RR^n\to\RR^m$ (resp.\ $X\times G\to G$) which for every fixed $x\in X$ are linear (resp.\ $G$-equivariant) and maps $X\to\Hom(\RR^n,\RR^m)$ (resp.\ $X\to G$) restricts to a bijection between the subsets of continuous maps.
\end{lemma}

\begin{proof}
For one direction, note that the map $\RR^n\times\Hom(\RR^n,\RR^m)\to\RR^m$ (resp.\ $G\times G\to G$) is continuous, so its pullback along a continuous map $X\to\Hom(\RR^n,\RR^m)$ (resp.\ $X\to G$) remains continuous.
For the other direction, note that the `matrix entries' of a map $X\to\Hom(\RR^n,\RR^m)$ can be recovered from the map $X\times\RR^n\to X\times\RR^m$ by appropriate pre- and post-composition with maps $*\to\RR^n$ and $\RR^m\to\RR$, and similarly for Lie groups $G$.
\end{proof}

It follows immediately from Lemma \ref{mappingcts} that for any topological space $X$, the functor from the groupoid of maps $X\to\bigsqcup_{n\geq 0}*/\GL_n(\RR)$ to vector bundles over $X$ defined by pulling back the vector bundle $\bigsqcup_{n\geq 0}\RR^n/\GL_n(\RR)\to\bigsqcup_{n\geq 0}*/\GL_n(\RR)$ is an equivalence, similarly for $X\to */G$ and principal $G$-bundles, and similarly for $X\to\bigsqcup_{n\geq 0}*/O(n)$ and vector bundles with inner product.
These statements automatically extend to arbitrary stacks $X$: a vector bundle $V\to X$ is the same as the specification, compatible with pullback, of a vector bundle $V_Z\to Z$ for every map $Z\to X$ from a topological space $Z$, which is, by the above result for topological spaces, the same as the specification, compatible with pullback, of a map $Z\to */\GL_n(\RR)$ for every map $Z\to X$ from a topological space $Z$, which is the same as a map of stacks $X\to\bigsqcup_{n\geq 0}*/\GL_n(\RR)$ (and similarly for $*/G$ and $*/O(n)$).

There is a standard deformation retraction from $\Inj(\RR^n,\RR^m)$ to the subspace of isometric injections given by $f\mapsto f(f^*f)^{-t/2}$ for $t\in[0,1]$.
Since this deformation retraction is $O(n)\times O(m)$-equivariant, by Lemma \ref{mappingcts} it induces, for any injective map of vector bundles with inner products over a topological space $X$, a canonical homotopy through injections to an isometric injection (moreover, the same holds for over arbitrary stacks $X$, by the reasoning as in the previous paragraph).

We now move on to some foundational results which are specific to orbispaces.

\begin{lemma}
Every principal $G$-bundle over an orbispace $X$ is locally of the form $(G\times Y)/\Gamma\to Y/\Gamma$ for $\Gamma\acts Y$ and $\Gamma\to G$.
\end{lemma}

\begin{proof}
The case of $G=\GL_n(\RR)$ (i.e.\ vector bundles) was proven in \cite[Lemma 6.7]{orbibundle}.
The essential point in generalizing the proof given there to general $G$ is to note that there is a $G$-conjugation, $G$-translation, and $S_n$ invariant `averaging' operation giving a retraction onto the diagonal $G\subseteq G^n$ defined in its neighborhood.
\end{proof}

Some important properties of vector bundles and principal bundles require a paracompactness assumption.

\begin{lemma}[{\cite[Lemma 5.1]{orbibundle}}]\label{paracompactmetric}
Every vector bundle over a paracompact orbispace has a inner product.
\qed
\end{lemma}

\begin{lemma}\label{intervalpullback}
For a paracompact orbispace $X$, every principal $G$-bundle over $X\times[0,1]$ is pulled back from $X$.
\end{lemma}

\begin{proof}
The case of $G=\GL_n(\RR)$ (i.e.\ vector bundles) was proven in \cite[Lemma 6.2]{orbibundle}.
The same averaging operation as before allows this proof to apply to general $G$.
\end{proof}

\subsection{Gluing orbispaces}

We now explain how some basic topological gluing constructions are generalized to the orbispace context.
These constructions provide the foundation for doing algebraic topology with orbispaces.

We begin with a discussion of how to glue together stacks along open substacks.
The first step is to observe the following `descent for morphisms' property:

\begin{lemma}\label{mordescent}
Let $X=\bigcup_\alpha U_\alpha$ be a cover by open substacks.
The functor
\begin{equation}\label{descentequivalence}
\Hom(X,Y)\xrightarrow\sim\left\{\begin{matrix}f_\alpha\in\Hom(U_\alpha,Y)\hfill\\ g_{\alpha\beta}:f_\alpha|_{U_\alpha\cap U_\beta}\xrightarrow\sim f_\beta|_{U_\alpha\cap U_\beta}\hfill\end{matrix}\,\middle|\,g_{\alpha\beta}g_{\beta\gamma}=g_{\alpha\gamma}\textrm{ over }U_\alpha\cap U_\beta\cap U_\gamma\right\}
\end{equation}
is an equivalence for any stack $Y$ (on the right side, an isomorphism $(f_\alpha,g_{\alpha\beta})\to(f_\alpha',g_{\alpha\beta}')$ consists of $\pi_\alpha:f_\alpha\xrightarrow\sim f_\alpha'$ such that $g_{\alpha\beta}'\pi_\beta=\pi_\alpha g_{\alpha\beta}$ over $U_\alpha\cap U_\beta$).
\end{lemma}

\begin{proof}
We may construct an inverse to \eqref{descentequivalence} as follows.
Given an element of the right hand side, we may associate to any map $Z\to X$ (where $Z$ is a topological space) a map $Z\to Y$ as follows.
The map $Z\to X$ induces an open cover $Z=\bigcup_\alpha Z\times_XU_\alpha$.
Each map $Z\times_XU_\alpha\to U_\alpha$ may be composed with our chosen element on the right side of \eqref{descentequivalence} to a map $Z\times_XU_\alpha\to Y$.
The compatibility data on the right side of \eqref{descentequivalence} provides descent data to glue these maps together (using the stack property for $Y$) to a map $Z\to Y$.
We have thus associated to each map $Z\to X$ a map $Z\to Y$.
This construction is compatible with pullback, hence defines a map of stacks $X\to Y$.
Tracing through definitions, it can be checked that this map is a two-sided inverse to \eqref{descentequivalence}.
\end{proof}

Lemma \ref{mordescent} may be reformulated as saying that $X$ is the colimit of the diagram consisting of the open substacks $U_\alpha$, their pairwise intersections $U_\alpha\cap U_\beta$, and their triple intersections $U_\alpha\cap U_\beta\cap U_\gamma$ (and no higher intersections).

Going in the opposite direction, let us argue that pushouts of open inclusions of stacks always exist.
Namely, consider a pair of open inclusions $X\hookleftarrow U\hookrightarrow Y$.
Given such data, we may define a stack $X\cup_UY$ by the following natural mapping in property: a map $Z\to X\cup_UY$ ($Z$ a topological space) consists of an open cover $Z=Z_X\cup Z_Y$, maps $Z_X\to X$ and $Z_Y\to Y$ such that in both cases the inverse image of $U$ is $Z_X\cap Z_Y$, together with an isomorphism between the two resulting maps $Z_X\cap Z_Y\to U$.
It is immediate to check that the maps $X\to X\cup_UY\leftarrow Y$ are both open inclusions intersecting along $U$, so Lemma \ref{mordescent} implies that
\begin{equation}\label{opengluingpushout}
\begin{tikzcd}
U\ar[r]\ar[d]&X\ar[d]\\
Y\ar[r]&X\cup_UY
\end{tikzcd}
\end{equation}
is a pushout.
Since $X\hookrightarrow X\cup_UY\hookleftarrow Y$ are open inclusions, it follows that if $X$ and $Y$ both admit \'etale atlases, then so do $U$ and $X\cup_UY$.
Also, if $X$ and $Y$ are locally of the form $V/\Gamma$ for a finite group $\Gamma$ acting on a Hausdorff space $V$, then the same holds for $X\cup_UY$.
Thus if $X$ and $Y$ are orbispaces, to verify that $X\cup_UY$ is an orbispace, it suffices to show that $\left|X\cup_UY\right|$ is Hausdorff.
Since the coarse space functor $\left|\cdot\right|$ is a left adjoint, it preserves all colimits, so $\left|X\cup_UY\right|=\left|X\right|\cup_{\left|U\right|}\left|Y\right|$.
This gives an effective procedure to glue together a pair of orbispaces along a common open subspace and to show that the result is again an orbispace.

The next construction we wish to discuss is the formation of mapping cylinders for representable maps of orbispaces.
For a map of topological spaces $A\to X$, the mapping cylinder $\cyl(A\to X)$ is defined as the pushout
\begin{equation}\label{cylinderdiagram}
\begin{tikzcd}
A\ar[r]\ar[d,"\times\{0\}"]&X\ar[d]\\
A\times[0,1]\ar[r]&\cyl(A\to X).
\end{tikzcd}
\end{equation}
A basic property of colimits in the category of topological spaces is that the property of a diagram being a colimit diagram is local on the colimit object.
It follows that colimit diagrams are preserved under \'etale pullback: if $U\to\colim p$ is \'etale, then the natural map $\colim(p\times_{\colim p}U)\to U$ is an isomorphism.
Thus the formation of mapping cylinders commutes with \'etale pullback: if $U\to X$ is \'etale, then the natural map $\cyl(A\times_XU\to U)\to\cyl(A\to X)\times_XU$ is an isomorphism.
This fact allows us to define the mapping cylinder of any representable map of stacks $A\to X$ for which $X$ (hence also $A$) admits an \'etale atlas.
Indeed, let $A\to X$ be such a map.
Choose an \'etale atlas $U\to X$, which pulls back to an \'etale atlas $A\times_XU\to A$.
We thus obtain a topological groupoid $U\times_XU\righttwoarrows U$ presenting the stack $X$, and we obtain a topological groupoid $U\times_XA\times_XU\righttwoarrows A\times_XU$ presenting $A$.
The map $A\to X$ induces a map of topological groupoids from the latter to the former, which presents the map $A\to X$.
We may consider the `cylinder' of this map of groupoids, namely
\begin{equation}\label{cylindergroupoid}
\cyl(U\times_XA\times_XU\to U\times_XU)\righttwoarrows\cyl(A\times_XU\to U),
\end{equation}
and we define $\cyl(A\to X)$ to be the topological stack presented by this groupoid.
Note that $\cyl(U\times_XA\times_XU\to U\times_XU)=\cyl(A\times_XU\to U)\times_XU$ since $U\to X$ is \'etale.

\begin{lemma}
$\cyl(A\to X)$ is independent, up to canonical equivalence, of the choice of \'etale atlas $U\to X$.
\end{lemma}

\begin{proof}
It suffices to show that for any two atlases $U\to X\leftarrow U'$, the inclusions of the groupoids \eqref{cylindergroupoid} for $U$ and $U'$ into the groupoid for $U\sqcup U'$ induce equivalences of stacks.
To show this, it in turn suffices to show that the map
\begin{equation}
\cyl(U'\times_XA\times_XU\to U'\times_XU)\to\cyl(A\times_XU\to U)
\end{equation}
admits local sections.
This in turn is implied by the assertion that the natural map
\begin{equation}
\cyl(U'\times_XA\times_XU\to U'\times_XU)\to U'\times_X\cyl(A\times_XU\to U)
\end{equation}
is an isomorphism, which holds since formation of mapping cylinders of topological spaces commutes with \'etale pullback (both sides are $(U'\times_XU)\times_U\cyl(A\times_XU\to U)$).
\end{proof}

Formation of mapping cylinders commutes with passing to the coarse space: for any representable map of stacks $A\to X$ admitting \'etale atlases, the natural map $\cyl(\left|A\right|\to\left|X\right|)\to\left|\cyl(A\to X)\right|$ is an isomorphism (this can be checked by inspection, using the fact that for any topological stack $X$ with atlas $U\to X$, the map $U\to\left|X\right|$ is the topological quotient by the image of $U\times_XU\to U\times U$, and using the fact that $\left|A\times[0,1]\right|\to\left|A\right|\times[0,1]$ is an isomorphism for any topological stack $A$ \cite[Lemma 6.15]{orbibundle}).

It now follows that if $A$ and $X$ are both orbispaces, then so is $\cyl(A\to X)$.
Indeed, if $Y/\Gamma\hookrightarrow X$ is an open inclusion for $Y$ Hausdorff and $\Gamma$ finite, we obtain an open inclusion $(A\times_XY)/\Gamma=A\times_X(Y/\Gamma)\hookrightarrow A$.
Since $A$ is an orbispace, its diagonal is proper, so the action map $\Gamma\times(A\times_XY)\to(A\times_XY)\times(A\times_XY)$ is proper, hence its precomposition with $A\times_XY\xrightarrow{1\times}\Gamma\times(A\times_XY)$ is proper; this being the diagonal of $A\times_XY$, we conclude that $A\times_XY$ is Hausdorff.
We thus have an open inclusion $\cyl(A\times_XY\to Y)/\Gamma\hookrightarrow\cyl(A\to X)$ where $\cyl(A\times_XY\to Y)$ is Hausdorff.
The coarse space $\left|\cyl(A\to X)\right|=\cyl(\left|A\right|\to\left|X\right|)$ is Hausdorff since $\left|X\right|$ and $\left|A\right|$ are.

Our next task is to show that the mapping cylinder diagram \eqref{cylinderdiagram} is a pushout.
This gives another proof of the fact that formation of mapping cylinders commutes with passing to the coarse space.
We begin with an example to show that mapping cylinder diagrams, even of topological spaces, need not be pushouts in the category of \emph{all} stacks.
Our task is thus, more precisely, to identify a particular full subcategory of stacks in which mapping cylinder diagrams are pushouts (see Proposition \ref{mcpushout} below).

\begin{example}
Consider the pushout diagram
\begin{equation}\label{pushoutintopnotinstacks}
\begin{tikzcd}
\{1\}\ar[d]\ar[r]&{}[1,2]\ar[d]\\
{}[0,1]\ar[r]&{}[0,2]
\end{tikzcd}
\end{equation}
in the category of topological spaces.
Let $X$ denote the stack defined by the property that a map $Z\to X$ from a topological space $Z$ is a continuous map $f:Z\to[0,2]$ such that there exists an open cover $Z=U\cup V$ with $f(U)\subseteq[0,1]$ and $f(V)\subseteq[1,2]$ (note that $X$ is indeed a stack!).
Now there is a tautological diagram
\begin{equation}
\begin{tikzcd}
\{1\}\ar[d]\ar[r]&{}[1,2]\ar[d]\\
{}[0,1]\ar[r]&X
\end{tikzcd}
\end{equation}
which is not induced by a map $[0,2]\to X$ (there is no open covering $[0,2]=U\cup V$ with $U\subseteq[0,1]$ and $V\subseteq[1,2]$).
It follows that the diagram \eqref{pushoutintopnotinstacks} is not a pushout in the category of all stacks.
\end{example}

\begin{lemma}\label{pushoutff}
Let $\{X_\alpha\}_{\alpha\in A}$ be any diagram of topological spaces and $X:=\colim_{\alpha\in A}X_\alpha$ its colimit.
For any topological stack $T$, the map
\begin{equation}
\Hom(X,T)\to\lim_{\alpha\in A}\Hom(X_\alpha,T)
\end{equation}
is fully faithful.
\end{lemma}

\begin{proof}
We just need to recall the description of $\Hom(X,T)$ for $X$ a topological space and $T$ the stack associated to a topological groupoid $M\righttwoarrows O$.
An object of $\Hom(X,T)$ is an open cover $X=\bigcup_iU_i$ together with a collection of maps $\alpha_i:U_i\to O$ and $\beta_{ij}:U_i\cap U_j\to M$ projecting to $\alpha_i\times\alpha_j$ and satisfying $\beta_{ij}\beta_{jk}=\beta_{ik}$ over $U_i\cap U_j\cap U_k$.
An isomorphism between $(U_i,\alpha_i,\beta_{ij})$ and $(U_{i'}',\alpha_{i'}',\beta_{i'j'}')$ is a collection of maps $\gamma_{ii'}:U_i\cap U_{i'}'\to M$ projecting to $\alpha_i\times\alpha_{i'}'$ and satisfying $\beta_{ij}\gamma_{jj'}=\gamma_{ij'}$ over $U_i\cap U_j\cap U_{j'}$ and $\gamma_{ii'}\beta_{i'j'}=\gamma_{ij'}$ over $U_i\cap U_{i'}\cap U_{j'}$.
Composition of isomorphisms relies on the fact that $\Hom(-,M)$ is a sheaf.

Now fix two objects $(U_i,\alpha_i,\beta_{ij})$ and $(U_{i'}',\alpha_{i'}',\beta_{i'j'}')$ of $\Hom(X,T)$.
The set of isomorphisms between them is the set of collections of maps $\gamma_{ii'}:U_i\cap U_{i'}'\to M$ satisfying certain compatibility properties.
Now we note that for any open subset $U\subseteq X$, the map $\colim_{\alpha\in A}U_\alpha\xrightarrow\sim U$ is an isomorphism where $U_\alpha$ denotes the inverse image of $U$ inside $X_\alpha$.
Thus since $M$ is a topological space, the data of maps $U_i\cap U_{i'}'\to M$ is equivalent to giving a compatible collection of such maps over the inverse images of $U_i\cap U_{i'}'$ in each $X_\alpha$.
Such data is precisely the data of an isomorphism in $\lim_{\alpha\in A}\Hom(X_\alpha,T)$ between the images of $(U_i,\alpha_i,\beta_{ij})$ and $(U_{i'}',\alpha_{i'}',\beta_{i'j'}')$.
\end{proof}

\begin{lemma}\label{descentgroupoid}
For any topological stack $X$ with atlas $U\to X$, the functor
\begin{equation}\label{grpddescentequivalence}
\Hom(X,T)\xrightarrow\sim\Eq\bigl(\Hom(U,T)\righttwoarrows\Hom(U\times_XU,T)\rightthreearrows\Hom(U\times_XU\times_XU,T)\bigr)
\end{equation}
is an equivalence for any stack $T$ (concretely, an object on the right is a map $f:U\to T$ and an isomorphism $i:fp_1\xrightarrow\sim fp_2$ in $\Hom(U\times_XU,T)$ such that the composition of $i\circ p_{12}$ and $i\circ p_{23}$ agrees with $i\circ p_{13}$ in $\Hom(U\times_XU\times_XU,T)$, and an isomorphism $(f,i)\to(f',i')$ is an isomorphism $j:f\xrightarrow\sim f'$ such that $i'\circ jp_1=jp_2\circ i$).
\end{lemma}

\begin{proof}
This is similar to the proof of Lemma \ref{mordescent}.
Given a map $Z\to X$ from a topological space $Z$ and an element of the right side of \eqref{grpddescentequivalence}, we may define a map $Z\to T$ as follows.
Our map $Z\to X$ may be regarded as an open cover of $Z$, maps from the elements of the open cover to $U$, and maps from pairwise intersections to $U\times_XU$, satisfying a cocycle condition.
The element of the right side of \eqref{grpddescentequivalence} turns this into maps from the elements of the open cover to $T$ and isomorphisms between them on their pairwise overlaps, satisfying a cocycle condition.
The stack property for $T$ means that this data defines a map $Z\to T$.
One now checks that this is a two-sided inverse to \eqref{grpddescentequivalence}.
\end{proof}

\begin{proposition}\label{mcpushout}
For any representable map of stacks $A\to X$ admitting separated \'etale atlases, the mapping cylinder diagram \eqref{cylinderdiagram} is a pushout in the $2$-category of stacks which admit a separated \'etale atlas.
\end{proposition}

\begin{proof}
We are supposed to show that for any stack $T$ which admits a separated \'etale atlas, the map
\begin{equation}\label{pushmainequiv}
\Hom(\cyl(A\to X),T)\to\Hom(X,T)\times_{\Hom(A,T)}\Hom(A\times[0,1],T)
\end{equation}
is an equivalence of groupoids.

We begin with the case that $X$ and $A$ are topological spaces.
In this case, Lemma \ref{pushoutff} says that \eqref{pushmainequiv} is fully faithful, so it remains to prove essential surjectivity.
Thus suppose we have maps $X\to T$ and $A\times[0,1]\to T$ and an isomorphism between the respective induced maps $A\to T$.
We should glue these together into a map $\cyl(A\to X)\to T$.
Fix a separated \'etale atlas $O\to T$, hence a groupoid presentation $M\righttwoarrows O$ of $T$ with $M=O\times_TO$.
The map $X\to T$ thus may be regarded as an open cover $X=\bigcup_iU_i$, maps $U_i\to O$, and maps $U_i\cap U_j\to M$, which we may pull back under $f:A\to X$ to obtain the map $A\to T$.
This map is isomorphic to the restriction to $A=A\times\{0\}$ of the given map $A\times[0,1]\to T$, which is \emph{a priori} defined by a different open cover.
Now the key point is the following.
Consider one of the open sets $f^{-1}(U_i)\subseteq A$, which is equipped with a map $f^{-1}(U_i)\to O$.
This map may be regarded as a section over $f^{-1}(U_i)\subseteq A=A\times\{0\}$ of the separated \'etale map $O\times_T(A\times[0,1])\to A\times[0,1]$.
Since this map is \'etale, each point $p\in f^{-1}(U_i)$ has a neighborhood $V_p\times[0,\varepsilon_p)$ over which the section extends.
Since this map is separated and $[0,\varepsilon_p)$ is connected, these extensions are unique.
They hence glue together to give an open set $V_i\subseteq A\times[0,1]$ intersecting $A\times\{0\}$ in $f^{-1}(U_i)$ such that the map $f^{-1}(U_i)\to O$ admits a unique extension to $V_i$ together with an isomorphism of the resulting compostion to $T$ with the given map $A\times[0,1]\to T$.
Now $V_i$ and $U_i$ define together an open set $W_i\subseteq\cyl(A\to X)$, and we have defined thus a map $W_i\to O$.
These $W_i$, together with $A\times(0,1]$, cover $\cyl(A\to X)$, so this data defines for us a map $\cyl(A\to X)\to T$ lifting our given data on the right side of \eqref{pushmainequiv}.

Having treated the case that $X$ and $A$ are topological spaces, we deduce the general case using Lemma \ref{descentgroupoid}.
Fix an \'etale atlas $U\to X$, so that $\cyl(X\to A)$ is presented by the topological groupoid
\begin{equation}
\cyl(A\times_XU\to U)\times_XU\righttwoarrows\cyl(A\times_XU\to U).
\end{equation}
By Lemma \ref{descentgroupoid}, we conclude that $\cyl(X\to A)$ coincides with the coequalizer
\begin{equation}
\Coeq\bigl(\cyl(A\times_XU\to U)\lefttwoarrows\cyl(A\times_XU\to U)\times_XU\leftthreearrows\cyl(A\times_XU\to U)\times_XU\times_XU\bigr).
\end{equation}
Each term in the coequalizer is a cylinder (since $\times_XU$ is an \'etale pullback so can be brought inside $\cyl$) of a map of topological spaces.
Hence each of these terms is a pushout (in the $2$-category of stacks which admit a separated \'etale atlas).
Since coequalizers commute with pushouts, we conclude that $\cyl(X\to A)$ is the pushout of
\begin{equation}
\begin{tikzcd}
\Coeq(U\lefttwoarrows U\times_XU\leftthreearrows U\times_XU\times_XU)\\
\Coeq(A\times_XU\lefttwoarrows A\times_XU\times_XU\leftthreearrows A\times_XU\times_XU\times_XU)\ar[u]\ar[d]\\
\Coeq(A\times_XU\times[0,1]\lefttwoarrows A\times_XU\times_XU\times[0,1]\leftthreearrows A\times_XU\times_XU\times_XU\times[0,1]).
\end{tikzcd}
\end{equation}
The top two coequalizers are simply $X$ and $A$ by Lemma \ref{descentgroupoid}.
The bottom coequalizer is $A\times[0,1]$, not by pulling out the $\times[0,1]$ on general categorical principles, but rather by applying Lemma \ref{descentgroupoid} to the atlas $U\times[0,1]\to A\times[0,1]$.
\end{proof}

We now combine the results obtained thus far into a general gluing operation:

\begin{proposition}\label{collarpushout}
Let $B\to C$ be a representable map of orbispaces, and let $B\hookrightarrow A$ be a closed inclusion which is \emph{collared} in the sense that it factors as $B\xrightarrow{\times\{0\}}B\times[0,1)\hookrightarrow A$ where the second map is an open inclusion.
The pushout
\begin{equation}\label{generalgluing}
\begin{tikzcd}
B\ar[r]\ar[d,hook]&C\ar[d]\\
A\ar[r]&A\cup_BC
\end{tikzcd}
\end{equation}
exists in the category of stacks admitting a separated \'etale atlas, and $A\cup_BC$ is an orbispace.
\end{proposition}

\begin{proof}
Define $A\cup_BC$ by gluing $A\setminus B$ (an open substack of $A$) to $\cyl(B\to C)$ along $B\times(0,1)$ using a choice of collar $B\times[0,2)\hookrightarrow A$.
The fact that $\cyl(B\to C)$ is the pushout of $B\times[0,1]\leftarrow B\to C$ (Proposition \ref{mcpushout}) and $A\cup_BC$ is the pushout of $\cyl(B\to C)\leftarrow B\times(0,1)\hookrightarrow A$ (Lemma \ref{mordescent}) implies that \eqref{generalgluing} is a pushout.
In particular, the gluing $A\cup_BC$ does not depend on the choice of collar used to construct it.
\end{proof}

We are also interested in countable iterations of such attachment operations.
Let us call a map of orbispaces $X\to Y$ a \emph{mapping cylinder inclusion} iff it is a closed inclusion and admits a factorization of the form $X\hookrightarrow(X\cup_A(A\times\RR_{\geq 0}))\hookrightarrow Y$ where the second map is an open inclusion and the first map is the natural inclusion of $X$ into the open substack $X\cup_A(A\times[0,1))\subseteq X\cup_A(A\times[0,1])=\cyl(A\to X)$ for some representable map of orbispaces $A\to X$.
Equivalently, $X\to Y$ is a mapping cylinder inclusion iff it is the right vertical map in some pushout diagram \eqref{generalgluing} (without specifying a choice of such presentation).

\begin{proposition}\label{mcinfinite}
Let $X_0\to X_1\to\cdots$ be a sequence of mapping cylinder inclusions of orbispaces.
The colimit $\colim_iX_i$ exists in the category of stacks admitting a separated \'etale atlas, and this colimit is an orbispace.
\end{proposition}

\begin{proof}
We begin with the case that all $X_i$ are Hausdorff topological spaces, where we show that the colimit in the category of topological spaces $X:=\colim_iX_i$ is the desired colimit.
Let us first note that $X$ is itself Hausdorff.
Indeed, let $p,q$ be distinct points of $X$, and choose $i$ large so that they both lie in $X_i$.
Since $X_i$ is Hausdorff, choose disjoint open subsets $U_p^i$ and $U_q^i$ of it containing $p$ and $q$, respectively.
A factorization of $X_i\to X_{i+1}$ witnessing that it is a mapping cylinder inclusion gives disjoint open subsets $U_p^{i+1}$ and $U_q^{i+1}$ of $X_{i+1}$ whose intersections with $X_i$ are $U_p^i$ and $U_q^i$, respectively.
Iterating in this way, we produce disjoint open subsets $U_p$ and $U_q$ of $X$ containing $p$ and $q$, respectively.

Now let us show that for any stack $T$ admitting a separated \'etale atlas, the map $\Hom(X,T)\to\lim_i\Hom(X_i,T)$ is an equivalence (still in the case $X_i$ are topological spaces and $X=\colim_iX_i$ is the colimit in the category of topological spaces).
It is fully faithful by Lemma \ref{pushoutff}, so it remains to show essential surjectivity.
Choose a separated \'etale atlas $U\to T$.
Fix an object of $\lim_i\Hom(X_i,T)$; this consists, in particular, of open covers $X_i=\bigcup_jU_{ij}$ and maps $U_{ij}\to U$, with certain descent data on intersections.
Now the key step is to note that, as was proven already during the proof of Proposition \ref{mcpushout}, the open sets $U_{ij}$ covering $X_i$ extend to $X_{i+1}$ along with their maps to $U$.
We may thus construct new open covers of the $X_i$ by induction as follows: the new open cover of $X_0$ is simply the one we are given to start with, and the new open cover of $X_i$ is obtained by taking the new open cover of $X_{i-1}$, extending it to a neighborhood of $X_{i-1}$ inside $X_i$ as in the proof of Proposition \ref{mcpushout}, and then adding $X_i\setminus X_{i-1}$ (open since $X_{i-1}\subseteq X_i$ is closed) intersected with all the open sets in the given open cover of $X_i$.
We thus obtain an open cover of $X$ and continuous maps from the elements of this open cover to $U$, along with the relevant descent data to define a map $X\to T$.
This completes the proof in the case that the $X_i$ are topological spaces.

We now move on to the general case.
Note that if $X\to Y$ is a mapping cylinder inclusion, then so is $\left|X\right|\to\left|Y\right|$, since passing to the coarse space preserves open inclusions and mapping cylinders.
Hence $\left|X_i\right|\to\left|X_{i+1}\right|$ is a mapping cylinder inclusion.
Thus the colimit $\colim_i\left|X_i\right|$ (which must be the coarse space of $\colim_iX_i$ if it exists) is Hausdorff as above.
Now every $X_i$ maps to $\colim_i\left|X_i\right|$, and since the latter is Hausdorff, it suffices to prove the statement after restricting to an open cover of $\colim_i\left|X_i\right|$.
Thus fix a point $p\in\colim_i\left|X_i\right|$ and let us prove the statement in a neighborhood of $p$.
We have $p\in\left|X_i\right|$ for some $i$, and let us construct an open neighborhood of $p$ as in the paragraph above, i.e.\ we begin with an open neighborhood $U^i\subseteq\left|X_i\right|$ of $p$, we consider $U^{i+1}\subseteq\left|X_{i+1}\right|$ the inverse image of $U_i$ in the mapping cylinder $X_i\cup_A(A\times\RR_{\geq 0})$, and iterating gives the desired open neighborhood in the colimit.
The effect of restricting to such an open subset is that we have reduced ourselves to the situation of a chain of closed inclusions $X_0\hookrightarrow X_1\hookrightarrow\cdots$ where $X_{i+1}=X_i\cup_{A_i}(A_i\times\RR_{\geq 0})$.

We may now treat this special case as follows.
By restricting further to an open subset of $X_0$ (and its inverse image in every $X_i$), we may assume without loss of generality that $X_0=Y_0/G$ for some Hausdorff space $Y_0$ acted on by a finite group $G$.
Pulling back under each projection, we obtain a sequence of inclusions of Hausdorff topological spaces $Y_0\hookrightarrow Y_1\hookrightarrow\cdots$ each with an action of $G$, where $Y_{i+1}=Y_i\cup_{B_i}(B_i\times\RR_{\geq 0})$, $G$-equivariantly, and $X_i=Y_i/G$.
Now it suffices to show that $\colim_i(Y_i/G)=(\colim_iY_i)/G$ (note that $\colim_iY_i$ is Hausdorff as shown above).
Express each $Y_i/G$ via the topological groupoid $G\times Y_i\righttwoarrows Y_i$, and appeal to Lemma \ref{descentgroupoid} to see that for any stack $T$, we have
\begin{multline}
\lim_i\Hom(X_i,T)\xrightarrow\sim\lim_i\Eq\bigl(\Hom(Y_i,T)\righttwoarrows\Hom(G\times Y_i,T)\rightthreearrows\Hom(G\times G\times Y_i,T)\bigr)\\
=\Eq\bigl(\lim_i\Hom(Y_i,T)\righttwoarrows\lim_i\Hom(G\times Y_i,T)\rightthreearrows\lim_i\Hom(G\times G\times Y_i,T)\bigr).
\end{multline}
Now for $T$ admitting a separated \'etale atlas, we have $\lim_i\Hom(Y_i,T)=\Hom(\colim_iY_i,T)$ since $Y_i\hookrightarrow Y_{i+1}$ are mapping cylinder inclusions, and the same holds for $G\times Y_i$ and $G\times G\times Y_i$ for the same reason.
We also have $\colim_i(G\times Y_i)=G\times\colim_iY_i$ since $G$ is finite (the functor $\times G$ on topological spaces is cocontinuous whenever $G$ is locally compact since it then has a right adjoint $\Maps(G,-)$).
We therefore have, for $T$ admitting a separated \'etale atlas,
\begin{multline}
\lim_i\Hom(X_i,T)=\\
\Eq\bigl(\Hom(\colim_iY_i,T)\righttwoarrows\Hom(G\times\colim_iY_i,T)\rightthreearrows\Hom(G\times G\times\colim_iY_i,T)\bigr).
\end{multline}
Applying Lemma \ref{descentgroupoid} once more, we see that the right side is $\Hom((\colim_iY_i)/G,T)$, as was to be shown.
\end{proof}

\subsection{Orbi-CW-complexes (topology)}

We now define \emph{orbi-CW-complexes} (the definition we give realizes in some form a proposal of Gepner--Henriques \cite{gepnerhenriques}, but differs on some key details).
An orbi-CW-complex $X$ is specified as follows.
We begin with the `$(-1)$-skeleton' $X_{-1}:=\varnothing$.
The $k$-skeleton $X_k$ is defined in terms of $X_{k-1}$ by attaching cells of the form $D^k\times\BB G$ for finite groups $G$ along representable attaching maps $\partial D^k\times\BB G\to X_{k-1}$.
In other words, $X_k$ is defined as the pushout
\begin{equation}\label{attachcell}
\begin{tikzcd}
\smash{\displaystyle\bigsqcup_\alpha\partial D^k\times\BB G_\alpha}\ar[r,"\bigsqcup_\alpha f_\alpha"]\ar[d]&X_{k-1}\ar[d]\\
\displaystyle\bigsqcup\limits_\alpha D^k\times\BB G_\alpha\ar[r]&X_k
\end{tikzcd}
\end{equation}
in the category of topological stacks admitting an \'etale atlas, which exists by Proposition \ref{collarpushout}, which also guarantees that $X_k$ is an orbispace.
The orbispace $X$ is now defined as the ascending union
\begin{equation}\label{unionofskeleta}
X:=\colim_kX_k
\end{equation}
which exists and is an orbispace by Proposition \ref{mcinfinite}.
Since the coarse space functor $\left|\cdot\right|$ preserves colimits (since it is a left adjoint), it follows that the coarse space of an orbi-CW-complex is a CW-complex, with exactly the same attaching maps.

An orbi-CW-complex is equivalently a pair $(X,u_{k,\alpha})$ where $X$ is an orbispace and $\{u_{k,\alpha}:D^k\times\BB G_\alpha\to X\}_{k,\alpha}$ is a collection of representable maps which satisfy the following inductive condition: the restriction $u_{k,\alpha}|_{\partial D^k\times\BB G_\alpha}$ has image contained in the closed substack $X_{k-1}\subseteq X$ (begin with $X_{-1}:=\varnothing$), the resulting map $X_k:=X_{k-1}\cup_{u_{k,\alpha}}\bigsqcup_\alpha D^k\times\BB G_\alpha\to X$ is a closed inclusion, and $X$ is the colimit of the closed substacks $X_k$.

\begin{lemma}\label{checkcwatcoarse}
A pair $(X,u_{k,\alpha})$ consisting of an orbispace $X$ and a collection of representable maps $u_{k,\alpha}:D^k\times\BB G_\alpha\to X$ is an orbi-CW-complex iff the pair $(\left|X\right|,\left|u_{k,\alpha}\right|)$ is a CW-complex and all $u_{k,\alpha}|_{(D^k)^\circ\times\BB G_\alpha}$ induce isomorphisms on isotropy groups.
\end{lemma}

\begin{proof}
We show by induction that $X_k\subseteq X$ is the closed substack corresponding to the closed subset $\left|X\right|_k\subseteq\left|X\right|$ induced by the CW-structure $(\left|X\right|,\left|u_{k,\alpha}\right|)$.
The image of $u_{k,\alpha}|_{\partial D^k\times\BB G_\alpha}$ is contained in the closed substack $X_{k-1}\subseteq X$ since this can be checked at the level of coarse spaces.
We have by assumption that $\left|X\right|_{k-1}\cup_{\mathopen|u_{k,\alpha}\mathclose|}\bigsqcup_\alpha D^k\to\left|X\right|$ is a closed inclusion, and we would like to show that $X_{k-1}\cup_{u_{k,\alpha}}\bigsqcup_\alpha D^k\times\BB G_\alpha\to X$ is a closed inclusion, with the same image.
The coarse space of the domain of the second map coincides with the domain of the first map since coarse space commutes with colimits and $\left|X\right|_{k-1}=\left|X_{k-1}\right|$ by the induction hypothesis.
The second map therefore factors through the closed substack of $\left|X\right|$ corresponding to $\left|X\right|_k\subseteq\left|X\right|$ (the image of the first map, by definition).
To check that the first map of this factorization is an isomorphism, it suffices by Lemma \ref{orbiiso} to note that it induces a homeomorphism on coarse spaces and isomorphisms on isotropy groups (by the hypothesis on $u_{k,\alpha}|_{(D^k)^\circ\times\BB G_\alpha}$).

Next, we should show that the map $\colim_kX_k\to X$ is an isomorphism.
Again by Lemma \ref{orbiiso}, it suffices to note that it induces isomorphisms on isotropy groups (immediate since $X_k\subseteq X$ are closed substacks) and induces a homeomorphism on coarse spaces (since $(\left|X\right|,\left|u_{k,\alpha}\right|)$ is a CW-complex).
\end{proof}

\begin{definition}
A \emph{subcomplex} $A$ of an orbi-CW-complex $X$ consists of a subset of the set of cells of $X$ such that the attaching map of any $k$-cell in $A$ lands inside $A_{k-1}\subseteq X_{k-1}$ (which is a closed substack by induction on $k$).
By Lemma \ref{checkcwatcoarse}, subcomplexes of $X$ are in bijection with subcomplexes of $\left|X\right|$.
\end{definition}

Given two orbi-CW-complexes $(X,u_{k,\alpha})$ and $(Y,v_{\ell,\beta})$, we may ask whether their product $(X\times Y,u_{k,\alpha}\times v_{\ell,\beta})$ is an orbi-CW-complex (note that a product of cells $D^k\times\BB G_\alpha\times D^\ell\times\BB G_\beta$ is indeed a cell $D^{k+\ell}\times\BB(G_\alpha\times G_\beta)$).
In view of Lemma \ref{checkcwatcoarse} and Lemma \ref{coarseproduct}, this reduces to the corresponding question for the ordinary CW-complexes obtained by passing to coarse spaces.
It is known that a product of CW-complexes is a CW-complex if at least one of the factors is locally finite \cite{whiteheadcwproduct} or if both factors are locally countable \cite{milnorcwproduct} (in fact, the question of when a product of two CW-complexes is a CW-complex is completely solved in \cite{brooketaylorcwproduct}).

\section{Homotopy theory of orbispaces}

\subsection{Homotopies}

Two maps of orbispaces $f,g:X\to Y$ are called homotopic iff there exists a map $h:X\times[0,1]\to Y$ such that $h(0,\cdot)$ and $h(1,\cdot)$ are isomorphic to $f$ and $g$, respectively.
Note that, in particular, if $f$ and $g$ are isomorphic, then they are homotopic.
Homotopy classes of maps form a set.
A map with a two-sided inverse up to homotopy is called a homotopy equivalence.

\begin{example}
Homotopy classes of maps $\BB G\to\BB H$ are in bijection with conjugacy classes of group homomorphisms $G\to H$.
The `space' (properly defined) of maps $\BB G\to\BB H$ would be the homotopy quotient $\Hom(G,H)\hq H$.
\end{example}

\begin{lemma}\label{homotopicrepresentable}
If $f$ and $g$ are homotopic, then $f$ is representable iff $g$ is representable.
\end{lemma}

\begin{proof}
Since for maps of orbispaces, representability is equivalent to injectivity on isotropy groups, it suffices to consider the case of maps from $\BB G$.
Thus consider a map $\BB G\times[0,1]\to Y$.
Locally $Y=U/\Gamma$ for $\Gamma$ finite acting on $U$ Hausdorff.
So, a map $\BB G\times[0,1]\to Y$ is (locally) a map $G\to\Gamma$ and a map $[0,1]\to U$ landing in the $G$-fixed locus.
This is injective on isotropy groups iff $G\to\Gamma$ is injective, which is obviously an open and closed condition on $[0,1]$.
\end{proof}

\begin{corollary}\label{heqrepresentable}
A homotopy equivalence of orbispaces is representable.
\end{corollary}

\begin{proof}
Let $f$ and $g$ be homotopy inverses of each other.
Since $f\circ g$ and $g\circ f$ are homotopic to identity maps, they are representable by Lemma \ref{homotopicrepresentable}.
Thus $f\circ g$ and $g\circ f$ are both injective on isotropy groups, from which it follows that $f$ and $g$, respectively, are injective on isotropy groups.
\end{proof}

\begin{lemma}\label{repfactor}
Let $f:Y\times\BB G\to X$ be a map where $Y$ is a topological space and $X$ is an orbispace.
There exists a partition of $Y$ into open subsets $Y_N$ indexed by the normal subgroups $N\trianglelefteq G$, \emph{representable} maps $Y_N\times\BB(G/N)\to X$, and an isomorphism between $f$ and the composition
\begin{equation}
Y\times\BB G\to\bigsqcup_{N\trianglelefteq G}Y_N\times\BB(G/N)\to X.
\end{equation}
Moreover, this data is unique up to unique isomorphism.
\end{lemma}

\begin{proof}
Applying Lemma \ref{descentgroupoid} to the atlas $Y\to Y\times\BB G$, we find that the data of a map $Y\times\BB G\to X$ is the same as the data of a map $f:Y\to X$ together with a homomorphism $G\to\Aut(f)$.
Moreover, for a point $y\in Y$, the homomorphism $G\to G_{f(y)}$ induced by restricting $G\to\Aut(f)$ coincides with the action of the corresponding map $Y\times\BB G\to X$ on isotropy groups.
Now we recall that for maps of orbispaces, representability is equivalent to injectivity on isotropy groups.
It thus suffices to show that for any map $f:Y\to X$ and any a homomorphism $G\to\Aut(f)$, the map $y\mapsto\ker(G\to G_{f(y)})$ is locally constant.

Thus fix $f:Y\to X$ and $g\in\Aut(f)$, and let us show that the set of $y\in Y$ for which $g|_y\in G_{f(y)}$ is the identity is open and closed.
The set of such $y$ is the fiber product of $Y$ and $X$ over $X\times_{X\times X}X$.
Thus it suffices to show that $X\to X\times_{X\times X}X$ is an open and closed inclusion.
We can check this locally, so we can assume that $X=Z/H$ for $Z$ Hausdorff and $H$ finite.
Then $X\times_{X\times X}X=(\bigsqcup_{h\in H}X^h)/H$ (the action of $H$ is by conjugation) and the map from $X$ is the inclusion of the component $h=\mathbf 1$.
\end{proof}

\subsection{Orbi-CW-complexes (homotopy)}

The basic objects with which we shall do homotopy theory are orbi-CW-complexes.
Many basic facts about CW-complexes generalize immediately to orbi-CW-complexes, with identical proofs.
For example, for $X$ an orbi-CW-complex and $A\subseteq X$ a subcomplex, the pair $(X,A)$ has the homotopy extension property (by the universal property of colimits, we may proceed by induction on cells, for which the statement is obvious since cells have boundary collars).
Every map of orbi-CW-complexes is homotopic to one which sends the $k$-skeleton of the domain to the $k$-skeleton of the target.
Also a homotopy equivalence between orbispaces $X\xrightarrow\sim X'$ and homotopies between attaching maps $\{f_\alpha:\partial D^k\times\BB G_\alpha\to X\}_\alpha$ and $\{f_\alpha':\partial D^k\times\BB G_\alpha\to X'\}_\alpha$ induces a homotopy equivalence
\begin{equation}
X\cup_{\{f_\alpha\}_\alpha}\bigsqcup_\alpha\partial D^k\times\BB G_\alpha\xrightarrow\sim X'\cup_{\{f_\alpha'\}_\alpha}\bigsqcup_\alpha\partial D^k\times\BB G_\alpha.
\end{equation}

\begin{proposition}
A compact orbifold is homotopy equivalent to a finite orbi-CW-complex.
\end{proposition}

\begin{proof}[Proof Sketch]
Let $X$ be a compact orbifold, and choose a Morse function $f:X\to\RR$ as follows.
We define $f$ by induction on the stratification of $X$ by the order of the stabilizer group.
A given stratum is a purely ineffective smooth suborbifold, so we just choose any Morse function on it (generic ones are Morse), and we extend it in the normal directions by a \emph{positive definite} quadratic form.
Note that in this way, at any critical point, the isotropy group acts trivially on the negative eigenspace of the Hessian.
Thus the change in the homotopy type of sublevel sets when passing a critical point of index $k$ is precisely to attach a cell $(D^k,\partial D^k)\times\BB G$.
\end{proof}

We now discuss homotopy groups of orbi-CW-complexes.
For an orbispace $X$, we have a set $\pi_0^G(X)$ of homotopy classes of maps $\BB G\to X$.
These are functorial in $X$ and (contravariantly) in $G$.
More generally, we define $\pi_k^G(X,p)$ for a `basepoint' $p:\BB G\to X$ as the set of maps $f:S^k\times\BB G\to X$ together with an isomorphism between $f|_{*\times\BB G}$ and $p$ (where $*\in S^k$ is a fixed basepoint), modulo homotopy.
A homotopy here means a map $h:[0,1]\times S^k\times\BB G\to X$ together with an isomorphism between $h|_{[0,1]\times *\times\BB G}$ and $p\circ\pi_{*\times\BB G}$.
The sets $\pi_k^G(X,p)$ are functorial in $(X,p)$ and $G$.
As with ordinary homotopy groups, $\pi_k^G(X,p)$ is a pointed set for $k=0$, a group for $k=1$, and an abelian group for $k=2$.

\begin{example}
We have $\pi_0^G(\BB H)=\Hom(G,H)/H$ (quotient by the conjugation action).
For a map $\varphi:G\to H$ (inducing a basepoint $\BB\varphi:\BB G\to\BB H$), we have $\pi_1^G(\BB H,\BB\varphi)=Z_H(\im(\varphi))$ (the centralizer of $\varphi(G)\subseteq H$) and $\pi_k^G(\BB H,\BB\varphi)=0$ for $k\geq 2$.
\end{example}

In view of Lemma \ref{homotopicrepresentable}, there is a distinguished subset $\pi_0^{G,\rep}(X)\subseteq\pi_0^G(X)$ of homotopy classes of representable maps $\BB G\to X$.
Evidently $\pi_0^{G,\rep}(X)$ is functorial under representable maps of $X$ and injective maps of $G$.
The sets $\pi_0^G(X)$ and $\pi_0^{G,\rep}(X)$ contain the same information, in the sense that
\begin{align}
\label{reprecover}\pi_0^{G,\rep}(X)&{}=\pi_0^G(X)\setminus\bigcup_{\1\ne N\trianglelefteq G}\im(\pi_0^{G/N}(X)\to\pi_0^G(X)),\\
\pi_0^G(X)&{}=\bigsqcup_{N\trianglelefteq G}\pi_0^{G/N,\rep}(X).
\end{align}
A basepoint $p:\BB G\to X$ factors uniquely as $\BB G\to\BB(G/\ker p)\xrightarrow{p^\rep}X$ where the second map is representable, and by Lemma \ref{repfactor} we have $\pi_k^G(X,p)=\pi_k^{G/\ker p}(X,p^\rep)$ for $k\geq 1$ (so in this sense the information in the homotopy groups of an orbispace $X$ is already contained in the case of representable basepoints).

We will also make use of relative homotopy groups $\pi_k^G(X,Y,p)$ for a map $Y\to X$ and a basepoint $p:\BB G\to Y$.
For $k\geq 1$, an element of $\pi_k(X,Y,p)$ is represented by a diagram
\begin{equation}\label{relativehtpydiagram}
\begin{tikzcd}
\partial D^k\times\BB G\ar[r]\ar[d]&Y\ar[d]\\
D^k\times\BB G\ar[r]&X
\end{tikzcd}
\end{equation}
together with an isomorphism between the restriction of $\partial D^k\times\BB G\to Y$ to $*\times\BB G$ and the basepoint $p$.
These are considered up to homotopy, i.e.\ diagrams in which the orbispaces on the left are replaced with their product with $[0,1]$, and we specify an isomorphism with $p\circ\pi_{*\times\BB G}$ over $[0,1]\times *\times\BB G$.
Now $\pi_k^G(X,Y,p)$ is a pointed set for $k=1$, a group for $k=2$, and an abelian group for $k\geq 3$.

It is essentially immediate from the definitions that for a map $Y\to X$ and a basepoint $p:\BB G\to Y$, there is a long exact sequence (of pointed sets)
\begin{equation}\label{htpyles}
\cdots\to\pi_2^G(X,Y,p)\to\pi_1^G(Y,p)\to\pi_1^G(X,p)\to\pi_1^G(X,Y,p)\to\pi_0^G(Y,p)\to\pi_0^G(X,p).
\end{equation}
It is thus natural to define $\pi_0^G(X,Y)$ as the pointed set $\pi_0^G(X)/\pi_0^G(Y)$.
We now have the following version of Whitehead's theorem.

\begin{proposition}\label{whitehead}
A map of orbi-CW-complexes is a homotopy equivalence iff it induces isomorphisms on $\pi_k^G$ for all basepoints and on $\pi_0^G$.
\end{proposition}

\begin{proof}
Let $Y\to X$ be a map of orbi-CW-complexes which induces an isomorphism on all $\pi_k^G$.
In view of \eqref{reprecover}, the map $Y\to X$ respects the subsets $\pi_0^{G,\rep}\subseteq\pi_0^G$.
It follows that $Y\to X$ is representable.

Since $Y\to X$ is representable, we may form its mapping cylinder $\cyl(Y\to X)$; moreover, by first applying cellular approximation to $Y\to X$ to homotope it to be cellular, we may ensure that $\cyl(Y\to X)$ is again an orbi-CW-complex.

It now suffices to construct a dotted lift in the diagram
\begin{equation}
\begin{tikzcd}
Y\ar[r]\ar[d]&Y\ar[d]\\
\cyl(Y\to X)\ar[r]\ar[ru,dashed]&X
\end{tikzcd}
\end{equation}
after possibly homotoping the bottom map rel $Y$.
Since $Y\subseteq\cyl(Y,X)$ is a subcomplex, it is equivalent to solve the homotopy lifting problem
\begin{equation}
\begin{tikzcd}
\partial D^k\times\BB G\ar[r]\ar[d]&Y\ar[d]\\
D^k\times\BB G\ar[r]\ar[ru,dashed]&X
\end{tikzcd}
\end{equation}
which is equivalent to the vanishing of all relative homotopy groups of $Y\to X$.
We are thus done by the long exact sequence \eqref{htpyles}.
\end{proof}

\begin{conjecture}
Any metrizable locally tame topological orbifold is homotopy equivalent to an orbi-CW-complex.
\end{conjecture}

A first step towards proving this was taken in \cite[Proposition 4.6]{orbibundle} which shows that there always exists a representable map $f$ from any paracompact orbispace to the geometric realization of a simplicial complex of groups (which is particular case of an orbi-CW-complex).
The next step would be to argue that, choosing the cover in the proof of this result to be sufficiently fine, it is possible to construct a map $g$ in the reverse direction (using equivariant contractibility of $\RR^n$ with respect to a linear $G$-action) and moreover that $g\circ f$ is homotopic to the identity.
One can then add cells to the target of $f$ (extending $g$ appropriately) until $f$ and $g$ both induce isomorphisms on all $\pi_k^G$ and then apply Proposition \ref{whitehead} to $f\circ g$.
This is how the standard proof of the corresponding assertion for topological manifolds goes.
That result also extends to absolute neighborhood retracts, so it is natural to ask whether this extension has a generalization to the orbispace setting.
One could also reasonably conjecture that a \emph{compact} locally tame topological orbifold is homotopy equivalent to a \emph{finite} orbi-CW-complex.
This is true for manifolds by the work of Kirby--Siebenmann \cite{ksbull,ksessays}, and other proofs were given later by West \cite{west} and Chapman \cite{chapman}; see Ferry--Ranicki \cite{ferryranicki} for further discussion.

\subsection{Homotopy categories of orbispaces}

We denote by $\OrbSpc$ (resp.\ $\RepOrbSpc$) the category whose objects are orbi-CW-complexes and whose morphisms are (resp.\ representable) homotopy classes of maps.
Note that by Lemma \ref{homotopicrepresentable}, representability is preserved by homotopies, and a homotopy between representable maps is itself representable.
The tautological functor $\RepOrbSpc\to\OrbSpc$ is thus faithful and conservative.
The homotopy category of CW-complexes is denoted $\Spc$, which is a full subcategory of both $\RepOrbSpc$ and $\OrbSpc$.

The functor $\Spc\hookrightarrow\OrbSpc$ has a left adjoint, namely the coarse space functor $\left|\cdot\right|:\OrbSpc\to\Spc$ which, as noted earlier, sends orbi-CW-complexes to CW-complexes.

We use $\Spc^f\subseteq\Spc$, $\RepOrbSpc^f\subseteq\RepOrbSpc$, and $\OrbSpc^f\subseteq\OrbSpc$ to denote the full subcategories spanned by finite (orbi-)CW-complexes.
Note that the adjectives `finite'/`compact' and the resulting notation for full subcategories is sometimes taken to instead indicate those objects which are compact objects in the categorical sense (finite orbi-CW-complexes are compact objects categorically, but the converse is not true).

\subsection{Classifying space}\label{classifyingspace}

The inclusion $\Spc\hookrightarrow\OrbSpc$ has a right adjoint denoted $X\mapsto\tilde X$, where $\tilde X$ is known as the \emph{classifying space} of $X$.
This right adjoint may be constructed as follows.
It suffices to show that for any orbi-CW-complex $X$, there exists a CW-complex $\tilde X$ and a map $\tilde X\to X$ such that the homotopy lifting problem
\begin{equation}\label{classifylift}
\begin{tikzcd}
\partial D^k\ar[r]\ar[d]&\tilde X\ar[d]\\
D^k\ar[r]\ar[ru,dashed]&X
\end{tikzcd}
\end{equation}
always has a solution.
Indeed, this implies that the map $\tilde X\to X$ induces a bijection between homotopy classes of maps $Y\to\tilde X$ and homotopy classes of maps $Y\to X$ for any CW-complex $Y$.
We may construct $\tilde X$ as follows.
We begin with $\tilde X_{-1}=\varnothing$.
We define $\tilde X_k$ from $\tilde X_{k-1}$ by attaching a $k$-cell for every element of $\pi_k(X,\tilde X_{k-1})$ (for every basepoint when $k>0$).
By cellular approximation and induction, we have $\pi_i(X,\tilde X_k)=0$ for all $i\leq k$.
It follows that $\pi_i(X,\tilde X)=0$ for every $i$, which is equivalent to solvability of the above lifting problem.

\begin{example}\label{bgtilde}
For a CW-complex $X$, the set of homotopy classes of maps $X\to\BB G$ is in natural bijection with the set of isomorphism classes of principal $G$-bundles over $X$, which is in turn in bijection with the set of homotopy classes of maps $X\to BG$.
It follows that $\widetilde{\BB G}=BG$.
\end{example}

The natural map $(X\times Y)^\sim\to\tilde X\times\tilde Y$ is an isomorphism since $X\mapsto\tilde X$ is a right adjoint.

\begin{lemma}\label{classifyglue}
Let $X$ be an orbi-CW-complex covered by subcomplexes $P,Q\subseteq X$ intersecting in $A:=P\cap Q$, so $X=P\cup_AQ$.
Fix classifying spaces $\tilde P\to P$, $\tilde Q\to Q$, and $\tilde A\to A$, with subcomplex inclusions $\tilde A\to\tilde P$ and $\tilde A\to\tilde Q$ such that the diagram
\begin{equation}
\begin{tikzcd}
\tilde P\ar[d]&\tilde A\ar[l]\ar[r]\ar[d]&\tilde Q\ar[d]\\
P&A\ar[r]\ar[l]&Q
\end{tikzcd}
\end{equation}
strictly commutes (this may be achieved by replacing $\tilde P$ and $\tilde Q$ by the mapping cylinders of $\tilde A\to\tilde P$ and $\tilde A\to\tilde Q$).
Then $\tilde X=\tilde P\cup_{\tilde A}\tilde Q$ with the obvious map to $P\cup_AQ=X$.
\end{lemma}

\begin{proof}
Given a CW-complex $Z$ and a map $Z\to X=P\cup_AQ$, let us lift it (up to homotopy) to $\tilde P\cup_{\tilde A}\tilde Q$.
By subdividing $Z$ and homotoping the map $Z\to X$, we may assume that each cell of $Z$ maps either entirely to $P$ or entirely to $Q$.
Now we first lift the cells which map to $A=P\cap Q$ to $\tilde A$.
A cell which maps to $P$ (resp.\ $Q$) but not entirely to $A$ is now lifted to $\tilde P$ (resp.\ $\tilde Q$).
This shows that the map from homotopy classes of maps $Z\to\tilde P\cup_{\tilde A}\tilde Q$ to homotopy classes of maps $Z\to X=P\cup_AQ$ is surjective.
To show injectivity, apply the same argument rel boundary to a homotopy between the compositions of two maps $Z\to\tilde P\cup_{\tilde A}\tilde Q$ with $\tilde P\cup_{\tilde A}\tilde Q\to P\cup_AQ$.
Note that in this proof we used the lifting property \eqref{classifylift} which is \emph{a priori} stronger than (albeit \emph{a posteriori} equivalent to) the adjointness property of the classifying space at the level of homotopy categories (the lifting property instead corresponds to a universal property at the $\infty$-category level).
\end{proof}

One can similarly argue by induction on cells that one can construct $\tilde X$ by taking each cell $(D^k,\partial D^k)\times\BB G$ of $X$ and replacing it with $(D^k,\partial D^k)\times BG$ and attaching as appropriate.

\begin{remark}
There are other, more point set topological, definitions of the classifying space of more general topological stacks.
These include taking $\tilde X$ to be the nerve of the simplicial space $[p]\mapsto U\times_X\cdots\times_XU$ ($p+1$ times) where $U\to X$ is a suitable atlas.
It is also possible to define $\tilde X\to X$ by the universal property that $\tilde X\times_XZ\to Z$ should be `fiberwise contractible' for any topological space $Z$ mapping to $X$ (compare Noohi \cite{noohiclassifying}).
We will not make either of these definitions precise, nor prove that they give the right adjoint of $\Spc\hookrightarrow\OrbSpc$, though this is also possible.
\end{remark}

\subsection{Right adjoint}\label{rightadjoint}

Here is another adjoint.

\begin{proposition}\label{proprightadjoint}
The functor $\RepOrbSpc\to\OrbSpc$ has a right adjoint $R$.
\end{proposition}

\begin{proof}
It suffices to show that for every orbi-CW-complex $X$, there exists an orbi-CW-complex $R(X)$ and a map $R(X)\to X$ such that for every commuting diagram of solid arrows
\begin{equation}\label{rlifting}
\begin{tikzcd}
\partial D^k\times\BB G\ar[r,"\rep"]\ar[d]&R(X)\ar[d]\\
D^k\times\BB G\ar[r]\ar[ru,"\rep",dashed]&X
\end{tikzcd}
\end{equation}
there exists, after possibly homotoping the bottom map rel boundary, a dotted lift.
Indeed, given such a map $R(X)\to X$, it follows that the induced map $\RepOrbSpc(Z,R(X))\to\OrbSpc(Z,X)$ is a bijection (proof by induction on cells), which implies that the adjoint $R$ exists as a functor.

We may now construct the orbi-CW-complex $R(X)$ inductively just as we constructed $\tilde X$ above.
We begin with $R(X)_{-1}=\varnothing$, and we define $R(X)_k$ by attaching copies of $D^k\times\BB G$ to $R(X)_{k-1}$ along the upper horizontal map in diagrams \eqref{rlifting} with $R(X)_{k-1}$ in place of $R(X)$ representing every homotopy class of such (note that the attaching maps are by definition representable).
Using cellular approximation and induction, it follows that $R(X)_r$ satisfies the desired lifting property \eqref{rlifting} for all $k\leq r$; hence $R(X)$ is as desired.
\end{proof}

The proof of Proposition \ref{proprightadjoint} given above shows existence, but is not so amenable to computation.
So, let us sketch what we expect is a more concrete definition of the functor $R:\OrbSpc\to\RepOrbSpc$, without claiming to give a complete proof.
We define a functor $R:\OrbSpc\to\RepOrbSpc$ as
\begin{equation}\label{Rformula}
R(X):=\bigsqcup_{G_0\hookrightarrow\cdots\hookrightarrow G_p}\Delta^p\times\BB G_0\times\Mapstilde(\BB G_p,X)\Bigm/{\sim}.
\end{equation}
The reader may recognize this formula as a `homotopy coend'.
Here $\Mapstilde(\BB G,X)$ denotes the classifying space of the mapping orbispace $\Maps(\BB G,X)$, which is defined by the universal property that a map $Y\to\Maps(\BB G,X)$ is the same as a map $Y\times\BB G\to X$.
When $X$ is an orbi-CW-complex, one can instead be much more concrete: $\Mapstilde(\BB G,X)$ may be defined by replacing each cell $D^k\times\BB H$ in $X$ with $D^k\times\Hom(G,H)\hq H$, where $H\acts\Hom(G,H)$ by conjugation and $\hq$ denotes the homotopy quotient.
Now \eqref{Rformula} is meant to be modelled on the nerve of the $2$-category $\InjFinGrp$ of finite groups, injective homomorphisms, and conjugations; the quotient $\sim$ indicates the colimit over the natural face and degeneracy identifications.

\begin{conjecture}\label{Rconjecture}
The expression \eqref{Rformula} defines a functor $R:\OrbSpc\to\RepOrbSpc$ which is right adjoint to $\RepOrbSpc\to\OrbSpc$.
\end{conjecture}

\begin{proof}[Proof Sketch]
There is a tautological map $R(X)\to X$, and it suffices to show that the induced map
\begin{equation}
\smash{\RepMapstilde}(\BB G,R(X))\to\Mapstilde(\BB G,X)
\end{equation}
is a homotopy equivalence.
Now $\smash{\RepMapstilde}(\BB G,R(X))$ is given by
\begin{equation}
\bigsqcup_{G_0\hookrightarrow\cdots\hookrightarrow G_p}\Delta^p\times\smash{\RepMapstilde}(\BB G,\BB G_0)\times\Mapstilde(\BB G_p,X)\Bigm/{\sim}.
\end{equation}
Now $\smash{\RepMapstilde}(\BB G,\BB G_0)$ is just the classifying space of the groupoid of morphisms $G\to G_0$ in $\InjFinGrp$, so we can equivalently write this as
\begin{equation}
\bigsqcup_{G\hookrightarrow G_0\hookrightarrow\cdots\hookrightarrow G_p}\Delta^p\times\Mapstilde(\BB G_p,X)\Bigm/{\sim}.
\end{equation}
This being a homotopy colimit over a category with an initial object simply reduces to $\Mapstilde(\BB G,X)$, as desired.
\end{proof}

The space $R(*)\in\RepOrbSpc$ is the terminal object, and so may be expected to play a role in the homotopy theory of orbispaces (for example, see Conjectures \ref{spacesoverR} and \ref{spectraoverR}).
It is characterized by the universal property that the space of representable maps to it is contractible (in the precise sense that $\pi_0^{G,\rep}(R(*))=*$ for all $G$, and $\pi_k^G(R(*),p)=0$ for representable, hence all, basepoints $p$).
Combining this with Lemma \ref{repfactor} implies that the space of all (not necessarily representable) maps from $\BB G$ to $R(*)$ is homotopy equivalent to the (discrete) set of normal subgroups of $G$.
It follows that $R(*)$ is, in Rezk's language, the normal subgroup classifier $\mathcal N$ \cite{rezkcohesion} (more precisely, the tautological functor $\RepOrbSpc\to\OrbSpc$ sends $R(*)\in\RepOrbSpc$ to $\mathcal N\in\OrbSpc$; from our perspective, the categories $\RepOrbSpc$ and $\OrbSpc$ have the `same' objects, namely orbi-CW-complexes, so it makes sense to simply say that $R(*)$ is $\mathcal N$, however in Rezk's setup the functor $R:\OrbSpc\to\RepOrbSpc$ is the more natural one, being given by a restriction of presheaves, so the use of its left adjoint $\RepOrbSpc\to\OrbSpc$ becomes more significant).

Specializing \eqref{Rformula} gives
\begin{equation}\label{rpoint}
R(*)=\bigsqcup_{G_0\hookrightarrow\cdots\hookrightarrow G_p}\Delta^p\times\BB G_0\Bigm/{\sim},
\end{equation}
which is an orbi-CW-complex, and one can follow the proof sketch above to see that it is indeed $R(*)$.
Since every object of $\RepOrbSpc$ admits a unique up to homotopy representable map to $R(*)$, we may think of objects of $\RepOrbSpc$ informally as being `representable over $R(*)$'.
More precisely, we make the following conjecture (a form of which is proven by Rezk \cite[Proposition 4.6.1]{rezkcohesion}):

\begin{conjecture}\label{spacesoverR}
The category $\RepOrbSpc$ is equivalent to the category of representable fibrations over $R(*)$ (with reasonable fibers).
\end{conjecture}

We can specify this further: the equivalence should send a representable fibration over $R(*)$ to its total space, and the fiber over a generic point of $R(*)$ should be the classifying space $\tilde X$ of the orbispace $X$.
There are also interesting functors to $G$-spaces given by pulling back under the unique up to contractible choice representable map $\BB G\to R(*)$.
In fact, it seems that fibrations over $R(*)$ should be the same (in the $\infty$-categorical context) as $\PSh(\Rep\{\BB G\})$ where $\Rep\{\BB G\}\subseteq\RepOrbSpc$ denotes the full subcategory spanned by the objects $\BB G$, and $\PSh$ denotes presheaves.
The proposed formula \eqref{Rformula} and the sketch of proof of Conjecture \ref{Rconjecture} in fact would apply to define an inverse of the restricted Yoneda functor $\RepOrbSpc\to\PSh(\Rep\{\BB G\})$.

\subsection{Homotopy categories of relative orbispaces}

We now introduce categories $\OrbSpc_*$ and $\RepOrbSpc_*$ of `relative orbispaces'.
These categories should be thought of as analogues of the category $\Spc_*$ of pointed CW-complexes and homotopy classes of pointed maps.
They are, however, \emph{not} the (representable) homotopy categories of pointed orbi-CW-complexes and homotopy classes of pointed maps.
The reason that pointed orbi-CW-complexes and pointed maps is not what we want to consider may be traced back to the fact that there is no well-defined quotient orbi-CW-complex $X/A$ of a given orbi-CW-pair $(X,A)$.

We begin with the categories of orbi-CW-pairs $\OrbSpcPair$ and $\RepOrbSpcPair$, whose objects are orbi-CW-pairs $(X,A)$ (meaning $X$ is an orbi-CW-complex and $A\subseteq X$ is a subcomplex), and whose morphisms $(X,A)\to(Y,B)$ are (representable) commutative squares.
Product of pairs is defined as usual $(X,A)\times(Y,B):=(X\times Y,(A\times Y)\cup(X\times B))$, as is the notion of homotopies between maps of pairs.

Now the categories $\OrbSpc_*$ and $\RepOrbSpc_*$ of `relative orbispaces' are defined as follows.
The objects are again orbi-CW-pairs $(X,A)$.
A `relative map' of orbi-CW-pairs $(X,A)\dashrightarrow(Y,B)$ consists of a closed set $A\subseteq A^+\subseteq X$, an open set $U\subseteq X$ with $X=U\cup(A^+)^\circ$, and a map of pairs $(U,U\cap A^+)\to(Y,B)$.
The composition of two relative maps
\begin{equation}
\begin{tikzcd}[column sep = large]
(X,A)\ar[r,dashed,"{(A^+,U,f)}"]&(Y,B)\ar[r,dashed,"{(B^+,V,g)}"]&(Z,C)
\end{tikzcd}
\end{equation}
is the triple $(A^+\cup f^{-1}(B^+),f^{-1}(V),g\circ f)$; composition is associative.
A homotopy between relative maps is a relative map $(X,A)\times[0,1]\to(Y,B)$.
The morphisms in $\OrbSpc_*$ and $\RepOrbSpc_*$ are (representable) relative maps modulo (representable) homotopy.
Note that it is not true that representability is preserved under homotopy, nor that a homotopy between representable maps is necessarily representable.
There is a tautological functor
\begin{equation}
\RepOrbSpc_*\to\OrbSpc_*
\end{equation}
which is not faithful.
The homotopy category of pointed CW-complexes $\Spc_*$ is a full subcategory of both $\RepOrbSpc_*$ and $\OrbSpc_*$.

\begin{proposition}[Excision]\label{excision}
The functors
\begin{align}
\OrbSpcPair&{}\to\OrbSpc_*\\
\RepOrbSpcPair&{}\to\RepOrbSpc_*
\end{align}
are localizations at the collection $W$ of morphisms of the form $(P,P\cap Q)\to(X,Q)$ where $X=P\cup Q$ is a cover by subcomplexes.

The same holds if we restrict both sides to the full subcategories spanned by finite orbi-CW-pairs.
\end{proposition}

\begin{proof}
First note that the morphisms $W$ in $\RepOrbSpcPair$ do indeed become isomorphisms in $\RepOrbSpc_*$ (hence also in $\OrbSpc_*$).
Indeed, such morphisms are, up to isomorphism in $\RepOrbSpcPair$, of the form
\begin{equation}
(X,A)\to(X\cup_{B\times\{0\}}(B\times[0,1])\cup_{B\times\{1\}}Y,A\cup_{B\times\{0\}}(B\times[0,1])\cup_{B\times\{1\}}Y)
\end{equation}
and these have an evident inverse up to homotopy in $\RepOrbSpc_*$.

We now show that $\OrbSpcPair\to\OrbSpc_*$ satisfies the universal property of localization at $W$, namely that for any functor $\OrbSpcPair\to\C$ which sends all morphisms in $W$ to isomorphisms factors uniquely up to unique isomorphism through $\OrbSpc_*$ (and the same for $\RepOrbSpcPair\to\RepOrbSpc_*$).
Let $F:\OrbSpcPair\to\C$ be given.
The action of $\tilde F$ on objects is fixed since $\OrbSpcPair\to\OrbSpc_*$ is essentially surjective.
We are thus reduced to showing that there exists a unique collection of maps $\tilde F:\OrbSpc_*((X,A),(Y,B))\to\C(F(X,A),F(Y,B))$ factoring $F$ which are compatible with composition.

Given a map $(A^+,U,f):(X,A)\to(Y,B)$ in $\OrbSpc_*$, it factors as
\begin{equation}
(X,A)\to(X,A^+)\xleftarrow W(U,U\cap A^+)\xrightarrow f(Y,B).
\end{equation}
By subdividing $X$, we may shrink $U\subseteq X$ and $A^+\subseteq X$ to be subcomplexes covering $X$ (so $U$ is, in particular, likely no longer open).
Since $F$ sends $W$ to isomorphisms it follows that $\tilde F$ applied to this map is determined uniquely by $F$.
It is a tautology that $\tilde F(A^+,U,f)$ defined in this way is invariant under homotopy of $(A^+,U,f)$, simply because the two maps $(X,A)\to(X\times[0,1],A\times[0,1])$ coincide in $\RepOrbSpcPair$.

Finally, we should check that $\tilde F$ respects composition, which follows from the following commuting diagram:
\begin{equation}
\begin{tikzcd}
{}&(X,A)\ar[d]\ar[ld]\\
(X,A^+\cup f^{-1}(B^+))&\ar[l](X,A^+)\\
(U,U\cap(A^+\cup f^{-1}(B^+)))\ar[u,"W"]\ar[rd,"f"]&(U,U\cap A^+)\ar[u,"W"]\ar[l]\ar[d,"f"]\\
(f^{-1}(V),f^{-1}(V)\cap(A^+\cup f^{-1}(B^+)))\ar[u,"W"]\ar[rd,"f"]&(Y,B^+)\\
&\ar[u,"W"](V,V\cap B^+)\ar[d,"g"]\\
&(Z,C)
\end{tikzcd}
\end{equation}
The point here is that once the maps $W$ are declared to be isomorphisms, commutativity of the diagram implies (being careful about the directions of the maps) that the rightmost vertical composition coincides with the leftmost vertical composition.

To see that the same holds after restricting to finite orbi-CW-pairs, we just need to observe that if then input orbi-CW-pairs in the above proof are all finite, then the additional orbi-CW-pairs appearing in the intermediate constructions can also be taken to be finite.
\end{proof}

The functor $\Spc_*\to\OrbSpc_*$ has both adjoints.
The existence of the left adjoint (the coarse space) is immediate (send an orbi-CW-pair $(X,A)$ to the CW-pair $(\left|X\right|,\left|A\right|)$).
For the existence of the right adjoint (the classifying space), we argue as in Lemma \ref{classifyglue}.
Given an orbi-CW-complex $(X,A)$, we may find classifying spaces $\tilde X\to X$ and $\tilde A\to A$ so that $\tilde A\subseteq\tilde X$ is a subcomplex and the classifying maps together define a map of pairs $(\tilde X,\tilde A)\to(X,A)$.
The argument of the proof of Lemma \ref{classifyglue} then shows that this map exhibits $(\tilde X,\tilde A)$ as the classifying space of $(X,A)$.

There is a symmetric monoidal `smash product' $\wedge$ on $\RepOrbSpc_*^f$ and $\OrbSpc_*^f$ defined as follows.
Product of finite orbi-CW-pairs $(X,A)\times(Y,B):=(X\times Y,(A\times Y)\cup(X\times B))$ is a symmetric monoidal structure on $\RepOrbSpcPair^f$ and $\OrbSpcPair^f$.
To see that it descends to $\RepOrbSpc_*^f$ and $\OrbSpc_*^f$, it suffices by Lemma \ref{excision} to note that $(P,P\cap Q)\times(Y,B)\to(X,Q)\times(Y,B)$ is again of the form $(P',P'\cap Q')\to(X',Q')$, namely $X'=X\times Y$, $P'=P\times Y$, and $Q'=(Q\times Y)\cup(X\times B)$.

Let us argue that there is a natural isomorphism $(Z\wedge W)^\sim=\tilde Z\wedge\tilde W$.
If $Z=(X,A)$ and $W=(Y,B)$, then (recalling that the classifying space of $(X,A)$ is $(\tilde X,\tilde A)$) this is the assertion that
\begin{equation}
((X\times Y)^\sim,((A\times Y)\cup_{A\times B}(X\times B))^\sim)=(\tilde X\times\tilde Y,(\tilde A\times\tilde Y)\cup_{\tilde A\times\tilde B}(\tilde X\times\tilde B))
\end{equation}
which follows from Lemma \ref{classifyglue} and the fact that classifying space commutes with products.

Given any $\infty$-category such as (the $\infty$-categorical refinement of) $\RepOrbSpc$, there is an $\infty$-category of `pointed objects of $\RepOrbSpc$' namely the under-category of the terminal object, in this case $R(*)$.
It is reasonable to expect this yields the same result as our explicit geometric definition of the category of relative orbispaces:

\begin{conjecture}
There is an equivalence $\RepOrbSpc_*=\RepOrbSpc_{R(*)/}$ as $\infty$-categories.
\end{conjecture}

\begin{conjecture}
The category $\RepOrbSpc_*$ is equivalent to the category of pointed representable fibrations over $R(*)$ and to the category of presheaves of pointed spaces on $\Rep\{\BB G\}$.
\end{conjecture}

\subsection{Cofiber sequences}

A \emph{cofiber sequence} in $\RepOrbSpc_*$ is a three term sequence isomorphic to
\begin{equation}\label{cofiber}
(Y,B)\to(X,A)\to(X,A\cup_BY)
\end{equation}
for an orbi-CW-complex $X$ with two subcomplexes $A,Y\subseteq X$ and $B:=A\cap Y$.

\begin{proposition}\label{makecofibration}
Every morphism $X\to Y$ in $\RepOrbSpc_*$ extends to a cofiber sequence $X\to Y\to Z$.
\end{proposition}

\begin{proof}
Equivalently, we are to show that every morphism in $\RepOrbSpc_*$ is isomorphic to an inclusion $(Y,B)\hookrightarrow(X,A)$ where $X$ is an orbi-CW-complex, $A,Y\subseteq X$ are subcomplexes, and $B=Y\cap A$.

First, note that any map of orbi-CW-pairs $(X,A)\to(Y,B)$ may be replaced by a map of the desired form by first homotoping it to be cellular and then considering $(X,A)\to(\cyl(X\to Y),\cyl(A\to B))$.
Thus it suffices to show that every morphism in $\RepOrbSpc_*$ is isomorphic to the image of a morphism in $\RepOrbSpcPair$.

A general morphism in $\RepOrbSpc_*$ may be expressed in terms of morphisms of orbi-CW-pairs as
\begin{equation}\label{tobemadecofibration}
(X,A)\to(X,A^+)\xleftarrow\sim(V,V\cap A^+)\to(Y,B)
\end{equation}
where $X$ is an orbi-CW-complex, $A^+,V\subseteq X$ are subcomplexes, and $X=V\cup A^+$.
We now consider the gluing
\begin{equation}\label{cofibergluing}
(X,A^+)\cup((V,V\cap A^+)\times I)\cup(Y,B).
\end{equation}
The inclusion of $\cyl((V,V\cap A^+)\to(Y,B))$ (hence also of $(Y,B)$) into this orbi-CW-pair is an isomorphism in $\RepOrbSpc_*$ by Proposition \ref{excision}.
Thus the natural map from $(X,A)$ to \eqref{cofibergluing} is a map of orbi-CW-pairs which becomes isomorphic to our given morphism in $\RepOrbSpc_*$.
\end{proof}

In fact, a quadruple $(X,A,Y,B)$ as above determines not only a three term sequence \eqref{cofiber}, but a half-infinite sequence, each of whose consecutive pairs of morphisms form cofiber sequences.
This so-called `Puppe sequence' takes the form
\begin{equation}\label{puppe}
\cdots\dashrightarrow(Y,B)\times(I^k,\partial I^k)\to(X,A)\times(I^k,\partial I^k)\to(X,A\cup Y)\times(I^k,\partial I^k)\dashrightarrow\cdots
\end{equation}
where the `connecting maps'
\begin{equation}
(X,A\cup Y)\times(I^k,\partial I^k)\dashrightarrow(Y,B)\times(I^{k+1},\partial I^{k+1})
\end{equation}
are defined as $(r,\id_{I^k},\varphi)$ where $r:(X,A\cup Y)\to(Y,B)$ is a retraction defined in a neighborhood of $(Y,B)$, and $\varphi:X\to[0,1]$ is a map which equals $1$ on $Y$ and equals $0$ outside a small neighborhood of $Y$.

\begin{proposition}\label{puppecofiber}
Every consecutive triple in the Puppe sequence is a cofiber sequence.
\end{proposition}

\begin{proof}
The three consecutive terms appearing in \eqref{puppe} certainly form a cofiber sequence.
Shifting forward by one, we have a cofiber sequence
\begin{multline}
(X,A)\times(I^k,\partial I^k)\xrightarrow{\times\{0\}}(X\times I,(A\times I)\cup(Y\times\{1\}))\times(I^k,\partial I^k)\\
\to(X\times I,(X\times\{0\})\cup(A\times I)\cup(Y\times\{1\}))\times(I^k,\partial I^k)
\end{multline}
into whose third term $(Y,B)\times(I,\partial I)\times(I^k,\partial I^k)$ includes isomorphically.
Shifting forward by one again, we have a cofiber sequence
\begin{multline}
(X\times[{\textstyle\frac 12},1],(A\times[{\textstyle\frac 12},1])\cup(Y\times\{1\}))\times(I^k,\partial I^k)\\
\to(X\times I,(X\times\{0\})\cup(A\times I)\cup(Y\times\{1\}))\times(I^k,\partial I^k)\\
\dashrightarrow(X\times[0,{\textstyle\frac 12}],(X\times\{0\})\cup(A\times[0,{\textstyle\frac 12}])\cup(X\times\{{\textstyle\frac 12}\}))\times(I^k,\partial I^k)
\end{multline}
This concludes the proof.
\end{proof}

\begin{example}\label{nocones}
Here is an example to show that there is no similar notion of cofiber sequences in $\OrbSpc_*$.
Consider the map $\BB G\to *$ (this is a map in $\OrbSpc$, and we consider its image in $\OrbSpc_*$ under the natural map $\OrbSpc\to\OrbSpc_*$ given by `disjoint union with a basepoint'), and suppose it has a cofiber $\BB G\to *\to X$ where $X\in\OrbSpc_*$.
Now the defining property of the cofiber is that for any $Y\in\OrbSpc_*$, a map $X\to Y$ is the same thing as a map $*\to Y$ and a null homotopy of the composition $\BB G\to *\to Y$.
On the other hand, a null homotopy of this composition induces a null homotopy of the original map $*\to Y$ by Lemma \ref{repfactor}.
Thus there is a unique map $X\to Y$, namely the zero map (sending everything to the basepoint).
It follows that $X=\varnothing$ is the terminal object of $\OrbSpc_*$.
Now if we additionally suppose that our cofiber sequence extends as the Puppe sequence to give $\BB G\to *\to X\to\BB G\times(I,\partial I)\to(I,\partial I)$, we obtain a contradiction, since $X=\varnothing$, so the cofiber of $X\to Z$ is $Z$ for any $Z\in\OrbSpc_*$.
The key point in this argument was the use of Lemma \ref{repfactor}.
\end{example}

\subsection{Enough vector bundles}

We now recall the `enough vector bundles property' proved in \cite{orbibundle}, which underlies most of our subsequent work in this paper.
We also derive some corollaries which we will also need.

We begin with some definitions.
By `vector bundle' we always mean a finite-dimensional vector bundle.
Recall that for any vector bundle $V$ over an orbispace $X$, the fiber over a point $p:*\to X$ is a vector space $V_p$ which carries a linear action of the isotropy group $G_p$ of $p$.
A vector bundle is called \emph{coarse} iff these isotropy representations $G_p\acts V_p$ are all trivial (this is equivalent to $V$ being pulled back from the coarse space $\left|X\right|$, hence the terminology).
A vector bundle is called \emph{faithful} iff the isotropy representations are all faithful.
A vector bundle is called \emph{module faithful} iff each $V_p$ is faithful $\RR[G_p]$-module (equivalently, every irreducible representation of $G_p$ occurs inside $V_p$).
The pullback of a coarse vector bundle is coarse, and the pullback of a (module) faithful vector bundle under a \emph{representable} map is (module) faithful.

\begin{lemma}\label{faithfulhasfree}
If $V$ is a faithful representation of a finite group $G$, then the open set on which $G$ acts freely $V^\mathrm{free}\subseteq V$ is open and dense.
\end{lemma}

\begin{proof}
The complement of $V^\mathrm{free}\subseteq V$ is the locus $\bigcup_{1\ne H\leq G}V^H$, which is a finite union of proper subspaces.
\end{proof}

\begin{lemma}\label{faithfulpowerall}
If $V$ is a faithful representation of a finite group $G$, then every irreducible representation of $G$ is a direct summand of a tensor power of $V$.
\end{lemma}

\begin{proof}
This is a classical fact with many known proofs whose correct attribution is not known to me.
By Lemma \ref{faithfulhasfree}, there exists a point $x\in V^*$ all of whose translates by $G$ are distinct.
By Weierstrass, there exists a polynomial function on $V^*$ (that is, an element of $\bigoplus_{i=0}^\infty\operatorname{Sym}^iV$) which is approximately a bump function supported around $x$.
The translates of this element under $G$ are thus linearly independent, so their span is a copy of the regular representation of $G$ inside $\bigoplus_{i=0}^\infty\operatorname{Sym}^iV$ (which is in turn contained in $\bigoplus_{i=0}^\infty V^{\otimes i}$).
\end{proof}

It follows from Lemma \ref{faithfulpowerall} that given a faithful vector bundle $E$ over a compact orbispace (or, more generally, an orbispace with isotropy groups of bounded order), there exists an $N<\infty$ such that $\bigoplus_{i=1}^NE^{\otimes i}$ is module faithful.
If $E$ is a module faithful vector bundle over a compact $X$ and $F$ is arbitrary (more generally, $X$ could be paracompact and $F$ bounded dimensional), there exists an $N<\infty$ and an embedding $F\hookrightarrow E^{\oplus N}$.

It was shown in \cite{orbibundle} that every orbispace satisfying certain mild hypotheses admits a faithful vector bundle.
In particular, all compact orbispaces admit faithful vector bundles.
In fact, the construction gives somewhat more precise control on these faithful vector bundles, however for us all we need is the following:

\begin{theorem}[\cite{orbibundle}]\label{enough}
Every finite orbi-CW-complex admits a faithful vector bundle.
\end{theorem}

Note that the same thus holds for any space homotopy equivalent to a finite orbi-CW-complex (such as a compact orbifold-with-boundary), since homotopy equivalences of orbispaces are representable.

\begin{corollary}\label{embedding}
Let $f:X\to Y$ be a representable smooth map of smooth compact orbifolds.
There exists a vector bundle $E/Y$ such that $f$ lifts to a smooth embedding of $X$ into the total space of $E$.
\end{corollary}

(Recall that a smooth embedding of orbifolds is locally modelled on $V/G\subseteq W/G$ for an inclusion $V\subseteq W$ of $G$-representations.)

\begin{proof}
Let $E$ be any module faithful vector bundle over $Y$.
Choose arbitrarily a connection on $E$, and equip $f^*E$ with the pullback connection.
We claim that there exists a section $s$ of $f^*E^N$ whose derivative $ds:TX\to f^*E^N$ is injective.
Indeed, in local coordinates $X=\RR^n/G$ and $E=(\RR^n\times V)/G$ for some actions of $G$ on $\RR^n$ and $V$, consider the map $s$ given by a $G$-equivariant linear map $\RR^n\to V^N$.
Since $V$ contains all irreducible representations of $G$, by taking $N$ large enough we can choose $\RR^n\to V^N$ to be injective.
Thus $ds$ is injective at zero, hence in a neighborhood; cutting it off we can make it compactly supported.
By compactness, we can take the direct sum of finitely many such $s$ to obtain a section $s:X\to f^*E^N$ whose derivative is injective everywhere.
Such $s$ is a smooth immersion.
A smooth immersion may be `separated' by a map from $X$ (necessarily factoring through $\left|X\right|$!) to $\RR^M$, so our desired vector bundle is $E^N\oplus\underline\RR^M$.
\end{proof}

Recall that an orbifold (resp.\ with boundary) is locally modelled on $\RR^n/G$ (resp.\ $(\RR^{n-1}\times\RR_{\geq 0})/G$), and that it is called effective when the homomorphisms $G\to\GL_n(\RR)$ are \emph{injective}.

\begin{corollary}\label{orbicwisorbifold}
Every finite orbi-CW-complex is homotopy equivalent to a compact effective orbifold-with-boundary.
\end{corollary}

In fact, Corollary \ref{orbicwisorbifold} is equivalent to Theorem \ref{enough} since every effective orbifold-with-boundary admits a faithful vector bundle, namely its tangent bundle.

\begin{proof}
We proceed by induction on the number of cells.
Thus suppose that $X$ is a compact effective orbifold-with-boundary and that $Z=X\cup_{\partial D^k\times\BB G}(D^k\times\BB G)$ for some representable map $\partial D^k\times\BB G\to X$, and let us show that $Z$ is homotopy equivalent to a compact effective orbifold-with-boundary.
The strategy is to realize the cell attachment to $X$ as a handle attachment.

By Corollary \ref{embedding}, after replacing $X$ with the total space of (the unit disk bundle of) a vector bundle over it (and smoothing its corners), we may assume that our map $\partial D^k\times\BB G\to X$ is a smooth embedding into $\partial X$.
By the tubular neighborhood theorem, this smooth embedding is locally modelled on the inclusion of the zero section into the total space of a vector bundle $\nu$ over $\partial D^k\times\BB G$.
Now we have $\nu\oplus T(\partial D^k\times\BB G)=T\partial X$ over $\partial D^k\times\BB G$, and identifying the outward normal along $\partial X$ with the inward normal along $\partial D^k\times\BB G$, we obtain an identification $\nu\oplus T(D^k\times\BB G)=TX$ over $\partial D^k\times\BB G$.
By Theorem \ref{enough}, there exists a vector bundle $\eta$ on $Z$ and an embedding $TX\hookrightarrow\eta|_X$.
By replacing $X$ with the total space of $\eta|_X/TX$, we may assume that $TX=\eta|_X$.
We thus have an embedding $T(D^k\times\BB G)\hookrightarrow\eta$ defined over $\partial D^k\times\BB G$.
By further enlarging $\eta$ (and modifying $X$ as this requires), we may ensure that this embedding $T(D^k\times\BB G)\hookrightarrow\eta$ extends to all of $D^k\times\BB G$.
The cokernel of this embedding is thus an extension of $\nu$ to $D^k\times\BB G$, so we can perform a handle attachment to construct our desired compact effective orbifold-with-boundary.
Finally, we should note that in the case $k=0$, we should not take $\BB G$ as this is not effective, rather we can take the unit ball in any faithful $G$-representation modulo $G$.
\end{proof}

We will need an analogue of the the previous corollary for orbi-CW-pairs.
Let us define an \emph{orbifold pair} $(X,A)$ to consist of an orbifold-with-boundary $X$ and a codimension zero suborbifold-with-boundary $A\subseteq\partial X$.
In this paper, we only ever deal with compact orbifold pairs.

\begin{corollary}\label{pairispair}
Every finite orbi-CW-pair is homotopy equivalent to a compact effective orbifold pair.
\end{corollary}

\begin{proof}
Corollary \ref{orbicwisorbifold} implies the result for finite orbi-CW-pairs of the form $(X,X)$ (namely realize $X$ as a compact orbifold-with-boundary $Z$ and take $(Z\times[0,1],Z)$).
Now given a finite orbi-CW-pair $(X,A)$, we begin with an orbifold pair homotopy equivalent to $(A,A)$, and we successively attach handles (away from the marked part of the boundary) as in the proof of Corollary \ref{orbicwisorbifold} to make it homotopy equivalent to $(X,A)$.
\end{proof}

\subsection{Stable homotopy categories of orbispaces}

We now describe how to ``stabilize'' the categories $\OrbSpc_*^f$ and $\RepOrbSpc_*^f$ to obtain categories of finite orbispectra.
The categories of finite `naive orbispectra' are defined by taking the direct limit of $\OrbSpc_*^f$ and $\RepOrbSpc_*^f$ under successive applications of the suspension operation $\Sigma$ (namely $\times(I,\partial I)$).
The formal desuspension $(X,A)^{-V}$ for any coarse vector bundle $V$ makes sense as a naive orbispectrum, since any such $V$ embeds into a trivial vector bundle.
We are more interested in the categories $\OrbSp^f$ and $\RepOrbSp^f$ of finite `genuine orbispectra', whose objects take the form $(X,A)^{-V}$ for any vector bundles $V$, with morphisms being a suitable direct limit over passing to Thom spaces of arbitrary vector bundles.

We first discuss naive orbispectra.
The suspension operation $\times(I,\partial I)$ defines an endofunctor $\Sigma$ of both $\RepOrbSpc_*$ and $\OrbSpc_*$.
The direct limit of successive applications of this endofunctor defines stable homotopy categories
\begin{align}
\RepOrbSpc_*[\Sigma^{-1}]:={}&\varinjlim\bigl(\RepOrbSpc_*\xrightarrow\Sigma\RepOrbSpc_*\xrightarrow\Sigma\cdots\bigr),\\
\OrbSpc_*[\Sigma^{-1}]:={}&\varinjlim\bigl(\OrbSpc_*\xrightarrow\Sigma\OrbSpc_*\xrightarrow\Sigma\cdots\bigr).
\end{align}
Concretely, the objects of both these categories are formal symbols $\Sigma^{-n}(X,A)$ for orbi-CW-pairs $(X,A)$ and integers $n\geq 0$, and the set of morphisms $\Sigma^{-n}(X,A)\to\Sigma^{-m}(Y,B)$ is the direct limit over $k\to\infty$ of morphisms in $\RepOrbSpc_*$ and $\OrbSpc_*$, respectively, from $\Sigma^{k-n}(X,A)$ to $\Sigma^{k-m}(Y,B)$ (which makes sense for $k\geq\max(m,n)$).
It is immediate that $\Sigma$ defines autoequivalences of $\RepOrbSpc_*[\Sigma^{-1}]$ and $\OrbSpc_*[\Sigma^{-1}]$, and that there is a natural isomorphism $\Sigma^{-1}((X,A)\times(I,\partial I))=(X,A)$ in both these categories.
There is a functor $\RepOrbSpc_*[\Sigma^{-1}]\to\OrbSpc_*[\Sigma^{-1}]$.
We can make sense out of symbols $(X,A)^{-V}$ ($(X,A)$ a compact orbi-CW-pair and $V$ a coarse vector bundle over $X$) as objects of $\RepOrbSpc_*[\Sigma^{-1}]$ and $\OrbSpc_*[\Sigma^{-1}]$, namely by embedding $V\hookrightarrow\underline\RR^n$ and taking $(X,A)^{-V}:=\Sigma^{-n}((X,A)^{\underline\RR^n/V})$ (which is independent of the choice of embedding $V\hookrightarrow\underline\RR^n$ up to canonical isomorphism).

The category $\Spc_*[\Sigma^{-1}]:=\varinjlim(\Spc_*\xrightarrow\Sigma\Spc_*\xrightarrow\Sigma\cdots)$ lies as a full subcategory inside both $\RepOrbSpc_*[\Sigma^{-1}]$ and $\OrbSpc_*[\Sigma^{-1}]$.

\begin{lemma}\label{additive}
The categories $\OrbSpc_*[\Sigma^{-1}]$ and $\RepOrbSpc_*[\Sigma^{-1}]$ are additive.
\end{lemma}

\begin{proof}
There is a natural abelian group structure on the morphism space in $\OrbSpc_*$ and $\RepOrbSpc_*$ from $(X,A)\times(I^2,\partial I^2)$ to $(Y,B)$.
This gives an enrichment of $\OrbSpc_*[\Sigma^{-1}]$ and $\RepOrbSpc_*[\Sigma^{-1}]$ over abelian groups.
It is immediate that finite disjoint unions are finite coproducts.
A category enriched over abelian groups and which has finite coproducts is additive.
\end{proof}

The categories $\OrbSpc_*[\Sigma^{-1}]$ and $\RepOrbSpc_*[\Sigma^{-1}]$ may be described alternatively as localizations as follows.
We consider the `Grothendieck construction'
\begin{equation}\label{grothconstr}
\operatorname{Groth}\bigl(\OrbSpc_*\xrightarrow\Sigma\OrbSpc_*\xrightarrow\Sigma\cdots\bigr),
\end{equation}
namely the category whose objects are formal symbols $\Sigma^{-n}(X,A)$ and whose morphisms $\Sigma^{-n}(X,A)\to\Sigma^{-m}(Y,B)$ are morphisms $(X,A)\to\Sigma^{n-m}(Y,B)$ in $\OrbSpc_*$ for $n\geq m$ (and there are no morphisms otherwise).
There is a class $A$ of morphisms $\Sigma^{-n}(\Sigma^{n-m}(X,A))\to\Sigma^{-m}(X,A)$ corresponding to the identity map of $\Sigma^{n-m}(X,A)$ for $n\geq m$.
It is immediate that this class $A$ forms a right multiplicative system, a notion whose definition we now recall.

\begin{definition}[Right multiplicative system]
A class of morphisms $W$ in a category $\C$ is called a \emph{right multiplicative system} iff it satisfies the following three axioms:
\begin{itemize}
\item $W$ contains all identities and is closed under composition.
\item(Right Ore condition) For every pair of solid arrows
\begin{equation}\label{rore}
\begin{tikzcd}
A\ar[r,dashed]\ar[d,dashed,swap,"\in W"]&B\ar[d,"\in W"]\\
C\ar[r]&D
\end{tikzcd}
\end{equation}
there exist an object $A$ and dotted arrows so that the diagram commutes.
\item(Right cancellability) For every commuting diagram of solid arrows
\begin{equation}\label{rcancel}
\begin{tikzcd}
A\ar[r,dashed]\ar[d,dashed,swap,"\in W"]&B\ar[d,"\in W"]\\
C\ar[ru,shift left]\ar[ru,shift right]\ar[r]&D
\end{tikzcd}
\end{equation}
there exists an object $A$ and dotted arrows so that the diagram commutes.
\end{itemize}
Note that $W$ is \emph{not} required to contain all isomorphisms.
This is rather antithetical to the philosophy of category theory, however this generality is significant for us.
\end{definition}

For any right multiplicative system $W$ in a category $\C$, the localization $\C\to\C[W^{-1}]$ exists, provided a certain smallness condition is satisfied.
Furthermore, for $X,Y\in\C$, the set of morphisms $X\to Y$ in $\C[W^{-1}]$ admits the following concrete description.
Given $X\in\C$, consider the category $\{Z\xrightarrow WX\}$ whose objects are arrows $Z\xrightarrow WX$ and whose morphisms are morphisms over $X$ (i.e.\ $\{Z\xrightarrow WX\}$ is a full subcategory of the over-category $\C_{/X}$).
The fact that $W$ is a right multiplicative system implies that $\{Z\xrightarrow WX\}$ is filtered.
Provided each category $\{Z\xrightarrow WX\}$ is essentially small, the localization $\C[W^{-1}]$ exists and the set of morphisms $X\to Y$ in $\C[W^{-1}]$ is the direct limit over $\{Z\xrightarrow WX\}$ of the set of morphisms $Z\to Y$.

The functor from the Grothendieck construction \eqref{grothconstr} to $\OrbSpc_*[\Sigma^{-1}]$ sends $A$ to isomorphisms, hence factors uniquely through the localization.
Using the explicit description of morphisms in the localization by a right multiplicative system, it is immediate that this functor is an equivalence.

We now define the categories of finite \emph{genuine orbispectra} $\OrbSp^f$ and $\RepOrbSp^f$.
Let us begin with the categories $\OrbSpcPair^{f,-\Vect}$ and $\RepOrbSpcPair^{f,-\Vect}$, whose objects are $(X,A)^{-\xi}$ where $(X,A)$ is a finite orbi-CW-pair and $\xi$ is a vector bundle over $X$.
A morphism in these categories $(X,A)^{-\xi}\to(Y,B)^{-\zeta}$ consists of a (representable) map $f:X\to Y$, an embedding $i:f^*\zeta\hookrightarrow\xi$, and a section $s:X\to\xi/i(f^*\zeta)$ such that $A\subseteq f^{-1}(B)\cup s^{-1}(\{\left|\cdot\right|\geq\varepsilon\})$ for some $\varepsilon>0$ (these triples $(f,i,s):(X,A)^{-\xi}\to(Y,B)^{-\zeta}$ are considered up to homotopy, namely the equivalence relation of there being such a morphism $(X\times[0,1],A\times[0,1])^{-\xi}\to(Y,B)^{-\zeta}$).
To compose $(X,A)^{-\xi}\xrightarrow{(f,i,s)}(Y,B)^{-\zeta}\xrightarrow{(g,j,t)}(Z,C)^{-\eta}$ we take $g\circ f$ and $j\circ i$; now there is an extension
\begin{equation}
0\to\zeta/i(\xi)\xrightarrow j\eta/j(i(\xi))\to\eta/j(\zeta)\to 0,
\end{equation}
and choosing a splitting (which is unique up to homotopy) allows us to take $tf\oplus js$ as our section for the composition.
Composition is associative.

We now define categories $\OrbSpc_*^{f,-\Vect}$ and $\RepOrbSpc_*^{f,-\Vect}$ by modifying the definition above by declaring that a morphism $(X,A)^{-\xi}\dashrightarrow(Y,B)^{-\zeta}$ consists of $A\subseteq A^+\subseteq X$ closed and $U\subseteq X$ open with $X=U\cup(A^+)^\circ$ and a morphism $(U,U\cap A^+)^{-\xi}\to(Y,B)^{-\zeta}$ in $\OrbSpcPair^{f,-\Vect}$ or $\RepOrbSpcPair^{f,-\Vect}$; these are considered modulo homotopy as usual.
Composition of
\begin{equation}
\begin{tikzcd}[column sep = large]
(X,A)^{-\xi}\ar[r,dashed,"{(A^+,U,f,i,s)}"]&(Y,B)^{-\zeta}\ar[r,dashed,"{(B^+,V,g,j,t)}"]&(Z,C)^{-\eta}
\end{tikzcd}
\end{equation}
is given by $(A^+\cup f^{-1}(B^+),f^{-1}(V),g\circ f,j\circ i,tf\oplus js)$ as before.
The proof of Proposition \ref{excision} applies without modification to show that:

\begin{proposition}[Excision]\label{excisionspectra}
The functors
\begin{align}
\OrbSpcPair^{f,-\Vect}&{}\to\OrbSpc_*^{f,-\Vect}\\
\RepOrbSpcPair^{f,-\Vect}&{}\to\RepOrbSpc_*^{f,-\Vect}
\end{align}
are localizations at the collection $W$ of morphisms of the form $(P,P\cap Q)^{-\xi}\to(X,Q)^{-\xi}$ where $X=P\cup Q$ is a cover by subcomplexes and $\xi$ is a vector bundle over $X$.
\qed
\end{proposition}

We now localize $\RepOrbSpc_*^{f,-\Vect}$ and $\OrbSpc_*^{f,-\Vect}$ at the class of morphisms $S$ given by \emph{the images of} the isomorphism classes in $\RepOrbSpcPair^{f,-\Vect}$ of the tautological morphisms $((X,A)^\xi)^{-(\xi\oplus\zeta)}\to(X,A)^{-\zeta}$.
The following deserves emphasis: the objects of $\RepOrbSpc_*^{f,-\Vect}$ and $\OrbSpc_*^{f,-\Vect}$ remain symbols $(X,A)^{-\xi}$, and while two different symbols $(X,A)^{-\xi}$ and $(X',A')^{-\xi'}$ may be \emph{isomorphic} in the localizations $\RepOrbSpc_*^{f,-\Vect}$ or $\OrbSpc_*^{f,-\Vect}$, they need not be isomorphic in $\RepOrbSpcPair^{f,-\Vect}$, and hence are regarded as \emph{completely different} when it comes to the question of whether a morphism is or is not in $S$.

\begin{lemma}\label{rms}
The morphisms $S$ form a right multiplicative system in $\RepOrbSpc_*^{f,-\Vect}$ and $\OrbSpc_*^{f,-\Vect}$.
\end{lemma}

\begin{proof}
Closedness under composition holds because any vector bundle on the total space of a vector bundle is pulled back from the base, so a Thom space of a Thom space is a Thom space.
(Note that isomorphisms in $\RepOrbSpc_*^{f,-\Vect}$ and $\OrbSpc_*^{f,-\Vect}$ need not be in $S$, and that we \emph{have not} shown that $S$ is not closed under composing with such isomorphisms!)

We verify the right Ore condition.
In the case that the morphism $C\to D$ in $\OrbSpc_*^{f,-\Vect}$ comes from $\OrbSpcPair^{f,-\Vect}$ (resp.\ with `Rep' prefixes), we can simply pull back the bundle involved in $B\to D$:
\begin{equation}\label{firstSpullback}
\begin{tikzcd}
((X,A)^{f^*\eta})^{-(\xi\oplus f^*\eta)}\ar[r,dashed]\ar[d,dashed,swap,"\in S"]&((Y,B)^\eta)^{-(\zeta\oplus\eta)}\ar[d,"\in S"]\\
(X,A)^{-\xi}\ar[r]&(Y,B)^{-\zeta}
\end{tikzcd}
\end{equation}
In the general case, the bottom row becomes $(X,A)^{-\xi}\to(X,A^+)^{-\xi}\leftarrow(U,U\cap A^+)^{-\xi}\to(Y,B)^{-\zeta}$.
Now we may pull back $\eta$ to $U$, but not to $X$.
Instead, we appeal to Theorem \ref{enough} (enough vector bundles) to embed the pullback of $\eta$ to $U$ into the restriction of a vector bundle $\tau$ on $X$.
We thus obtain the following diagram:
\begin{equation}\label{secondSpullback}
\begin{tikzcd}[column sep = tiny, row sep = small]
((X,A)^\tau)^{-(\xi\oplus\tau)}\ar[r,dashed]\ar[dd,dashed,swap,"\in S"]&((X,A^+)^\tau)^{-(\xi\oplus\tau)}\ar[dd,dashed,swap,"\in S"]&\ar[l,dashed,swap,"\in W"]((U,U\cap A^+)^\tau)^{-(\xi\oplus\tau)}\ar[d,dashed,"\in S"]\\
&&((U,U\cap A^+)^{f^*\eta})^{-(\xi\oplus f^*\eta)}\ar[r,dashed]\ar[d,dashed,"\in S"]&((Y,B)^\eta)^{-(\zeta\oplus\eta)}\ar[d,"\in S"]\\
(X,A)^{-\xi}\ar[r]&(X,A^+)^{-\xi}&\ar[l,swap,"\in W"](U,U\cap A^+)^{-\xi}\ar[r]&(Y,B)^{-\zeta}
\end{tikzcd}
\end{equation}
The desired result follows.

We verify right cancellability.
It suffices to show that given maps $C\to B\xrightarrow{\in S}D$, applying the pullback procedure above to $C\to D\xleftarrow{\in S}B$ results in $C\xleftarrow{\in S}A\to B$ for which $A\to B$ and $A\to C\to B$ coincide (in this way, the dotted arrows in \eqref{rcancel} that we produce depend only on the maps $C\to D\xleftarrow{\in S}B$).
As above, we first consider the situation of a morphism $C\to D$ coming from $\OrbSpcPair^{f,-\Vect}$ (resp.\ $\RepOrbSpcPair^{f,-\Vect}$), i.e.\ we consider \eqref{firstSpullback}.
Even in this setting, the desired commutativity is not obvious and requires the following calculation.
We reproduce the relevant diagram, rewriting it in a more convenient way:
\begin{equation}
\begin{tikzcd}
((X,A)^{f^*\eta})^{-(f^*\zeta\oplus f^*\eta\oplus\xi\oplus f^*\eta)}\ar[r]\ar[d]&((Y,B)^\eta)^{-(\zeta\oplus\eta)}\ar[d]\\
(X,A)^{-(f^*\zeta\oplus f^*\eta\oplus\xi)}\ar[r]\ar[ru]&(Y,B)^{-\zeta}
\end{tikzcd}
\end{equation}
Let the diagonal map $(X,A)^{-\xi}\to((Y,B)^\eta)^{-(\zeta\oplus\eta)}$ be given by $f:X\to Y$, $g:X\to f^*\eta$, the obvious inclusion $f^*(\zeta\oplus\eta)\hookrightarrow f^*\zeta\oplus f^*\eta\oplus\xi$, and $s:X\to\xi$.
Now we have two maps $((X,A)^{f^*\eta})^{-(f^*\zeta\oplus f^*\eta\oplus\xi\oplus f^*\eta)}\to((Y,B)^\eta)^{-(\zeta\oplus\eta)}$ which we would like to show are homotopic.
The top horizontal arrow is given by $f^\eta:X^{f^*\eta}\to Y^\eta$, the inclusion $f^*\zeta\oplus f^*\eta\hookrightarrow f^*\zeta\oplus f^*\eta\oplus\xi\oplus f^*\eta$ where $f^*\eta$ goes to the last copy, and the section $X^{f^*\eta}\to X\xrightarrow{s\oplus g}\xi\oplus f^*\eta$.
The composition of the left vertical arrow and the diagonal arrow is given by $(f\circ\pi_X,g):X^{f^*\eta}\to Y^\eta$, the inclusion $f^*\zeta\oplus f^*\eta\hookrightarrow f^*\zeta\oplus f^*\eta\oplus\xi\oplus f^*\eta$ where $f^*\eta$ goes to the first copy, and $s\oplus\pi_{f^*\eta}:X^{f^*\eta}\to\xi\oplus f^*\eta$.
These maps are evidently not the same, but they are homotopic as follows.
We first apply the obvious linear homotopy from one inclusion $f^*\zeta\oplus f^*\eta\hookrightarrow f^*\zeta\oplus f^*\eta\oplus\xi\oplus f^*\eta$ to the other, noting that the induced action on the cokernel, naturally identified in both cases with $f^*\eta$, is multiplication by $-1$.
Now our two maps coincide except that we need to transform $(g,\id_{f^*\eta})$ into $(\id_{f^*\eta},-g)$ (note the sign picked up from the first homotopy), which we can do using rotation matrices for $\theta\in[0,\pi/2]$.

Finally, we should show right cancellability in the case of general maps $C\to D$; we use the same strategy as above.
Note that a pullback along $C\to D$ involves a choice of vector bundle $\xi$ over $C$ (let us call this a $\xi$-pullback), and that a choice of embedding $\xi\hookrightarrow\zeta$ determines a map from $\xi$-pullbacks to $\zeta$-pullbacks (`stabilization').
Now it is evident from the definition that given a $\xi$-pullback and a $\xi'$-pullback, there exist embeddings $\xi\hookrightarrow\zeta$ and $\xi'\hookrightarrow\zeta$ such that the induced $\zeta$-pullbacks coincide.
Note that this includes the assertion that pullbacks induced by homotopic maps are equivalent, which is shown by considering the pullback along the homotopy itself.
It now suffices to show the same commutativity as before, namely that $A\to C\to B$ and $A\to B$ agree.
To see this, note that the present situation is that of the solid arrows in \eqref{secondSpullback} plus a single additional diagonal arrow $(U,U\cap A^+)^{-\xi}\to((Y,B)^\eta)^{-(\zeta\oplus\eta)}$ in the lower rightmost square.
The resulting dotted arrows in that square commute by the reasoning in the previous paragraph, which combined with the commutativity of the rest of the diagram imply that everthing commutes.
\end{proof}

The right multiplicative system $S$ also satisfies the smallness condition needed to localize: the category of $S$-morphisms over $(X,A)^{-\xi}$ has as its objects all vector bundles over $X$, and the isomorphism classes of these form a set.

It follows from Lemma \ref{rms} and the smallness condition that the localizations of $\RepOrbSpc_*^{f,-\Vect}$ and $\OrbSpc_*^{f,-\Vect}$ at $S$ exist, and we denote these localizations by $\RepOrbSp^f$ and $\OrbSp^f$, respectively.
Morphisms $(X,A)^{-\xi}\to(Y,B)^{-\zeta}$ in $\RepOrbSp^f$ and $\OrbSp^f$ are thus described as the direct limit over vector bundles $\eta$ over $X$ of morphisms $((X,A)^\eta)^{-(\eta\oplus\xi)}\to(Y,B)^{-\zeta}$ in $\RepOrbSpc_*^{f,-\Vect}$ and $\OrbSpc_*^{f,-\Vect}$, respectively.
Now it makes sense to write $(X,A)^V$ for an object of $\RepOrbSp^f$ for any finite orbi-CW-pair $(X,A)$ and any stable vector bundle $V$ over $X$, since a stable isomorphism $E-F=E'-F'$ induces an isomorphism $((X,A)^E)^{-F}=((X,A)^{E'})^{-F'}$ in $\RepOrbSp^f$.

There are functors $\RepOrbSpc_*^f[\Sigma^{-1}]\to\RepOrbSp^f$ and $\OrbSpc_*^f[\Sigma^{-1}]\to\OrbSp^f$.
To construct them, note that $\Sigma^{-n}(X,A)\mapsto(X,A)^{-\underline\RR^n}$ defines functors out of the relevant the Grothendieck constructions which send $A$ to isomorphisms.

\begin{lemma}
The categories $\RepOrbSp^f$ and $\OrbSp^f$ are additive.
\end{lemma}

\begin{proof}
Same as Lemma \ref{additive}.
\end{proof}

For any covering space $(X',A')\to(X,A)$ (meaning $X'\to X$ is a covering space and $A'=A\times_XX'$), there is an induced map $(X,A)\to(X',A')$ in $\RepOrbSp^f$.
It is defined by embedding $X'$ into the total space of a vector bundle over $X$ (which exists by enough vector bundles Theorem \ref{enough}) and taking the usual collapse map.

There is a functor $\RepOrbSp^f\to\OrbSp^f$, and the category of finite spectra $\Sp^f$ is a full subcategory of both.

There is a classifying space functor $\OrbSp^f\to\Sp$ defined as follows.
We first define the corresponding functor $\OrbSpcPair^{f,-\Vect}\to\Sp$, which sends $(X,A)^{-V}$ to $(\tilde X,\tilde A)^{-V}\in\Sp$; concretely, $(\tilde X,\tilde A)^{-V}$ is the direct limit over finite subcomplexes $(\tilde X_0,\tilde A_0)$ and embeddings $V|_{\tilde X_0}\hookrightarrow\underline\RR^N$ of $\Sigma^{-N}((\tilde X_0,\tilde A_0)^{\underline\RR^N/V})$).
Given a map $(X,A)^{-V}\to(Y,B)^{-W}$ in $\OrbSpcPair^{f,-\Vect}$ consisting of a map $f:X\to Y$, an inclusion $f^*W\hookrightarrow V$, and a section $s:X\to V/f^*W$, we define a map $(\tilde X,\tilde A)^{-V}\to(\tilde Y,\tilde B)^{-W}$ in $\Sp$ as follows.
Choose classifying spaces $\tilde A\subseteq\tilde X$ and $\tilde B\subseteq\tilde Y$ and the map $\tilde f:\tilde X\to\tilde Y$ fitting into a strictly commutative diagram with $f$ such that $\tilde f(f^{-1}(B))\subseteq\tilde B$ (this can be done by induction on the cells of $(X,A)$ and $(Y,B)$).
Now the map $\tilde f:\tilde X\to\tilde Y$, the pullback inclusion $\tilde f^*W\hookrightarrow V$, and the pullback of $s$ define a map $(\tilde X,\tilde A)^{-V}\to(\tilde Y,\tilde B)^{-W}$ in $\SpcPair^{-\Vect}$, hence in $\Sp$; concretely, the induced map in $\Sp$ is given (over a given finite subcomplex of $(\tilde X,\tilde A)$) by taking $N$ large enough so that the map $f^*W\hookrightarrow V\hookrightarrow\underline\RR^N$ is homotopic to the pullback of $W\hookrightarrow\underline\RR^N$, so $f^*(\underline\RR^N/W)=(\underline\RR^N/V)\oplus V/f^*W$ so we may use the identity on the first factor and the section $s$ on the second factor to define a map $(\tilde X,\tilde A)^{\underline\RR^N/V}\to(\tilde Y,\tilde B)^{\underline\RR^N/W}$ (strictly speaking, only on arbitrary finite subcomplexes thereof) which we desuspend by $N$.
It is immediate to check that morphisms $W$ are sent to isomorphisms, and morphisms $S$ are also by inspection.
We therefore have a classifying space functor $\OrbSp^f\to\Sp$.

There is a symmetric monoidal `smash product' $\wedge$ on $\RepOrbSp^f$ and $\OrbSp^f$ defined as follows.
The product $(X,A)^{-\xi}\times(Y,B)^{-\zeta}:=(X\times Y,(A\times Y)\cup(X\times B))^{-\xi-\zeta}$ is a symmetric monoidal structure on $\RepOrbSpcPair^{f,-\Vect}$ and $\OrbSpcPair^{f,-\Vect}$.
To see that it descends to $\RepOrbSp^f$ and $\OrbSp^f$, it suffices to show that a morphism in $W$ or $S$ times a fixed $(Y,B)^{-\zeta}$ again lies in $W$ or $S$.
For $S$ this is obvious, and for $W$ this follows by inspection exactly as in the construction of the smash product on $\RepOrbSpc_*^f$ and $\OrbSpc_*^f$.

There is a natural isomorphism $(Z\wedge W)^\sim=\tilde Z\wedge\tilde W$ for $Z,W\in\OrbSp^f$; to define this, it suffices to define it on the corresponding functors out of $\OrbSpcPair^{f,-\Vect}\times\OrbSpcPair^{f,-\Vect}$, where it is defined in the same way as the corresponding isomorphisms for $Z,W\in\OrbSpc_*$.

\subsection{Exact triangles}

A triple of morphisms in $\RepOrbSp^f$ is called a cofiber sequence iff it is isomorphic to
\begin{equation}
(Y,B)^{-\xi}\to(X,A)^{-\xi}\to(X,A\cup_BY)^{-\xi}
\end{equation}
for an orbi-CW-complex $X$ with two subcomplexes $A,Y\subseteq X$ and $B:=A\cap Y$ and a vector bundle $\xi$ over $X$ (compare \eqref{cofiber}).
We now show that every morphism in $\RepOrbSp^f$ has a cofiber, from which it follows that every morphism can be extended to a bi-infinite cofiber sequence by desuspending the Puppe sequence.

\begin{proposition}
Every morphism in $\RepOrbSp^f$ is isomorphic to one of the form $(Y,B)^{-\xi}\to(X,A)^{-\xi}$ for $X$ an orbi-CW-complex carrying a vector bundle $\xi$ and $A,Y\subseteq X$ subcomplexes with $B=A\cap Y$.
\end{proposition}

\begin{proof}
Since the localization $\RepOrbSpc_*^{f,-\Vect}\xrightarrow{S^{-1}}\RepOrbSp^f$ is by a right multiplicative system, every morphism in the target is isomorphic to one coming from the source.
In other words, every morphism in $\RepOrbSp^f$ is (up to isomorphism) a formal composition $(X,A)^{-\xi}\to(X,A^+)^{-\xi}\leftarrow(U,U\cap A^+)^{-\xi}\xrightarrow f(Y,B)^{-\zeta}$.
As in the proof of Proposition \ref{makecofibration}, we may assume that $U,A^+\subseteq X$ are subcomplexes covering $X$.
Form the gluing $X\cup_U(U\times[0,1])\cup_UY$, and find, using enough vector bundles Theorem \ref{enough}, a vector bundle $\tau$ over it together with embeddings $\xi\hookrightarrow\tau|_X$ and $\zeta\hookrightarrow\tau|_Y$.
By replacing $\tau$ with $\tau\oplus\tau$, we may ensure that the composition $f^*\zeta\hookrightarrow\xi|_U\hookrightarrow\tau|_U$ is homotopic to the pullback under $f$ of the embedding $\zeta\hookrightarrow\tau|_Y$.
We thus obtain a commutative diagram
\begin{equation}
\begin{tikzcd}[column sep = tiny]
((X,A)^{\tau|_X/\xi})^{-\tau|_X}\ar[d,"\in S"]\ar[r]&((X,A^+)^{\tau|_X/\xi})^{-\tau|_X}\ar[d,"\in S"]&\ar[l]\ar[d,"\in S"]\ar[r]((U,U\cap A^+)^{\tau|_U/\xi|_U})^{-\tau|_U}\ar[r]&((Y,B)^{\tau|_Y/\zeta})^{-\tau|_Y}\ar[d,"\in S"]\\
(X,A)^{-\xi}\ar[r]&(X,A^+)^{-\xi}&\ar[l](U,U\cap A^+)^{-\xi}\ar[r]&(Y,B)^{-\zeta}
\end{tikzcd}
\end{equation}
Now the top row is just the desuspension by $\tau$ of maps $(X,A)^{\tau|_X/\xi}\to(X,A^+)^{\tau|_X/\xi}\leftarrow(U,U\cap A^+)^{\tau|_U/\xi|_U}\to(Y,B)^{\tau|_Y/\zeta}$ all of which respect the vector bundle $\tau$ being desuspended by.
Now take this as \eqref{tobemadecofibration} and apply the construction of that proof to it, and then desuspend by $\tau$ (which we crucially must note does indeed make sense on the result).
\end{proof}

\subsection{Stabilizing over \texorpdfstring{$R(*)$}{R(*)}}

Let us now give an alternative definition of $\RepOrbSp^f$ (but not of $\OrbSp^f$).
Morally, we would like to simply say that
\begin{equation}\label{stabilizewrong}
\RepOrbSp^f=\varinjlim_{\Vect(R(*))}\RepOrbSpc,
\end{equation}
in the sense that every orbi-CW-complex $X$ admits a unique up to contractible choice representable map to $R(*)$, so vector bundles on $R(*)$ act by endofunctors on $\RepOrbSpc$ by pulling back and passing to Thom spaces.
There is a problem with taking \eqref{stabilizewrong} literally: there are not enough vector bundles on $R(*)$, so instead we will filter $R(*)$ by subcomplexes.
Here are the details.

For $N\geq 1$, let $R(*)_N$ denote the image of $*\in\OrbSpc_N$ under the right adjoint to $\RepOrbSpc_N\to\OrbSpc_N$ (which exists by the same argument as in Proposition \ref{proprightadjoint}) where the subscript $N$ indicates restricting to orbi-CW-complexes with isotropy groups of order $\leq N$.
There are representable maps $R(*)_N\to R(*)_M$ for $N\leq M$ by abstract nonsense (the functor $\RepOrbSpc_N\to\RepOrbSpc_M$ induces a map between their terminal objects), and the infinite mapping cylinder of $R(*)_1\to R(*)_2\to R(*)_3\to\cdots$ is $R(*)$.
Concretely, $R(*)_N$ is given by \eqref{rpoint} restricted to groups of order $\leq N$.

Fix orbi-CW-complexes $R(*)_N$ and cellular maps $R(*)_N\to R(*)_{N+1}$, and let $R(*)$ denote their infinite mapping cylinder.
Let $R(*)_{N,k}$ denote the $k$-skeleton of $R(*)_N$.
Note that, whereas $R(*)_N$ has an intrinsic functorial description, $R(*)_{N,k}$ does not: it depends on the chosen orbi-CW-complex realization of $R(*)_N$.
Concretely, we may (but are not obliged to) take $R(*)_{N,k}$ to be the subcomplex of \eqref{rpoint} spanned by groups of order $\leq N$ and simplices of dimension $\leq k$.
Now every map $\partial D^r\times\BB G\to R(*)_{N,k}$ extends to $D^r\times\BB G$ provided $r\leq k$ and $\left|G\right|\leq N$.

Now suppose $X$ is an orbi-CW-complex of dimension $\leq k$ and with isotropy groups of order $\leq N$.
Then there exists a representable map $X\to R(*)_{N,k+2}$, any two such maps are homotopic, and any two homotopies are homotopic rel endpoints.
In particular, for every vector bundle $\xi$ over $R(*)_{N,k+2}$, we obtain a vector bundle $\xi_X$ which is well-defined up to unique homotopy class of isomorphism.
Moreover, for any representable map $X\to Y$, the pullback of $\xi_Y$ is isomorphic to $\xi_X$ by an isomorphism which is well-defined up to homotopy, and this rule is compatible with composition $X\to Y\to Z$.

Now let $\RepOrbSpcPair_{N,k}\subseteq\RepOrbSpcPair$ denote the full subcategory spanned by those orbi-CW-pairs $(X,A)$ for which $X$ (though not necessarily $A$) is homotopy equivalent to an orbi-CW-complex of dimension $\leq k$ with isotropy groups of size $\leq N$.
Given any vector bundle $\xi$ over $R(*)_{N,k+2}$, suspension by the pullback of $\xi$ defines a functor from $\RepOrbSpcPair_{N,k}$ to itself.
Let $\Vect^2(R(*)_{N,k})$ denote the $2$-category whose objects are vector bundles over $R(*)_{N,k}$, whose morphisms are inclusions $V\hookrightarrow V'$, and whose $2$-morphisms are homotopy classes of paths in $\Emb(V,V')$.
We may consider the direct limit
\begin{equation}\label{Rpointdirlimsingle}
\varinjlim_{\Vect^2(R(*)_{N,k+2})}\RepOrbSpcPair_{N,k}
\end{equation}
where to an inclusion of vector bundles $\xi\hookrightarrow\xi'$ on $R(*)_{N,k+2}$, we associate the endofunctor of $\RepOrbSpcPair_{N,k}$ given by suspending by the pullback of $\xi'/\xi$.
We may also define \eqref{Rpointdirlimsingle} more concretely (without discussing direct limits of categories over filtered $2$-categories): its objects are triples $(X,A,\xi)$ where $(X,A)\in\RepOrbSpcPair_{N,k}$ and $\xi$ is a vector bundle on $R(*)_{N,k+2}$, and the set of morphisms $(X,A,\xi)\to(Y,B,\zeta)$ is the direct limit over $\eta\in\Vect(R(*)_{N,k+2})$ of the set of pairs of embeddings $\xi\hookrightarrow\eta\hookleftarrow\zeta$ and maps $(X,A)^{(\eta/\xi)_X}\to(Y,B)^{(\eta/\zeta)_Y}$, modulo simultaneous homotopy of the embeddings and the map.

Now we note that increasing $N$ and $k$ induces a full faithful inclusion of categories \eqref{Rpointdirlimsingle}, since restriction of vector bundles between these subcomplexes of $R(*)$ is cofinal by enough vector bundles Theorem \ref{enough} (or, rather, the stronger version \cite[Theorem 1.1]{orbibundle} which applies since $R(*)_{N,k+2}$ has bounded dimension and bounded isotropy groups).
We therefore obtain a category
\begin{equation}\label{Rpointdirlim}
\varinjlim_{N,k}\varinjlim_{\Vect^2(R(*)_{N,k+2})}\RepOrbSpcPair_{N,k}.
\end{equation}
Restricting to the full subcategory spanned by finite orbi-CW-pairs (i.e.\ replacing $\RepOrbSpcPair_{N,k}$ with $\RepOrbSpcPair_{N,k}^f$), we obtain a natural functor to $\RepOrbSp^f$, namely the map given by sending the object $(X,A)$ in the $\eta$-term of the direct limit to $(X,A)^{-\eta_X}$ (with the obvious action on morphisms); this is a map out of each copy of $\RepOrbSpcPair_{N,k}^f$, and descends by the natural coherences.

\begin{proposition}\label{reporbsecondlocalization}
The functor
\begin{equation}\label{reporbsplocaliz}
\varinjlim_{N,k}\varinjlim_{\Vect^2(R(*)_{N,k+2})}\RepOrbSpcPair^f_{N,k}\to\RepOrbSp^f
\end{equation}
is the localization at the morphisms $(P,P\cap Q)^{-\eta_P}\to(X,Q)^{-\eta_X}$ for $X=P\cup Q$ and $\eta$ a vector bundle on $R(*)_{N,k+2}$.
\end{proposition}

\begin{proof}
The morphisms $(P,P\cap Q)^{-\eta}\to(X,Q)^{-\eta}$ for $\eta$ any vector bundle on $X$ (not necessarily pulled back from $R(*)_{N,k+2}$) are sent to isomorphisms by Proposition \ref{excisionspectra}.

First, we note that in the definition of $\RepOrbSp^f$ as the double localization
\begin{equation}
\RepOrbSpcPair^{f,-\Vect}\xrightarrow{W^{-1}}\RepOrbSpc_*^{f,-\Vect}\xrightarrow{S^{-1}}\RepOrbSp^f,
\end{equation}
(i.e.\ first localizing at the class $W$ of morphisms $(P,P\cap Q)^{-\xi}\to(X,Q)^{-\xi}$ and then at the class $S$ of morphisms $((X,A)^V)^{-V-\xi}\to(X,A)^{-\xi}$), we could instead localize in the reverse order.
Indeed, given that the localization of $\RepOrbSpcPair^{f,-\Vect}$ at $W\sqcup S$ exists, it suffices to argue that its localization at $S$ exists.
In fact, $S$ forms a right multiplicative system in $\RepOrbSpcPair^{f,-\Vect}$ (this `easier' variant of Lemma \ref{rms} was the first step in its proof, in fact).
Thus the localization $\RepOrbSpcPair^{f,-\Vect}[S^{-1}]$ exists, and morphisms $(X,A)^{-\xi}\to(Y,B)^{-\zeta}$ in it are the direct limit over $\eta/X$ of morphisms $((X,A)^\eta)^{-\eta-\xi}\to(Y,B)^{-\zeta}$ in $\RepOrbSpcPair^{f,-\Vect}$.

Now it suffices to show that there is a natural equivalence
\begin{equation}\label{reporbspbeforelocaliz}
\varinjlim_{N,k}\varinjlim_{\Vect^2(R(*)_{N,k+2})}\RepOrbSpcPair^f_{N,k}\to\RepOrbSpcPair^{f,-\Vect}[S^{-1}].
\end{equation}
First, let us describe the functor: the copy of $\RepOrbSpcPair_{N,k}^f$ over $\eta\in\Vect^2(R(*))_{N,k+2}$ maps to $\RepOrbSpcPair^{f,-\Vect}[S^{-1}]$ as $(X,A)\mapsto(X,A)^{-\eta_X}$.
This functor is obviously essentially surjective, since $(X,A)^{-\xi}$ in the target is isomorphic to $((X,A)^{\eta_X/\xi})^{-\eta_X}$ for any embedding $\xi\hookrightarrow\eta_X$.
It thus remains to show that \eqref{reporbspbeforelocaliz} is fully faithful.

In both the source and target of \eqref{reporbspbeforelocaliz}, the morphisms sets are expressed as direct limits.
The set of morphisms $(X,A)^{-\eta_X}\to(Y,B)^{-\eta_Y}$ in the domain of \eqref{reporbspbeforelocaliz} (where $(X,A),(Y,B)\in\RepOrbSpcPair_{N,k}^f$ and $\eta\in\Vect^2(R(*)_{N,k+2})$) is the direct limit over $\tau\in\Vect(R(*)_{N,k+2})$ of the set of representable morphisms $(X,A)^{\tau_X}\to(Y,B)^{\tau_Y}$.
The set of morphisms between their images under \eqref{reporbspbeforelocaliz} is the direct limit over vector bundles $E$ over $X$ of tuples consisting of a representable map $f$ from the total space of $E$ to $Y$, an embedding $f^*\eta_Y\hookrightarrow\eta_X\oplus E$, and a section of $(\eta_X\oplus E)/i(f^*\eta_Y)$ such that the `relative part' of $(X,A)^E$ is contained in $f^{-1}(B)\cup s^{-1}(\{\left|\cdot\right|\geq\varepsilon\})$.
Since $f$ is representable, there is a natural identification $f^*\eta_Y=\eta_X$, which gives a canonical choice of embedding $f^*\eta_Y\hookrightarrow\eta_X\oplus E$ which need not coincide with the embedding which is chosen as part of the data.
However for the purposes of calculating the direct limit over $E$, we may assume that the chosen embedding $f^*\eta_Y\hookrightarrow\eta_X\oplus E$ is the canonical one: indeed, given any such embedding, passing to an appropriate $E'\supseteq E$ makes it homotopic to the canonical one, and similarly any homotopy from the canonical embedding to itself can be made null homotopic rel endpoints by enlarging $E$.
Hence the set of morphisms in the target of \eqref{reporbspbeforelocaliz} is the direct limit over vector bundles $E$ over $X$ of tuples consisting of a representable map $f$ from the total space of $E$ to $Y$ and a section of $E$ such that the relative part of $(X,A)^E$ is contained in $f^{-1}(B)\cup s^{-1}(\{\left|\cdot\right|\geq\varepsilon\})$.
By cofinality, we may instead declare that $E=\tau_X$ and take the direct limit over $\tau\in\Vect(R(*)_{N,k+2})$ of the set of maps $(X,A)^{\tau_X}\to(Y,B)^{\tau_Y}$ in $\RepOrbSpcPair_{N,k}^f$ (since $\tau_X=f^*\tau_Y$ canonically for any representable map $f$).
\end{proof}

Given the description of $\RepOrbSp^f$ as a direct limit over suspension by vector bundles pulled back from $R(*)$, it is natural to make the following conjecture parallel to Conjecture \ref{spacesoverR}.

\begin{conjecture}\label{spectraoverR}
The category $\RepOrbSp^f$ is a generating full subcategory of the category of parameterized spectra over $R(*)$.
\end{conjecture}

If Conjecture \ref{spectraoverR} is valid, it is natural then to ask whether fiberwise Spanier--Whitehead duality makes sense and, if so, how it is related to the duality involution from Theorem \ref{duality} which we prove immediately below.

\subsection{Duality}

We now define the contravariant involution $D:\RepOrbSp^f\to(\RepOrbSp^f)^\op$ as stated in Theorem \ref{duality}.

\begin{proof}[Proof of Theorem \ref{duality}]
To begin, we define $D:\RepOrbSpcPair^f\to(\RepOrbSp^f)^\op$.
By Corollary \ref{pairispair}, we may regard $\RepOrbSpcPair^f$ as the category of compact orbifold pairs (and morphisms thereof).
For any compact orbifold pair $(X,A)$, we set
\begin{equation}
D(X,A):=(X,\partial X-A^\circ)^{-TX}.
\end{equation}
The functoriality of $D$ under maps of orbifold pairs $(X,A)\to(Y,B)$ is defined as follows.
First, for any vector bundle $E$ over $Y$, denote by $(Y^E,B^E)$ the pair consisting of the total spaces of the unit disk bundles of $E$ over $Y$ and $B$.
Obviously $(Y^E,B^E)\to(Y,B)$ is an isomorphism in $\RepOrbSpcPair^f$, and there is also a natural identification $D(Y^E,B^E)=D(Y,B)$ (the effect on duals of passing from $(Y,B)$ to $(Y^E,B^E)$ is to suspend and desuspend by $E$).
Now, choose $E$ so that our map $X\to Y$ lifts to a smooth embedding $X\to Y^E$, meeting the boundary of $Y^E$ transversely precisely in $A$.
There is now an obvious collapse map $(Y^E,\partial Y^E-(B^E)^\circ)\dashrightarrow(X,\partial X-A^\circ)^{TY^E/TX}$ independent up to homotopy of the choice of lift of $X\to Y$ to $X\to Y^E$, and this collapse map is our desired map $D(Y,B)=D(Y^E,B^E)\to D(X,A)$.
By embedding any two choices of $E$ into a third, we see that the map $D(Y,B)\to D(X,A)$ thus defined is independent of the choice of $E$.
Making the same construction in a family over $[0,1]$ shows that it is invariant under homotopy.
One also checks that this recipe is compatible with composition, hence defines a functor $D:\RepOrbSpcPair^f\to(\RepOrbSp^f)^\op$.

To descend this functor to $D:\RepOrbSpc_*^f\to(\RepOrbSp^f)^\op$, it suffices, by Proposition \ref{excision}, to check that $D$ sends certain maps to isomorphisms.
Specifically, let $(X,A)$ be a compact orbifold pair, let $P$ be a compact orbifold-with-boundary, and let $\partial A\hookleftarrow Q\hookrightarrow\partial P$ be an identification between compact codimension zero suborbifolds-with-boundary of $\partial A$ and $\partial P$.
We may form $(X,A)\#_Q(P\times[0,1],P)$ and consider the inclusion $(X,A)\hookrightarrow(X,A)\#_Q(P\times[0,1],P)$.
These inclusions are precisely the morphisms inverted by the localization $\RepOrbSpcPair^f\to\RepOrbSpc_*^f$ from Proposition \ref{excision}.
Now, it is evident that the dual of the inclusion $(X,A)\hookrightarrow(X,A)\#_Q(P\times[0,1],P)$ is a map $(X,\partial X-A^\circ)^{-T}\leftarrow((X,\partial X-A^\circ)\#_Q(P\times[0,1],P))^{-T}$ (the superscript $^{-T}$ denotes desuspension by the tangent bundle) which is also an isomorphism in $\RepOrbSp^f$, so we are done.

We now define
\begin{equation}\label{Dbeforelocalization}
D:\varinjlim_{N,k}\varinjlim_{\Vect^2(R(*)_{N,k+2})}\RepOrbSpcPair_{N,k}^f\to(\RepOrbSp^f)^\op
\end{equation}
As above, we define $D((X,A)^{-\eta_X}):=(X,\partial X-A^\circ)^{\eta_X-TX}$ for $(X,A)$ a compact orbifold pair.
The functoriality under maps in $\RepOrbSpcPair_{N,k}^f$ is the same as before: the space of stable maps $(Y,\partial Y-B^\circ)^{\eta_Y-TY}\to(X,\partial X-A^\circ)^{\eta_X-TX}$ is the same as the space of stable maps $(Y,\partial Y-B^\circ)^{-TY}\to(X,\partial X-A^\circ)^{-TX}$, so we simply take the same map $D(Y,B)\to D(X,A)$ associated to our original map $(X,A)\to(Y,B)$.
This recipe is compatible with the morphisms in the direct limit over $\Vect^2(R(*)_{N,k+2})$ by inspection (and then obviously with the direct limit over $N$ and $k$), so we obtain the functor \eqref{Dbeforelocalization}.

By Proposition \ref{reporbsecondlocalization}, to descend $D$ from \eqref{Dbeforelocalization} to $\RepOrbSp^f$, it suffices to verify that it sends certain maps to isomorphisms, but these are exactly the same as we already saw above.
We therefore obtain the desired functor $D$.

By inspection, $D$ sends cofiber sequences to cofiber sequences.

There is an obvious identification $DDX=X$ for every $X\in\RepOrbSp^f$, directly from the definition.
To check that this defines a natural isomorphism of functors $D^2=\1$, we just need to show that it is compatible with morphisms.
Any morphism in $\RepOrbSp^f$ can be expressed as $(X,A)^{-\eta_X}\to(Y,B)^{-\eta_Y}$ for $(X,A),(Y,B)\in\RepOrbSpcPair_{N,k-2}$ compact orbifold pairs and $\eta\in\Vect^2(R(*)_{N,k})$ and some map $f^\eta:(X,A)\to(Y,B)$ which is a smooth embedding of orbifold pairs (meaning $A=X\cap B$ meeting the boundary of $Y$ transversely).
We then have a collapse map $(Y,\partial Y-B^\circ)\to(Y,Y-N_\varepsilon X)$ the target of which is relatively homotopy equivalent to $(X,\partial X-A^\circ)^{TY/TX}$; this collapse map is $(Df)^{TY-\xi}$.
Now we may realize this collapse map as an embedding
\begin{equation}
\begin{tikzcd}(Y,\partial Y-B^\circ)\ar[r,hook,"\times\{\frac 12\}"]&(Y\times[0,1],((\partial Y-B^\circ)\times[0,1])\cup((Y-N^\varepsilon X)\times\{1\})).\end{tikzcd}
\end{equation}
We may now dualize it again to obtain
\begin{align}
\Sigma(X,A)&{}=(N_\epsilon X\times[0,1],(N_\varepsilon X\times\partial[0,1])\cup(A\times[0,1]))\\
&=(Y\times[0,1],(B\times[0,1])\cup(Y\times\{0\})\cup(N_\varepsilon X\times\{1\}))\to\Sigma(Y,B),
\end{align}
which is indeed (the suspension of) the map we started with.
\end{proof}

Duality commutes with smash product: there are natural isomorphisms $D(Z\wedge W)=DZ\wedge DW$ for $Z,W\in\RepOrbSp^f$.
Indeed, it suffices to define such a natural isomorphism of functors of $Z$ and $W$ in the left side of \eqref{reporbsplocaliz} (where we may assume all objects are compact orbifold pairs), and such an isomorphism is evident by inspection.

There is a natural pairing $Z\wedge DZ\to R(*)$ defined as follows.
Let $Z=(X,A)^{-\xi}$ be an orbifold pair desuspended by a vector bundle.
The diagonal gives a map
\begin{equation}
(X,\partial X)\to(X,A)^{-\xi}\wedge(X,\partial X-A^\circ)^\xi.
\end{equation}
We desuspend to obtain $(X,\partial X)^{-TX}\to(X,A)^{-\xi}\wedge(X,\partial X-A^\circ)^{\xi-TX}$ and then dualize to obtain a map $Z\wedge DZ\to X$.
Composing with the canonical map $X\to R(*)$ defines the desired map $Z\wedge DZ\to R(*)$.

\section{Vector bundles}

\subsection{Classifying spaces}

For any compact Lie group $G$, let us argue that there is an object $\BB G\in\OrbSpc$ which classifies principal $G$-bundles, in the sense that it carries a principal $G$-bundle $\EE G\to\BB G$ such that the induced map from homotopy classes of maps $X\to\BB G$ to isomorphism classes of principal $G$-bundles over $X$ is a bijection for any orbi-CW-complex $X$.
(Note that by Lemma \ref{intervalpullback}, for any orbi-CW-complex $X$, necessarily paracompact, every principal $G$-bundle over $X\times[0,1]$ is pulled back from $X$, so there is indeed such a map.)
Note that $*/G$ has this representing property for all stacks, not just orbi-CW-complexes, however it is not itself an orbi-CW-complex unless $G$ is finite, so it is not (in the present context) $\BB G$.

\begin{lemma}\label{bgexist}
The classifying space $\BB G\in\OrbSpc$ exists.
\end{lemma}

\begin{proof}
We argue as in Proposition \ref{proprightadjoint}.
We construct, by induction, an orbi-CW-complex $\BB G$ carrying a faithful principal $G$-bundle $\EE G\to\BB G$.
Begin with $(\BB G)_{-1}=\varnothing$.
Consider triples consisting of a map $\partial D^k\times\BB\Gamma\to(\BB G)_{k-1}$, a faithful principal $G$-bundle $P$ over $D^k\times\BB\Gamma$, and an isomorphism over $\partial D^k\times\BB\Gamma$ between the restriction of $P$ and the pullback of $(\EE G)_{k-1}$.
Note that since $P\to D^k\times\BB\Gamma$ is faithful, the map $\partial D^k\times\BB\Gamma\to(\BB G)_{k-1}$ must necessarily be representable.
To define $(\BB G)_k$, we attach a cell to $(\BB G)_{k-1}$ for each homotopy class of such triple (we may omit trivial homotopy classes, i.e.\ those which are induced by maps $D^k\times\BB\Gamma\to(\BB G)_{k-1}$); the data of each triple tells us how to extend $(\EE G)_{k-1}\to(\BB G)_{k-1}$ to $(\EE G)_k\to(\BB G)_k$.

Now we claim that $\EE G\to\BB G$ is the desired universal principal $G$-bundle over an orbi-CW-complex.
It suffices to show that for every triple consisting of a map $\partial D^k\times\BB\Gamma\to\BB G$, a principal $G$-bundle $P$ over $D^k\times\BB\Gamma$, and an isomorphism over $\partial D^k\times\BB\Gamma$ between the restriction of $P$ and the pullback of $\EE G$, we can extend the map and the isomorphism to $D^k\times\BB\Gamma$.
By cellular approximation, we may assume the map $\partial D^k\times\BB\Gamma\to\BB G$ lands inside $(\BB G)_{k-1}$.
Now our principal $G$-bundle over $D^k\times\BB\Gamma$ is necessarily pulled back from $\BB\Gamma$ (since $D^k$ is contractible) hence is classified by a conjugacy class of homomorphisms $\Gamma\to G$.
In particular, it is pulled back from $D^k\times\BB(\Gamma/N)$ for $N\trianglelefteq\Gamma$ the kernel.
Since the principal $G$-bundle over $\BB G$ is faithful, this map also factors, uniquely, through $D^k\times\BB(\Gamma/N)$ by Lemma \ref{repfactor}.
Now that we have a representable map to $(\BB G)_{k-1}$, we can appeal to the definition of $(\BB G)_k$ to see that the triple involving $D^k\times\BB(\Gamma/N)$ extends as desired, hence by pre-composition the original triple as well.
\end{proof}

\begin{remark}
Another construction of $\BB G$ is given in \cite{orbibundle}.
There, the $G$-CW-complex $\EE G$ is defined by the property of carrying a $G$-action with finite stabilizers such that $(\EE G)^H$ is contractible for every finite subgroup $H\leq G$.
The orbispace $\BB G$ is then defined as the quotient $(\EE G)/G$.
\end{remark}

Note that since $\EE G\to\BB G$ is faithful, a principal $G$-bundle $P\to X$ is faithful iff the corresponding map $X\to\BB G$ is representable.

It is important to note that the extension property shown above in the proof that $\BB G$ represents the functor of \emph{isomorphism classes} of principal $G$-bundles is strictly stronger than the representing property (though of course \emph{a posteriori} it is equivalent).
The extension property corresponds to a more homotopical ($\infty$-categorical or model categorical) universal property of $\BB G$, and it will be used implicitly at later points (e.g.\ to know that every isomorphism of principal $G$-bundles is induced by a homotopy of maps to $\BB G$).

\begin{remark}
The object classifying principal $G$-bundles depends strongly on the category we are working in.
For example, the CW-complex $BG\in\Spc$ classifying principal $G$-bundles over CW-complexes evidently does not coincide with the orbi-CW-complex $\BB G\in\OrbSpc$ classifying principal $G$-bundles over orbi-CW-complexes.
Rather, it is immediate that the right adjoint to the inclusion $\Spc\hookrightarrow\OrbSpc$, namely the classifying space functor, sends $\BB G$ to $BG$ (and $R(\BB G)$ classifies principal $G$-bundles in the category $\RepOrbSpc$).
Similarly, if we were to define a larger category of `Lie orbispaces' allowing objects such as $*/G$, then the right adjoint (if it exists) to the inclusion of $\OrbSpc$ into this larger category would send $*/G$ to (what we have decided to call) $\BB G\in\OrbSpc$.
As a more explicit warning to the reader: the most natural meaning of the symbol $\BB G$ thus differs from context to context, and it should probably default to $\BB G:=*/G$ unless the contrary is explicitly stated, as we have done here.
\end{remark}

\subsection{Stable vector bundles}

We discuss stable vector bundles on orbi-CW-complexes.
The principal new feature in this discussion compared with the corresponding discussion for CW-complexes is that there are many different ways to `stabilize'.
We will consider only two extreme notions: `coarse stabilization', involving a direct limit over $\oplus\underline\RR$, or, equivalently, over $\oplus V$ for arbitrary \emph{coarse} vector bundles, and `stabilization', involving a direct limit over $\oplus V$ for arbitrary vector bundles $V$.

For an orbi-CW-complex $X$, let $\Vect(X)$ denote the category whose objects are vector bundles over $X$ and whose morphisms are homotopy classes of injective maps.
As a set, $\Vect(X)$ is the set of homotopy classes of maps $X\to\bigsqcup_{n\geq 0}\BB O(n)$.

\begin{lemma}
The category $\Vect(X)$ is filtered.
\qed
\end{lemma}

\begin{example}
A vector bundle over $\BB G$ is a $G$-representation.
Thus objects of $\Vect(\BB G)$ are in bijection with elements of $\ZZ_{\geq 0}^{\hat G}$ where $\hat G$ denotes the set of isomorphism classes of real irreducible representations of $G$.
An automorphism of an object of $\Vect(\BB G)$ also splits as a direct sum of isotypic pieces.
The component group of the space of automorphisms of $\rho^{\oplus n}$ for $n>0$ is $\ZZ/2$ if $\End(\rho)=\RR$ and is trivial otherwise (i.e.\ if $\End(\rho)=\CC$ or $\HH$).
\end{example}

For two vector bundles $V$ and $W$ on $X$, let $\pi_0\Iso(V,W)$ denote homotopy classes of isomorphisms $V\to W$.
Vector bundles and isomorphisms up to homotopy form a groupoid $\Vect(X)_\iso$.
A \emph{stable isomorphism} $V\dashrightarrow W$ up to homotopy is an element of
\begin{equation}\label{stiso}
\pi_0\Iso^\st(V,W):=\varinjlim_{E\in\Vect(X)}\pi_0\Iso(V\oplus E,W\oplus E).
\end{equation}
Vector bundles and stable isomorphisms also form a groupoid $\Vect(X)_\iso^\st$.
If we restrict \eqref{stiso} to coarse vector bundles $E$, we obtain the notion of a \emph{coarsely stable isomorphism} and a resulting groupoid $\Vect(X)^\cst_\iso$.
If $X$ is compact, then the sequence $0\hookrightarrow\underline\RR\hookrightarrow\underline\RR^2\hookrightarrow\cdots$ is cofinal in coarse vector bundles on $X$, so it is equivalent to stabilize just by these.
When stabilizing with respect to all vector bundles, there seems to be no such nice canonical sequence (though see \cite[Remark 1.4]{orbibundle}).
The notion of stable isomorphism is most reasonable when $X$ is compact (or at least has enough vector bundles).

The groupoid of vector bundles and stable isomorphisms may be extended to a larger groupoid of \emph{stable vector bundles} (similarly, the groupoid of vector bundles and coarsely stable isomorphisms extends to a groupoid of \emph{coarsely stable vector bundles}).
A (coarsely) stable vector bundle is a formal difference $E-F$ (where $F$ is coarse); if $X$ is compact a coarse vector bundle is equivalently a formal difference $E-\underline\RR^n$.
An isomorphism of (coarsely) stable vector bundles $(E-F)\to(E'-F')$ is a (coarsely) stable isomorphism $E\oplus F'\dashrightarrow E'\oplus F$; note that we can indeed compose these.
Provided $X$ is compact, the groupoid of coarsely stable vector bundles is the direct limit of $\Vect(X)\xrightarrow{\oplus\underline\RR}\Vect(X)\xrightarrow{\oplus\underline\RR}\cdots$.
The groupoid of stable vector bundles is the direct limit of $\Vect(X)$ over the $2$-categorical refinement $\Vect^2(X)$ of $\Vect(X)$ in which a morphism is an inclusion of vector bundles and a $2$-morphism is a homotopy class of paths of inclusions.

\begin{example}\label{stablevboverbg}
Isomorphism classes of stable vector bundles on $\BB G$ are in bijection with $\ZZ^{\hat G}$.
The automorphism group of every one is the product of $\ZZ/2$ over all $\rho\in\hat G$ with $\End(\rho)=\RR$.
Coarsely stable vector bundles on $\BB G$ are in bijection with $\ZZ\oplus\ZZ_{\geq 0}^{\hat G-\1}$, and the automorphism group of a coarsely stable vector bundle is $\ZZ/2$ (corresponding to $\rho=\1$) times the product of $\ZZ/2$ over all $\rho\ne\1$ for which $\End(\rho)=\RR$ and whose isotypic piece is nontrivial.
\end{example}

There is an orbi-CW-complex $\bO:=\varinjlim_n\BB O(n)$ defined as the infinite mapping cylinder of the maps $\BB O(n)\xrightarrow{\oplus\underline\RR}\BB O(n+1)$.
This orbispace $\bO$ classifies coarsely stable vector bundles: a map $X\to\bO$ up to homotopy is the same as a coarsely stable vector bundles of dimension zero over $X$ up to isomorphism.
The notation $\bO$ is chosen to coincide with the notation for a corresponding global space defined by Schwede \cite[\S 2.4]{schwedeglobal} which has the same classifying property; see \S\ref{globalclassifyingspaces} below.

One might desire an orbispace $\BO$ classifying stable vector bundles; intuitively, it should be the group completion of $\bigsqcup_n\BB O(n)$.
There is indeed a global space $\BO$ \cite[\S 2.4]{schwedeglobal} which is the group completion \cite[Theorem 2.5.33]{schwedeglobal} and which has this desired classifying property, as we will see in \S\ref{globalclassifyingspaces}.
Note that if we were to naively apply the usual definition of group completion to the monoid $\bigsqcup_n\BB O(n)$, we would need to apply $B$ to it, and this would involve gluing along non-representable maps.
In fact:

\begin{lemma}\label{noBO}
There does not exist an orbi-CW-complex $\BO$ and a functorial bijection between isomorphism classes of stable vector bundles over orbi-CW-complexes $X$ and homotopy classes of maps $X\to\BO$.
\end{lemma}

\begin{proof}
Consider a vector bundle $V$ over a CW-complex $X$ which is not stably trivial (e.g.\ one with nontrivial Pontryagin classes).
Now fix a nontrivial irreducible representation $Q$ of a finite group $G$, and consider the stable vector bundle $(V-\RR^{\left|V\right|})\otimes Q$ over $X\times\BB G$.
The restriction of this stable vector bundle to any $*\times\BB G$ is evidently zero.
Hence if it were pulled back from a classifying map $X\times\BB G\to\BO$, each restriction $*\times\BB G\to\BO$ would factor through $*\to\BO$, hence by Lemma \ref{repfactor} the entire map $X\times\BB G\to\BO$ would factor through $X\to\BO$, implying that our given stable vector bundle is pulled back from $X$.
On the other hand, stable vector bundles on $X\times\BB G$ are simply the direct sum over $\hat G$ of stable vector bundles on $X$, so our given stable vector bundle is definitely not pulled back from $X$.
As in Example \ref{nocones}, the key point in this argument was the use of Lemma \ref{repfactor}.
\end{proof}

\subsection{Stable structures on vector bundles}

A \emph{structure on vector bundles} $\fS$ is a sequence of orbi-CW-complexes $\BB\fS(n)$ for $n\geq 0$ each carrying a vector bundle $\xi_n$ of rank $n$ (equivalently, we could specify the maps $\BB\fS(n)\to\BB O(n)$); we write $\xi$ for $\bigsqcup_{n\geq 0}\xi_n$ over $\bigsqcup_{n\geq 0}\BB\fS(n)$.
An \emph{$\fS$-structure} on a vector bundle $V$ over an orbi-CW-complex $X$ is a map $f:X\to\bigsqcup_n\BB\fS(n)$ together with an isomorphism $V=f^*\xi$.
The set of $\fS$-structures up to homotopy on a vector bundle $V$ is denoted $\Str_\fS(V)$.
An isomorphism $V\xrightarrow\sim W$ induces a bijection $\Str_\fS(V)\xrightarrow\sim\Str_\fS(W)$.

The notion of an $\fS$-structure provides a common language for many structures of interest on vector bundles.
In particular: for $\BB\fS(n)=\BB SO(n)$, an $\fS$-structure is an orientation; for $\BB\fS(n)=\BB U(n/2)$, an $\fS$-structure is a complex structure; for $\BB\fS(n)=*$, an $\fS$-structure is a trivialization (or framing).

A \emph{shift} on a structure on vector bundles $\fS$ is a collection of maps $s_n:\BB\fS(n)\to\BB\fS(n+1)$ and isomorphisms $s_n^*\xi_{n+1}=\xi_n\oplus\underline\RR$ (equivalently, we could specify for each diagram
\begin{equation}\label{stabilization}
\begin{tikzcd}
\BB\fS(n)\ar[r,"s_n"]\ar[d,"\xi_n"]&\BB\fS(n+1)\ar[d,"\xi_{n+1}"]\\
\BB O(n)\ar[r,"\oplus\underline\RR"]&\BB O(n+1)
\end{tikzcd}
\end{equation}
a homotopy between the two compositions).
A shift on $\fS$ gives rise to natural maps $\Str_\fS(V)\to\Str_\fS(V\oplus\underline\RR)$, and a homotopy class of \emph{coarsely stable $\fS$-structure} on $V$ is an element of
\begin{equation}
\Str_\fS^\cst(V):=\varinjlim_n\Str_\fS(V\oplus\underline\RR^n).
\end{equation}
We have $\Str_\fS^\cst(V)=\Str_\fS(V)$ if \eqref{stabilization} is a homotopy pullback square (in the sense that the relevant lifting property holds for every $(D^k,\partial D^k)\times\BB G$).
It also makes sense to put a coarsely stable $\fS$-structure on a coarsely stable vector bundle: $\Str_\fS^\cst(F-\underline\RR^k):=\varinjlim_n\Str_\fS(V\oplus\underline\RR^{n-k})$, and coarsely stable isomorphisms between coarsely stable vector bundles induce maps between their sets of homotoy classes of coarsely stable $\fS$-structures.
The orbi-CW-complex $\bb\fS:=\varinjlim_{n\to\infty}\BB\fS(n)$ (infinite mapping cylinder) classifies coarsely stable vector bundles with $\fS$-structure, in the sense that homotopy classes of maps $X\to\bb\fS$ are in bijection with isomorphism classes of coarsely stable vector bundles with $\fS$-structure.

The set $\Str_\fS^\cst(V)$ has a canonical involution defined by noting the canonical isomorphism $\Str_\fS^\cst(V)=\Str_\fS^\cst(V\oplus\underline\RR)$ and acting via $\id_V\oplus(-1)$ on $V\oplus\underline\RR$.
Note that, having defined the involution on every $\Str_\fS^\cst$ in this way, the isomorphism $\Str_\fS^\cst(V)=\Str_\fS^\cst(V\oplus\underline\RR)$ respects involutions since $\bigl(\begin{smallmatrix}-1&0\\\hfill 0&1\end{smallmatrix}\bigr)$ and $\bigl(\begin{smallmatrix}1&\hfill 0\\0&-1\end{smallmatrix}\bigr)$ lie in the same component of $O(2)$.

A coarsely stable orientation is simply an orientation, due to the aforementioned condition that \eqref{stabilization} be a homotopy pullback being satisfied.
To define coarsely stable complex structures, we should take $\fS_n:=\BB U(\lfloor\frac n2\rfloor)$ and $\xi_n$ to be the tautological bundle plus $\underline\RR$ for $n$ odd, and the map $s$ to be addition of $\underline\RR$ (choosing a convention for which homotopy class of complex structure on $\RR^2$ to use).
A coarsely stable complex structure is weaker than a complex structure (even in even dimensions).
A coarsely stable framing exists only on coarse vector bundles.

A \emph{stable structure on vector bundles} is a structure on vector bundles $\fS$ together with a map $i:*\to\BB\fS(1)$ with an isomorphism $i^*\xi=\RR$ and maps $s_{n,m}:\BB\fS(n)\times\BB\fS(m)\to\BB\fS(n+m)$ with isomorphisms $\xi_n\oplus\xi_m=s_{n,m}^*\xi_{n+m}$ which are associate and graded symmetric in the sense that we now explain.
Associativity means that the two resulting maps $\BB\fS(n)\times\BB\fS(m)\times\BB\fS(k)\to\BB\fS(n+m+k)$ covered by isomorphisms $\xi_n\oplus\xi_m\oplus\xi_k=\xi_{n+m+k}$ are homotopic.
Note that $s:=s_{n,1}\circ(\id\times i)$ defines a shift on $\fS$, so we can already make sense of coarse stabilization.
Graded symmetry is the statement that the maps $\Str_\fS(V)\times\Str_\fS(W)\to\Str_\fS^\cst(V\oplus W)$ given by adding in either order differ by $(-1)^{\left|V\right|\left|W\right|}$, where $-1$ denotes the canonical involution on $\Str_\fS^\cst$ defined above.
We thus obtain graded symmetric maps
\begin{equation}
\Str_\fS^\cst(V)\times\Str_\fS^\cst(W)\to\Str_\fS^\cst(V\oplus W)
\end{equation}
defined as the direct limit over $n$ and $m$ of $(-1)^{n\left|W\right|}$ times the map $\Str_\fS(V\oplus\underline\RR^n)\times\Str_\fS(W\oplus\underline\RR^m)\to\Str_\fS(V\oplus W\oplus\underline\RR^{n+m})$.

A homotopy class of \emph{stable $\fS$-structure} on $V$ is an element of the direct limit
\begin{equation}\label{ststrdef}
\Str_\fS^\st(V):=\varinjlim_W\Str_\fS^\cst(V\oplus W)
\end{equation}
over the category whose objects are vector bundles $W$ equipped with a homotopy class of coarsely stable $\fS$-structure and whose morphisms are injections of vector bundles $V\hookrightarrow W$ together with a homotopy class of coarsely stable $\fS$-structure on $W/V$ such that the resulting homotopy class of coarsely stable $\fS$-structure on $W=V\oplus W/V$ is the given one, modulo homotopy.
(We warn the reader that the forgetful functor from this category to $\Vect(X)$ need not be cofinal.)

\begin{lemma}
The indexing category above is filtered.
\end{lemma}

\begin{proof}
It is non-empty since there is the zero vector bundle.
Given objects $V$ and $V'$, they both admit morphisms to the same object $V\oplus V'$, namely $\oplus V'$ in the former case and $\oplus V$ in the latter case twisted by $(-1)^{\left|V\right|\left|V'\right|}$.
Finally, suppose given two morphisms $V\to V\oplus W$ and $V\to V\oplus W'$ where $V\oplus W=V\oplus W'$.
Then compose further with $\oplus V$, so that the two compositions become $\oplus(W\oplus V)$ and $\oplus(W'\oplus V)$, which we assumed were the same (notice that $\left|W\right|=\left|W'\right|$, so the sign twist in each case is the same).
\end{proof}

There are associative graded symmetric maps
\begin{equation}\label{fSstaddition}
\Str_\fS^\st(V)\times\Str_\fS^\st(W)\to\Str_\fS^\st(V\oplus W).
\end{equation}
In particular, $\Str_\fS^\st(0)$ is an abelian group (to see that it has inverses, note that an element of $\Str_\fS^\st(0)$ is given by a vector bundle $V$ with two coarsely stable $\fS$-structures, and exchange them with a sign twist), each $\Str_\fS^\st(V)$ is either empty or a principal homogeneous space for $\Str_\fS^\st(0)$, and the addition maps \eqref{fSstaddition} are maps of $\Str_\fS^\st(0)$-sets.
Each $\Str_\fS(V)$ also carries a canonical involution given by adding $\underline\RR$ and acting on it by $-1$.

It also makes sense to discuss stable structures on stable vector bundles, and the above continues to apply.

A stable orientation is the same as an orientation.
A stable almost complex structure is strictly weaker than a coarsely stable almost complex structure.
A stable framing is the same as a coarsely stable framing.

\section{Orbifold bordism}

\subsection{Definitions}

We define orbifold bordism $\Omega_*(X,A)$ and derived orbifold bordism $\Omega_*^\der(X,A)$ for any orbispace pair $(X,A)$ as follows.

Consider compact orbifolds with boundary $Z$ together with a representable map $f:(Z,\partial Z)\to(X,A)$.
A bordism between such pairs $(Z_1,f_1)$ and $(Z_2,f_2)$ consists of a compact orbifold with boundary $W$ with a codimension zero embedding $Z_1\sqcup Z_2\hookrightarrow\partial W$ and a representable map $f:(W,\partial W-(Z_1^\circ\cup Z_2^\circ))\to(X,A)$ whose restrictions to $Z_1$ and $Z_2$ are $f_1$ and $f_2$, respectively.
(Alternatively, one could regard $W$ as a compact orbifold with corners, where the corner locus is precisely $\partial Z_1\cup\partial Z_2\subseteq\partial W$.)
Bordism is an equivalence relation (by a collaring result, which allows one to glue together bordisms).
Now $\Omega_*(X,A)$ is the set of pairs $(Z,f)$ modulo compact bordism, graded by dimension.

We now consider a `derived' version of this construction.
A \emph{derived orbifold chart (with boundary)} $Z$ is a tuple $(D,E,s)$ where $D$ (the `domain') is an orbifold (with boundary), $E$ (the `obstruction bundle') is a vector bundle, and $s$ (`the obstruction section') is a smooth section.
A derived orbifold chart with boundary is called compact iff the zero set of $s$ is compact.
A \emph{restriction} of a derived orbifold chart with boundary replaces $D$ with an open subset of $D$ which contains the zero set of $s$ (we may always restrict to a pre-compact subset of $D$, hence the non-compactness of $D$ is never an issue).
A \emph{stabilization} of a derived orbifold chart with boundary $Z=(D,E,s)$ replaces $D$ with the total space of a vector bundle $F$ over $D$, replaces $E$ with its direct sum with $F$, and replaces $s$ with its direct sum with the identity map on $F$.
Bordism of derived orbifold charts is defined as before.
Now $\Omega_*^\der(X,A)$ is the set of compact derived orbifold charts with boundary $Z=(D,E,s)$ together with a representable map $(D,\partial D)\to(X,A)$, modulo compact bordism, restriction, and stabilization.
It is graded by virtual dimension $\dim D-\dim E$.
There is an obvious map $\Omega_*\to\Omega_*^\der$.

While bordism of orbifolds is obviously an equivalence relation (since boundaries of orbifolds with boundary have collars), the analogous assertion for bordisms of derived orbifold charts relies on enough vector bundles.

\begin{proposition}\label{gluingderivedbordisms}
Two compact derived orbifold charts with boundary representable over $(X,A)$ represent the same element of $\Omega_*^\der(X,A)$ iff they are compactly bordant after restricting and stabilizing.
\end{proposition}

\begin{proof}
It suffices to check that the stated relation is transitive.
Suppose that $Z_1\sim Z_2\sim Z_3$, and let us show that $Z_1\sim Z_3$.
The key obstacle to overcome is that the vector bundles by which one stabilizes $Z_2$ to become bordant to (stabilizations of) $Z_1$ and $Z_3$ may not coincide.

We begin by introducing a new perspective on stabilization.
Let $(D,E,s)$ be a derived orbifold chart, and let $Q$ be a vector bundle over $D$ together with a surjection $f:Q\twoheadrightarrow E$.
We obtain a new derived orbifold chart
\begin{equation}
\Bigl(\bigl\{d\in D,q\in Q:s(d)=f(q)\bigr\},Q,\pi_Q\Bigr).
\end{equation}
In fact, this new derived orbifold chart is a stabilization of $(D,E,s)$: indeed a choice of splitting $Q=E\oplus\ker f$ identifies the new derived orbifold chart above with the stabilization of $(D,E,s)$ by $\ker f$.

Let us now observe that stabilization is transitive: a stabilization of a stabilization is a stabilization.
The point is just that if $(D,E,s)\leadsto(D',E',s')\leadsto(D'',E'',s'')$ are stabilizations, then the vector bundle $W$ by which the second stabilization stabilizes is pulled back from $D\subseteq D'$ (the first stabilization says that $D'$ is the total space of a vector bundle over, hence has a projection map down to, $D$).
Choosing such an identification of $W$ with the pullback of its restriction to $D$ identifies $(D'',E'',s'')$ with a stabilization of $(D,E,s)$.

We now return to the problem at hand.
We have bordisms $C_{12}$ and $C_{23}$ between stabilizations of $Z_1$, $Z_2$, $Z_3$.
By making a small deformation, we may assume that these bordisms are collared (i.e.\ near the boundary are the product of the boundary times $[0,\varepsilon)$).
Consider the orbispace $C_{12}\cup_{Z_2}C_{23}$ (i.e.\ the gluing of the `domains' of the corresponding derived orbifolds, possibly after restricting to pre-compact open subsets thereof).
By enough vector bundles Theorem \ref{enough}, there is a module faithful vector bundle $Q$ over this space.
There thus exists an $N<\infty$ and surjections $\Phi_{12}$ and $\Phi_{23}$ from $Q^{\oplus N}|_{C_{12}}$ and $Q^{\oplus N}|_{C_{23}}$ to the obstruction spaces $E_{12}$ of $C_{12}$ and $E_{23}$ of $C_{23}$, respectively (say, independent of the radial coordinate of the collar near the boundary), thus determining stabilizations of $C_{12}$ and $C_{23}$, respectively.
The resulting composite stabilizations of $Z_2$ on the boundary are thus determined by surjections $\Psi_{12}\circ\Phi_{12}$ and $\Psi_{23}\circ\Phi_{23}$ from $Q^{\oplus N}$ to the obstruction space $E_2$ of $Z_2$ (where $\Psi_{12}:E_{12}\to E_2$ and $\Psi_{23}:E_{23}\to E_2$ are the surjections inducing the stabilizations of $Z_2$ on the boundaries of $C_{12}$ and $C_{23}$, respectively).
If these surjections $\Psi_{12}\circ\Phi_{12}$ and $\Psi_{23}\circ\Phi_{23}$ from $Q^{\oplus N}$ to $E_2$ are homotopic through surjections, we may insert such a homotopy in the collar coordinate and glue the stabilizations of $C_{12}$ by $\Phi_{12}$ and $C_{23}$ by $\Phi_{23}$ together to obtain the desired glued bordism between (stabilizations of) $Z_1$ and $Z_3$.
By replacing $N$ with $2N$ and replacing $\Phi_{12}$ and $\Phi_{23}$ with $\Phi_{12}\oplus 0$ and $0\oplus\Phi_{23}$, respectively, the desired homotopy through surjections is simply the obvious linear interpolation.
\end{proof}

\begin{remark}
A derived orbifold is an object with an atlas of derived orbifold charts.
It is a consequence of enough vector bundles that every derived orbifold has in fact a global chart.
Thus we may (and do) define derived orbifold bordism groups purely in terms of derived orbifold charts, without delving into the details of the definition of derived orbifolds.
The cost of this approach is that enough vector bundles becomes a crucial ingredient in the proofs of most properties of derived orbifold bordism as we have defined it here.
\end{remark}

\subsection{Basic properties}

The sets $\Omega_d$ and $\Omega_d^\der$ are both abelian groups under disjoint union; each element is its own inverse.

These groups $\Omega_d$ and $\Omega_d^\der$ are functorial under representable maps of pairs, namely they define functors $\RepOrbSpcPair\to\Ab$.
In fact, they descend to functors
\begin{equation}
\RepOrbSpc_*\to\Ab,
\end{equation}
which can be seen either directly from the definition or by appealing to Proposition \ref{excision}.
(The proof is exactly as for classical bordism, so we omit it.)

There is a natural map $\Omega_d\to\Omega_d^\der$ (take $E=0$).
Since a section of a vector bundle over a manifold can be perturbed to be transverse to zero, the map $\Omega_d(X,A)\to\Omega_d^\der(X,A)$ is an isomorphism for $(X,A)\in\Spc_*$.

There are natural product maps
\begin{align}
\Omega_*(X,A)\otimes\Omega_*(Y,B)&\to\Omega_*((X,A)\times(Y,B))\\
\Omega_*^\der(X,A)\otimes\Omega_*^\der(Y,B)&\to\Omega_*^\der((X,A)\times(Y,B))
\end{align}
given simply by taking product of (derived) orbifolds.

(Derived) orbifold bordism groups also satisfy exactness:

\begin{proposition}\label{bordismexact}
The functors $\Omega_d$ and $\Omega_d^\der$ send any cofiber sequence \eqref{cofiber} to an exact sequence of abelian groups.
\end{proposition}

\begin{proof}
We treat both cases ($\Omega_d$ and $\Omega_d^\der$) simultaneously, writing $\Omega_*^{(\der)}$ for either one.

It is immediate that any element of $\Omega_*^{(\der)}(X,A)$ represented by something mapped entirely to $A$ is zero (multiply by $I$ to obtain a bordism to the empty set).
It follows that the composition $\Omega_*^{(\der)}(Y,B)\to\Omega_*^{(\der)}(X,A)\to\Omega_*^{(\der)}(X,A\cup_BY)$ vanishes.

Now suppose an element $(Z,\partial Z)$ of $\Omega_*^{(\der)}(X,A)$ is sent to zero in $\Omega_*^{(\der)}(X,A\cup_BY)$.
There is thus a null bordism $C$ of $(Z,\partial Z)$ over $(X,A\cup_BY)$ (in the case of $\Omega_*^\der$, this uses Proposition \ref{gluingderivedbordisms}).
The boundary of this bordism consists of $Z$ (mapped to $(X,A)$) and its complement, which is mapped to $A\cup_BY$.
Replace the map $f:C\to X$ with its composition with a small perturbation of the identity $\Phi:X\to X$ satisfying $\Phi(A)\subseteq A$, $\Phi(Y)\subseteq Y$, and $\Phi(\Nbd B)\subseteq B$ (such a map $\Phi$ may be constructed by induction on cells).
Since the closure of $\Phi^{-1}(Y-B)$ is disjoint from $B$, it follows that the closure of the set $(f|_{\partial C})^{-1}(Y-B)$ is disjoint from $Z\subseteq\partial C$.
Now take $Z'\subseteq\partial C$ a compact codimension zero submanifold with boundary, disjoint from $Z$, containing $(f|_{\partial C})^{-1}(Y-B)$.
Thus $(Z',\partial Z')\to(Y,B)$ represents an element of $\Omega_*^{(\der)}(Y,B)$ which is sent to $Z$ in $\Omega_*^{(\der)}(X,A)$.
\end{proof}

Applying Proposition \ref{bordismexact} to the Puppe sequence gives a long exact sequence, which acquires the usual form once we observe that $\Omega_*(X,A)=\Omega_{*+1}((X,A)\times(I,\partial I))$ (and likewise for $\Omega_*^\der$) as we will see next.
Namely, for any cofiber sequence \eqref{cofiber}, we obtain a (bi-infinite) long exact sequence
\begin{equation}
\cdots\to\Omega_*(Y,B)\to\Omega_*(X,A)\to\Omega_*(X,A\cup_BY)\to\Omega_{*-1}(Y,B)\to\cdots
\end{equation}
and the same for $\Omega_*^\der$.

\begin{example}\label{dertonondernonzero}
Here is a way to detect nontrivial negative degree classes in derived bordism.
Let $G$ be any finite group.
There is an \emph{ungraded} map
\begin{align}
\Omega_*(\BB G)&\to\Omega_*(*)\\
M/G&\mapsto M^G
\end{align}
(note that every representable map $N\to\BB G$ is of the form $M/G\to\BB G$ for $M=N\times_{\BB G}*$).
Similarly, there is an ungraded map
\begin{align}
\Omega_*^\der(\BB G)&\to\Omega_*^\der(*)=\Omega_*(*)\\
(M,E,s)/G&\mapsto(M^G,E^G,s|_{M^G}).
\end{align}
One should be careful to note that this map does indeed respect bordism (in particular, stabilization).
For any $G$-representation $V$, this map sends
\begin{equation}
(\BB G,V/G,0)\in\Omega_{-\dim V}^\der(\BB G)
\end{equation}
to $(*,V^G,0)\in\Omega_{-\dim V^G}^\der(*)$, which is nonzero iff $V^G=0$.
We conclude that if $V^G=0$ then $\Omega_{-\dim V}^\der(\BB G)\ne 0$.
\end{example}

\subsection{(Inverse) Thom maps}

For any vector bundle $V$ over $X$, there are natural \emph{inverse Thom maps} (terminology following Schwede \cite[\S 6]{schwedeglobal})
\begin{align}
\Omega_d(X,A)&\to\Omega_{d+\left|V\right|}((X,A)^V),\\
\Omega_d^\der(X,A)&\to\Omega_{d+\left|V\right|}^\der((X,A)^V),
\end{align}
given by replacing a given (derived) orbifold with the Thom space of the pullback of $V$.
We also have \emph{Thom maps} in the opposite direction
\begin{align}
\Omega_{d+\left|V\right|}((X,A)^V)&\to\Omega_d(X,A),\\
\Omega_{d+\left|V\right|}^\der((X,A)^V)&\to\Omega_d^\der(X,A),
\end{align}
given by intersecting with the zero section of $V$.
More precisely, the Thom map on $\Omega_*$ is only defined for coarse vector bundles $V$, and it requires an appeal to Sard's theorem to conclude that intersecting with a generic perturbation of the zero section of $V$ is transverse.
The Thom map on $\Omega_*^\der$ is defined for all vector bundles $V$ and consists simply of adding $V$ to the obstruction bundle and the identity section to the obstruction section.

\begin{proposition}\label{thomiso}
The Thom map and the inverse Thom map are inverses.
\end{proposition}

\begin{proof}
We have four compositions to show are the identity map:
\begin{align}
\label{compA}\Omega_d^{(\der)}(X,A)&{}\to\Omega_{d+\left|V\right|}^{(\der)}((X,A)^V)\to\Omega_d^{(\der)}(X,A)\\
\label{compB}\Omega_{d+\left|V\right|}^{(\der)}((X,A)^V)&{}\to\Omega_d^{(\der)}(X,A)\to\Omega_{d+\left|V\right|}^{(\der)}((X,A)^V)
\end{align}
The map \eqref{compA} for $\Omega_*$ is the identity by inspection.
The map \eqref{compA} for $\Omega_*^\der$ is the identity since its action on a given derived orbifold chart is to stabilize by $V$.
The map \eqref{compB} for $\Omega_*$ is also the identity by inspection (given a transverse perturbation $\epsilon$ of the zero section which is transverse to a given orbifold, consider replacing $(X,A)^V$ with a small tubular neighborhood of the image of $\epsilon$ relative its boundary).

The map \eqref{compB} for $\Omega_*^\der$ may be expressed alternatively as
\begin{equation}
\Omega_{d+\left|V\right|}^\der((X,A)^V)\to\Omega_{d+2\left|V\right|}^\der((X,A)^{V\oplus V})\to\Omega_{d+\left|V\right|}^\der((X,A)^V)
\end{equation}
which looks very much like \eqref{compA}, except it is not quite the same since here the first map `inflates' along the second copy of $V$ whereas the second map intersects along the zero section of the first copy of $V$.
However, we may note that $(X,A)^{V\oplus V}$ has an automorphism, homotopic to the identity map, given by the matrix $\bigl(\begin{smallmatrix}\hfill0&1\\-1&0\end{smallmatrix}\bigr)\cdot\id_V$, conjugation by which turns the second map into intersection with the second copy of $V$, putting our composition into the form \eqref{compA} for $\Omega_*^\der$ which we already saw is the identity map.
\end{proof}

Given the Thom isomorphism, we may extend $\Omega_*$ and $\Omega_*^\der$ to orbispectra as follows.
Bordism $\Omega_*$ extends to naive orbispectra $\RepOrbSpc[\Sigma^{-1}]$ by taking $\Omega_*(\Sigma^{-n}(X,A)):=\Omega_{*+n}(X,A)$, which is consistent since $\Omega_*(X,A)=\Omega_{*+1}(\Sigma(X,A))$ by the Thom isomorphism.
Derived bordism $\Omega_*^\der$ extends to genuine orbispectra $\RepOrbSp^f$ by taking $\Omega_*((X,A)^{-V}):=\Omega_{*+\left|V\right|}(X,A)$, which is again consistent by the Thom isomorphism.

When $V$ is not coarse, the inverse Thom map $\Omega_*(X,A)\to\Omega_{*+\left|V\right|}((X,A)^V)$ is in general not an isomorphism.
It is thus natural to ask whether $\Omega_*((X,A)^V)$ may be expressed as bordism classes of some class of (derived) orbifolds mapping to $(X,A)$ (rather than $(X,A)^V$).
We will see how to do this below, based on Wasserman's theorem, which we will meet shortly.
This is the key to extending $\Omega_*$ to genuine orbispectra.

For the moment, we will observe that $\Omega_*\to\Omega_*^\der$ is the \emph{localization} at the inverse Thom maps, in the following sense:

\begin{lemma}
For finite orbi-CW-pairs $(X,A)$, the natural map
\begin{equation}
\Omega_*[{\textstyle\frac 1\tau}](X,A):=\varinjlim_{V/X}\Omega_{*+\left|V\right|}((X,A)^V)\xrightarrow\sim\varinjlim_{V/X}\Omega_{*+\left|V\right|}^\der((X,A)^V)=\Omega_*^\der(X,A)
\end{equation}
is an isomorphism.
\end{lemma}

\begin{proof}
We prove surjectivity.
Let $(D,E,s)$ be a derived orbifold chart representable over $(X,A)$.
To obtain the corresponding derived orbifold chart over $(X,A)^V$, we simply replace $D$ with the total space of the pullback of $V$ to it.
By enough vector bundles Theorem \ref{enough}, we may take $V$ so that its pullback to $D$ surjects onto $E$.
Now we may perturb $s$ by adding to it (epsilon times) this surjection, thus making it transverse.
Hence our derived orbifold chart lies in the image of $\Omega_{*+\left|V\right|}((X,A)^V)\to\Omega_{*+\left|V\right|}^\der((X,A)^V)$.

Injectivity follows from the same argument applied to a derived orbifold bordism between two orbifolds.
\end{proof}

\subsection{Wasserman's theorem}

A remarkable observation of Wasserman \cite{wasserman} provides a sufficient condition under which a section of a vector bundle over an orbifold may be perturbed to become transverse to zero.
In particular, it gives a condition under which a derived orbifold is bordant to an orbifold.

To state this condition, let us fix some notation.
For a vector bundle $V$ over an orbispace and a point $p$, we may decompose the fiber $V_p$ into a direct sum of isotypic pieces, indexed by the set $\hat G_p$ of isomorphism classes of real irreducible representations of the isotropy group $G_p$ of $p$.
In particular, we may split $V_p$ as the direct sum of the isotropy invariant part $(V_p)^{G_p}=(V_p)_\1$ and the direct sum $(V_p)_{\hat G_p-\1}$ of isotypic pieces of nontrivial representations.
We denote by $V_{\widehat\iso-\1}\subseteq V$ the sum of the isotypic pieces associated to nontrivial representations (note that $V_{\widehat\iso-\1}$ is not itself a vector bundle), and for a map of vector bundles $f$, we denote by $f_{\widehat\iso-\1}$ its action on these subspaces.
Given a vector bundle $V$ over an orbifold $X$ together with a map $\alpha:TX\to V$ for which $\alpha_{\widehat\iso-\1}$ is surjective, a section $s:X\to V$ is called \emph{$\alpha$-consistently transverse (to zero)} iff over its zero set $ds$ is surjective with $(ds)_{\widehat\iso-\1}=\alpha_{\widehat\iso-\1}$.

\begin{theorem}[Wasserman \cite{wasserman}]\label{wasserman}
Let $X$ be an orbifold, let $E$ be a vector bundle over $X$, and fix a map $\alpha:TX\to E$ for which $\alpha_{\widehat\iso-\1}$ is surjective.
Every section of $E$ has a $C^0$-small perturbation which is $\alpha$-consistently transverse.
This perturbation may be taken relative a neighborhood of any closed set over which it is already $\alpha$-consistently transverse.
\end{theorem}

(We credit this result to Wasserman \cite{wasserman}, although Wasserman only stated the special case that $X=\RR^n/G$ and $E$ is the descent of the trivial bundle $\RR^n$ with the same $G$-action, and $\alpha$ is the identity.)

\begin{proof}
We proceed by induction over the stratification of $X$ by order of stabilizer.
By triangulating a given stratum, it suffices to perturb on any given disk rel boundary inside $X$, which all have a standard local model.
In other words, it suffices to consider the case of $X=D^k\times D^\ell\times W/G$ where $G\acts W$ has zero invariant part $W^G=0$, the section $s$ is $\alpha$-consistently transverse over a neighborhood of $\partial D^k\times 0\times 0$, and we would like to make it $\alpha$-consistently transverse over a neighborhood of $D^k\times 0\times 0$.
Now over $D^k\times 0\times 0$, the derivative $ds$ is $G$-equivariant, hence respects the decomposition into isotypic pieces:
\begin{equation}
(ds)_\1\oplus(ds)_{\hat G-\1}:(TD^k\oplus TD^\ell)\oplus W\to E^G\oplus E_{\hat G-\1}.
\end{equation}
By perturbing (rel a neighborhood of the boundary) the restriction of $s$ to $D^k\times 0\times 0$, we may make $(ds)_\1$ surjective (note that $s$ is constrained to land inside $E^G$ over $D^k\times 0\times 0$), and we may then extend $s$ to a neighborhood of $D^k\times 0\times 0$ so that $(ds)_{\hat G-\1}=\alpha_{\hat G-\1}$ over $D^k\times 0\times 0$.

We are not quite done, however, since the above construction ensures that our perturbed section $s$ will be $\alpha$-consistently transverse over $D^k\times 0\times 0$ and a neighborhood of $\partial D^k\times 0\times 0$, but not over a neighborhood of $D^k\times 0\times 0$.
To fix this, choose an isomorphism $E=\pi^*E$ where $\pi$ denotes the projection $\pi:D^k\times D^\ell\times W/G\to D^k\times D^\ell$ forgetting the last coordinate.
Now given the section $s$ defined above, set
\begin{equation}
\bar s(a,b,c):=s(a,b,0)+\alpha_{\hat G-\1}(c),
\end{equation}
where we use the isomorphism $E=\pi^*E$ to make sense of the right hand side as an element of the fiber of $E$ over $(a,b,c)\in D^k\times D^\ell\times W/G$.
Now this section $\bar s$ is certainly $\alpha$-consistently transverse over a neighborhood of $D^k\times 0\times 0$, however it does not agree with $s$ over a neighborhood of $\partial D^k\times 0\times 0$.
Instead, let us use $\varphi\cdot\bar s+(1-\varphi)\cdot s$ for a smooth function $\varphi:D^k\to[0,1]$ vanishing near $\partial D^k$ and which equals $1$ over a large compact set.
This interpolation is now $\alpha$-consistently transverse over a neighborhood of $D^k\times 0\times 0$, noting that the restriction of $d\varphi$ to the $(\cdot)_{\widehat\iso-\1}$ part of the tangent bundle is zero.
\end{proof}

\begin{remark}
A stable homotopy theoretic analogue of this argument appears in tom Dieck \cite[Satz 5]{tomdieckII} Schwede \cite[Theorem 6.2.33]{schwedeglobal}.
It would be interesting to explore whether a stable homotopy theoretic analogue of Fukaya--Ono's `integer part' construction \cite{fukayaono} exists as well (that construction follows a strategy similar to Wasserman's strategy above, though rather than using $\alpha$ in the normal directions, one requires complex polynomial behavior in the normal directions).
\end{remark}

\begin{corollary}\label{wassermancoarselystable}
A derived orbifold chart whose tangent bundle is stably isomorphic to a coarsely stable vector bundle is bordant to an orbifold.
\end{corollary}

\begin{proof}
Let $Z=(D,E,s)$ be a derived orbifold chart.
By assumption, the stable vector bundle $TD-E$ is stably isomorphic to a coarsely stable vector bundle $F-\underline\RR^N$ (in fact, we will not use anything special about $\underline\RR^N$ other than that it is coarse).
In other words, there exists a vector bundle $V$ and an isomorphism $TD\oplus V\oplus\underline\RR^N=E\oplus V\oplus F$.
By stabilizing our derived orbifold chart $(D,E,s)$ by $V$, we may reduce this to $TD\oplus\underline\RR^N=E\oplus F$.
Now the composition $\alpha:TD\to TD\oplus\underline\RR^N=E\oplus F\to E$ is evidently surjective on $(\cdot)_{\widehat\iso-\1}$ pieces.
We can thus apply Wasserman Theorem \ref{wasserman} to perturb $s$ to a section $s'$ which is transverse to zero (and agrees with $s$ outside a compact set).
The desired bordism is thus $(D\times[0,1],E\times[0,1],ts+(1-t)s')$.
\end{proof}

The literal converse to Corollary \ref{wassermancoarselystable} is false for trivial reasons ($\partial[0,1]$ times anything is null bordant yet need not have coarsely stable tangent bundle).
The next subsection formulates an `up to bordism' version of Corollary \ref{wassermancoarselystable} which is an `if and only if' (or rather isomorphism) statement.

\subsection{Orbifold bordism as oriented derived orbifold bordism}\label{ordinaryasderived}

Let us now explain how Wasserman's theorem implies, as one might expect after seeing Corollary \ref{wassermancoarselystable}, that orbifold bordism may be expressed as derived orbifold bordism with a sort of tangential structure, namely what we will call a \emph{coarsely stable structure} on its stable tangent bundle.
We may thus think of $\Omega_*$ as an `oriented' version of $\Omega_*^\der$, in the sense that modifying the definition of $\Omega_*^\der$ by imposing a marking on the stable tangent bundle yields $\Omega_*$.

A \emph{coarsely stable structure} on a stable vector bundle $V$ is a coarsely stable vector bundle $W$ and a stable isomorphism $V=W$.
A given stable vector bundle may admit multiple non-isomorphic coarsely stable structures (non-isomorphic coarsely stable vector bundles may be stably isomorphic).

Derived orbifold bordism with coarsely stable tangential structure $\Omega_*^{\cst,\der}$ is defined as follows.
Consider derived orbifold charts $Z=(D,E,s)$ representable over $(X,A)$ together with a vector bundle $A$ and a stable isomorphism $A-\underline\RR^{\left|E\right|-\left|TD\right|-\left|A\right|}=TD-E$, modulo restriction, stabilization, $A\mapsto A\oplus\underline\RR$, and bordism.
Let us argue that bordism after restriction, stabilization, and $A\mapsto A\oplus\underline\RR$ is transitive (hence is an equivalence relation).
As argued in the proof of Proposition \ref{gluingderivedbordisms}, given two bordisms $C_{12}$ and $C_{23}$, we may stabilize so that the requisite stabilizations of $Z_2$ coincide.
The bordisms may thus be glued, so it suffices to argue that the coarsely stable structures can also be glued.
We have vector bundles $A_{12}$ on $C_{12}$ and $A_{23}$ on $C_{23}$ and a coarsely stable isomorphism between their restrictions to $Z_2$.
Thus stabilizing $A_{12}$ and $A_{23}$ by adding $\underline\RR^k$, we get a genuine isomorphism on $Z_2$, which thus allows us to glue them together.
Now we have stable isomorphisms between this glued coarsely stable vector bundle and the tangent space to our glued derived bordism, separately on $C_{12}$ and $C_{23}$, and their restrictions to $Z_2$ are homotopic.
They may thus be glued (non-uniquely).
We thus conclude that bordism after restriction, stabilization, and $A\mapsto A\oplus\underline\RR$ is an equivalence relation, as desired.

\begin{remark}
One can similarly define a theory $\Omega_*^{-\cst,\der}$ of bordism of derived orbifolds with coarsely stable structure on minus their tangent bundle.
\end{remark}

Given that bordism after restriction and stabilization and $A\mapsto A\oplus\underline\RR$ is an equivalence relation, the proof of Proposition \ref{bordismexact} now applies to show that $\Omega_*^{\cst,\der}:\OrbSpc_*\to\Ab$ sends cofiber sequences to exact sequences.

\begin{proposition}
The natural map $\Omega_*\xrightarrow\sim\Omega_*^{\cst,\der}$ is an isomorphism.
\end{proposition}

\begin{proof}
Surjectivity is the statement that every derived orbifold chart $(D,E,s)$ with coarsely stable vector bundle $\xi$ and stable isomorphism $\xi=TD-E$ is bordant to an orbifold (i.e.\ a derived orbifold chart whose obstruction section is transverse).
Corollary \ref{wassermancoarselystable} provides a transverse perturbation of $s$ which, executed over $[0,1]$, defines the desired bordism.

Injectivity is (given the nontrivial result, proved just above, that derived bordism with coarsely stable tangential structure is an equivalence relation) the statement that every derived orbifold bordism between stabilizations of orbifolds, with coarsely stable structure on its tangent bundle, agreeing with the tautological such on the boundary, can be perturbed rel boundary to be transverse.
Concretely, such a structure is (after stabilizing as in the proof of Corollary \ref{wassermancoarselystable}) a vector bundle $F$ and an isomorphism $F\oplus E=TD\oplus\underline\RR^k$, which on the boundary must coincide with the isomorphism $E=TD$ given by $ds$ and $F=\underline\RR^k$ (some isomorphism); thus $s$ is already $\alpha$-transverse over the boundary, so the relative form of Wasserman's Theorem \ref{wasserman} gives us what we want.
\end{proof}

The theory $\Omega_*^{\cst,\der}$ may be twisted: for any stable vector bundle $\xi$ on $X$, we may define a group $\Omega_*^{\xi\oplus\cst,\der}(X,A)$ of bordism classes of derived orbifolds carrying a coarsely stable vector bundle $W$ and an isomorphism $TD-E=\xi\oplus W$.
These twisted theories are the natural setting for inverse Thom maps
\begin{equation}
\Omega_*^{\xi\oplus\cst,\der}(X,A)\to\Omega_{*+\left|V\right|}^{\xi\oplus V\oplus\cst,\der}((X,A)^V).
\end{equation}
Now there is an obvious Thom map in the reverse direction (add $V$ to the obstruction space and the identity map to the obstruction section) which is an inverse to the inverse Thom map exactly as in Proposition \ref{thomiso}.
There are also forgetful maps
\begin{equation}
\Omega_*^{\xi\oplus V\oplus\cst,\der}(X,A)\to\Omega_*^{\xi\oplus\cst,\der}(X,A)
\end{equation}
for vector bundles $V$, which need not be isomorphisms.
This refines the discussion of inverse Thom maps for $\Omega_*$ given above.

The Thom isomorphism for these twisted theories allows us to extend $\Omega_*=\Omega_*^{\cst,\der}$ to genuine orbispectra by defining $\Omega_*((X,A)^{-V})=\Omega_*^{\cst,\der}((X,A)^{-V}):=\Omega_{*+\left|V\right|}^{V\oplus\cst,\der}(X,A)$.
To check that this indeed defines a functor $\RepOrbSp^f\to\Ab$, use the localization result Proposition \ref{reporbsecondlocalization} and the twisted Thom isomorphism.
Indeed, the definition above gives a functor out of the direct limit of $\RepOrbSpc_{N,k}$ (by the Thom isomorphism), and it satisfies excision (by inspection), thus descending to $\RepOrbSp^f$.
This functor sends cofiber sequences to exact sequences (the proof for twisted $\Omega_*^{\cst,\der}$ is the same as for untwisted, which was already mentioned above).

\subsection{Tangential structure}

We define orbifold and derived orbifold bordism groups with tangential structure, and we show how to generalize the basic properties proven above to this setting.
In a word, a structure on vector bundles $\fS$ with a shift allows us to define orbifold bordism groups $\Omega_*^\fS$, and a stable structure on vector bundles $\fS$ allows us to define derived orbifold bordism groups $\Omega_*^{\fS,\der}$.

For $\fS$ a structure on vector bundles with a shift, we define bordism groups $\Omega_*^\fS$ as follows.
We consider orbifolds with coarsely stable $\fS$-structure on their tangent bundle.
Using the isomorphism $\Str_\fS^\cst(V)=\Str_\fS^\cst(V\oplus\underline\RR)$, we can define a notion of bordism of orbifolds with coarsely stable $\fS$-structure: given a boundary marking $Z_0\sqcup Z_1\subseteq\partial W$, we use the isomorphisms $\Str_\fS^\cst(TZ_i)=\Str_\fS^\cst(TZ_i\oplus\underline\RR)=\Str_\fS^\cst(TW|_{Z_i})$ (where, crucially, we identify $\underline\RR$ with the inward normal along $Z_0$ and the outward normal along $Z_1$) to require compatibility between the coarsely stable $\fS$-structure on $W$ with those on $Z_0$ and $Z_1$.
Bordism is a symmetric relation, as can be seen by inverting the coarsely stable structure on the bordism (i.e.\ applying the canonical involution of $\Str_\fS^\cst$).
It is transitive since coarsely stable $\fS$-structures glue: by applying $\oplus\underline\RR^k$ enough times, we reduce to gluing for $\fS$-structures; an $\fS$-structure over $C_{12}$ and over $C_{23}$ which are homotopic over $Z_2$ glue, non-uniquely, to an $\fS$-structure over $C_{12}\cup_{Z_2}C_{23}$.
The resulting $\fS$-bordism groups $\Omega_*^\fS$ satisfy functoriality (including excision) and exactness by the same reasoning as before.
They have inverse Thom maps
\begin{equation}
\Omega_*^\fS(X,A)\to\Omega_{*+1}^\fS((X,A)\times(I,\partial I))
\end{equation}
and Thom maps in the reverse direction which are inverse to the inverse Thom maps; this extends $\Omega_*^\fS$ to a functor on naive orbispectra.
As before, we may extend $\Omega_*^\fS$ to genuine orbispectra by viewing it as a structured version of derived orbifold bordism.
Namely, $\Omega_*^\fS$ coincides with the group $\Omega_*^{\fS\cst,\der}$ of bordism classes of derived orbifold charts carrying a coarsely stable vector bundle $A$ with isomorphism $A=TD-E$ and a coarsely stable $\fS$-structure on $A$.
These groups $\Omega_*^{\fS\cst,\der}$ satisfy the same properties as above, and the map $\Omega_*^\fS\to\Omega_*^{\fS\cst,\der}$ is an isomorphism.
There are also twisted versions $\Omega_*^{\xi\oplus\fS\cst,\der}(X,A)$ for any stable vector bundle $\xi$ on $X$, and there are inverse Thom maps $\Omega_*^{\xi\oplus\fS\cst,\der}(X,A)\to\Omega_{*+\left|V\right|}^{\xi\oplus V\oplus\fS\cst,\der}((X,A)^V)$ and forgetful maps $\Omega_*^{\xi\oplus V\oplus\fS\cst,\der}\to\Omega_*^{\xi\oplus\fS\cst,\der}$ for any vector bundle $V$ with coarsely stable structure.
We may thus extend $\Omega_*^\fS$ to genuine orbispectra by taking $\Omega_*^\fS((X,A)^{-\xi}):=\Omega_{*+\left|\xi\right|}^{\xi\oplus\fS\cst,\der}(X,A)$.

Now suppose $\fS$ is a stable structure on vector bundles, and let us define derived $\fS$-orbifold bordism.
We consider derived orbifold charts $(D,E,s)$ together with a stable $\fS$-structure on $TD-E$, modulo restriction, stabilization, and bordism as before.
The equivalence relation proof Proposition \ref{gluingderivedbordisms} applies; for this, we need to know that stable structures on vector bundles glue, and the main point to see that is to use enough vector bundles to know that we can stabilize by vector bundles on $C_{12}\cup_{Z_2}C_{23}$ to reduce to gluing (again, non-uniqely) $\fS$-structures on $C_{12}$ and $C_{23}$ which agree over $Z_2$.
The resulting theory thus satisfies exactness.
These theories can be twisted: we may define $\Omega_*^{V\oplus\fS,\der}(X,A)$ to be bordism classes of derived orbifold charts with a stable $\fS$-structure on $TD-E-V$, where $V$ is any stable vector bundle on $X$; an $\fS$-structure on $V$ gives an isomorphism $\Omega_*^{V\oplus\fS,\der}(X,A)=\Omega_*^{\fS,\der}(X,A)$.
Inverse Thom maps for $\Omega_*^{\fS,\der}$ now take the form
\begin{align}
\Omega_*^{\xi\oplus\fS,\der}(X,A)\to\Omega_{*+\left|V\right|}^{\xi\oplus V\oplus\fS,\der}((X,A)^V),
\end{align}
and there are also Thom maps which are inverse to these.
We may thus extend $\Omega_*^{\fS,\der}$ to orbispectra as $\Omega_*^{\fS,\der}((X,A)^{-\xi}):=\Omega_{*+\left|\xi\right|}^{\xi+\fS,\der}(X,A)$.
The natural map
\begin{equation}
\varinjlim_W\Omega_{*+\left|W\right|}^\fS((X,A)^W)\to\varinjlim_W\Omega_{*+\left|W\right|}^{\fS,\der}((X,A)^W)=\Omega_*^{\fS,\der}(X,A)
\end{equation}
is an isomorphism, where the direct limit is over all vector bundles with $\fS$-structure as in \eqref{ststrdef}.
There are also graded symmetric product maps on $\Omega_*^\fS$ and $\Omega_*^{\fS,\der}$.

\subsection{Fundamental classes}

We make a few remarks about fundamental classes of orbifolds and derived orbifolds.

A closed orbifold $M$ has a tautological fundamental class $[M]\in\Omega_{\dim M}(M)$.
This class is best viewed as arising from the more refined fundamental class $[M]\in\Omega_0^\fr(M^{-TM})$ lying in the bordism group of derived orbifolds representable over $M$ with a stable isomorphism between their tangent bundle and $TM$.
This class may be pushed forward under the map $\Omega_0^\fr(M^{-TM})\to\Omega_0(M^{-TM})$ forgetting the framing and under the inverse Thom map $\Omega_0(M^{-TM})\to\Omega_{\dim M}(M)$ to obtain the naive fundamental class $[M]\in\Omega_{\dim M}(M)$.
If $TM$ is equipped with an $\fS$-structure, then we may push forward to $\Omega_{\dim M}^\fS(M)$ using the $\fS$-structured inverse Thom map to obtain the $\fS$-structured fundamental class.
The same applies when $M$ is a compact orbifold with boundary, just replacing $M$ with the pair $(M,\partial M)$.

Let us now work towards the fundamental class of a derived orbifold.
Consider an inclusion of subcomplexes $(Y,B)\to(X,A)$ (so $B=Y\cap A$) and a vector bundle $E$ over $X$ with a section $s:X\to E$ whose zero set is (contained in) $Y$.
There is then an induced map $(X,A)\to(Y,B)^E$, obtained by appealing to the fact that $(Y,B)\subseteq(X,A)$ is a retract of any sufficiently small neighborhood, and any two such retracts are homotopic.
Thus if $(D,E,s)$ is a derived orbifold chart and $Z:=s^{-1}(0)$ has the same neighborhood retract property, we obtain a map
\begin{equation}
(D,\partial D)^{-TD}\to(Z,\partial Z)^{-TZ}
\end{equation}
where $TZ:=TD-E$, $\partial Z:=Z\cap\partial D$, and $\dim Z:=\dim D-\dim E$ (note that whereas the left side is the dual of $D$, the right side is very much \emph{not} the dual of $Z$ unless $s$ is transverse to zero).
We may now define $[Z]\in\Omega_0^\fr((Z,\partial Z)^{-TZ})$ as the image of $[D]$ under the above map.
Since $TZ$ is not a vector bundle, rather only a stable vector bundle, we are no longer able to map this fundamental class to $\Omega_{\dim Z}(Z,\partial Z)$, rather only to $\Omega_{\dim Z}^\der(Z,\partial Z)$ (the map now involves the inverse of an inverse Thom map which only exists for $\Omega_*^\der$).
If $TZ$ is equipped with a stable $\fS$-structure, then we can also push forward the fundamental class to $\Omega_{\dim Z}^{\fS,\der}(Z,\partial Z)$.

If our derived orbifold $Z\subseteq D$ does not have the neighborhood retract property, the above reasoning produces only a class in the inverse limit $\varprojlim_{\epsilon>0}\Omega_0^\fr((Z_\varepsilon,\partial Z_\varepsilon)^{-TZ})$ where $Z_\varepsilon\subseteq D$ denotes the $\varepsilon$-neighborhood of $Z$.
This is not really the correct bordism group to associate to $(Z,\partial Z)^{-TZ}$, rather differing from it by a $\varprojlim^1$ term.
In the correct bordism group to attach to it, a cycle would be a collection of (derived, with structure) orbifolds $(M_{\varepsilon_i},\partial M_{\varepsilon_i})\to(Z_{\varepsilon_i},\partial Z_{\varepsilon_i})$ together with bordisms between $M_{\varepsilon_i}$ and $M_{\varepsilon_{i+1}}$ over $Z_{\varepsilon_i}$ (fixing some sequence $\epsilon_1>\epsilon_2>\cdots$ converging to zero).

\section{Global homotopy theory}

This section shows one way to connect the homotopy theory of orbispaces developed thus far and \emph{global homotopy theory}.
We prove only what we need for the Pontryagin--Thom isomorphism; there is yet much to be worked out.
We refer to the treatment by Schwede \cite{schwedeglobal} for the foundations of global homotopy theory.
Global homotopy theory depends on a choice of set $\F$ of isomorphism classes of compact Lie groups; we will always take this set be the class of finite groups, and it will not be mentioned further.

\subsection{Global spaces}

Here we relate the category $\OrbSpc$ with the global homotopy category $\GloSpc$, whose objects we call \emph{global spaces}.

Let $\LL$ denote the topological category of finite-dimensional real vector spaces with a positive definite inner product and linear isometric (in particular, injective) maps.
An \emph{orthogonal space} is a continuous functor $F:\LL\to\kTop$ where $\kTop$ is the category of $k$-spaces \cite[Definition 1.1.1]{schwedeglobal} (a $k$-space is a topological space which is compactly generated and weakly Hausdorff).
In other words, it is the assignment to each $V\in\LL$ of a $k$-space $F(V)$ and to each pair $V,W\in\LL$ a continuous map $F(V)\times\LL(V,W)\to F(W)$, such that this rule is compatible with composition for triples $V,W,U\in\LL$.
The category of orthogonal spaces is denoted $\OrthSpc$.

A map of orthogonal spaces $F\to F'$ is called a \emph{global equivalence} \cite[Definition 1.1.2]{schwedeglobal} iff for every finite group $G$, every orthogonal $G$-representation $V$, and every diagram of solid arrows
\begin{equation}\label{gleqdiagspc}
\begin{tikzcd}
\partial D^k\ar[r]\ar[d]&F(V)^G\ar[d]\ar[r,dashed]&F(W)^G\ar[d,dashed]\\
D^k\ar[r]\ar[rru,dashed]&F'(V)^G\ar[r,dashed]&F'(W)^G
\end{tikzcd}
\end{equation}
there exists an orthogonal $G$-representation $W$ and an inclusion $V\hookrightarrow W$ such that after pushing forward under it, the bottom map $D^k\to F'(W)^G$ above may be homotoped rel boundary so as to lift to $F(W)^G$.
The category of global spaces $\GloSpc$ is the localization of the category of orthogonal spaces $\OrthSpc$ at the global equivalences.
(There is a model structure on $\OrthSpc$ whose weak equivalences are the global equivalences, giving an effective way to understand the localization $\GloSpc$ \cite[\S 1.2]{schwedeglobal}.)

An orthogonal space gives rise, in particular, to a representable map $\bigsqcup_{n\geq 0}F(\RR^n)/O(n)\to\bigsqcup_{n\geq 0}*/O(n)$.
Thus for any vector bundle $V$ with inner product over a stack $X$, we may pull back under the classifying map to obtain a representable map $F(V)\to X$.
Moreover, for any isometric inclusion $V\hookrightarrow W$ of vector bundles with inner product, we get a map $F(V)\to F(W)$ over $X$.
Denote by $\Vect^O(X)$ the category of vector bundles with inner product on $X$ and homotopy classes of injective isometric maps; this category is filtered.
There is thus a directed system over $\Vect^O(X)$ assigning to a vector bundle $V$ the set of homotopy classes of sections of $F(V)\to X$.
Note that the forgetful functor $\Vect^O(X)\to\Vect(X)$ is an equivalence, due to Lemma \ref{paracompactmetric} and the deformation retraction from injections to isometric injections given by $f\mapsto f(f^*f)^{-t/2}$.
Therefore in the event that the orthogonal space $F$ is pulled back from the category of finite-dimensional vector spaces and injective maps, we may simply take the direct limit over $\Vect(X)$ and forget about inner products.

Given a finite orbi-CW-complex $X$ and an orthogonal space $F$, let $\Hom(X,F)$ denote the direct limit over $V\in\Vect^O(X)$ of homotopy classes of sections of $F(V)\to X$.
This set $\Hom(X,F)$ is functorial in $X$ (pull back vector bundles) and in $F$.
This is only a reasonable definition because of enough vector bundles; in particular, enough vector bundles is used crucially in the following proof that maps from a finite orbi-CW-complex to an orthogonal space descends to a functor $(\OrbSpc^f)^\op\times\GloSpc\to\Set$.

\begin{lemma}\label{maptoorthspcdescend}
For a finite orbi-CW-complex and a global equivalence of orthogonal spaces $F\to F'$, the induced map $\Hom(X,F)\to\Hom(X,F')$ is a bijection.
\end{lemma}

\begin{proof}
Let $V$ be a vector bundle with inner product over $X$, let a section of $F'(V)\to X$ be given, and let us lift it (up to homotopy) to $(F)$ (after possibly enlarging $V$).
Since $X$ is finite, it suffices to do this lifting cell by cell.
So, fix a cell $(D^k,\partial D^k)\times\BB G$ of $X$.
The pullback of $V$ to this cell is classified by a map $D^k\times\BB G\to\bigsqcup_{n\geq 0}*/O(n)$, which up to homotopy (hence isomorphism by Lemma \ref{intervalpullback}) factors through $\BB G\to\bigsqcup_{n\geq 0}*/O(n)$, which is an orthogonal $G$-representation $V_0$.
Now the section of $F'(V)\to X$ pulled back from $D^k\times\BB G$ is a map $D^k\to F'(V_0)^G$, which over $\partial D^k$ we have lifted to $F(V_0)^G$.
We are thus exactly in the situation of the solid arrows in \eqref{gleqdiagspc}, so we conclude that there exists another orthogonal $G$-representation $W_0$ and an embedding $V_0\hookrightarrow W_0$ so that the desired lift exists over $D^k\times\BB G$ after pushing forward to $W_0$.
Now by enough vector bundles Theorem \ref{enough}, there exists a $W'$ on $X$ and an embedding $V\hookrightarrow W'$ which over $D^k\times\BB G$ factors through $V_0\hookrightarrow W_0\hookrightarrow W'_0$.
\end{proof}

There is much more to this story, however further precise discussion would take us too far afield.
There is a functor
\begin{align}\label{orbispacetoglobalspace}
\OrbSpc^f&\to\GloSpc,\\
X&\mapsto\Emb_X(E,-)\quad\text{($E/X$ faithful),}
\end{align}
where $\Emb_X(E,V)$ denotes the total space of the fibration over $X$ whose fiber over $x\in X$ is the space of embeddings $\Emb(E_x,V)$ (this is a space since $E$ is faithful).
The spaces $\Emb_X(E,-)$ form an inverse system on the category of vector bundles on $X$, and for an inclusion of faithful vector bundles $E\hookrightarrow E'$, the induced map $\Emb_X(E',-)\to\Emb_X(E,-)$ is a global equivalence \cite[Proposition 1.1.26(ii) and Definition 1.1.27]{schwedeglobal}.
A map of orbispaces $f:X\to Y$ induces maps $\Emb_X(f^*E_Y,-)\to\Emb_Y(E_Y,-)$ for any vector bundle $E_Y$ over $Y$.
Taking $E_Y$ to be faithful and choosing an embedding of $f^*E_Y$ into a faithful $E_X$, we obtain a map $\Emb_X(E_X,-)\to\Emb_X(f^*E_Y,-)\to\Emb_Y(E_Y,-)$.

\begin{conjecture}
For $X\in\OrbSpc^f$ and $F\in\GloSpc$, the set $\Hom(X,F)$ coincides with the morphisms from the image of $X$ under \eqref{orbispacetoglobalspace} to $F$.
\end{conjecture}

\begin{conjecture}
The functor \eqref{orbispacetoglobalspace} is fully faithful.
\end{conjecture}

Schwede \cite{schwedeequivalence} has shown that $\GloSpc$ is equivalent to $\PSh(\{\BB G\})$ where $\{\BB G\}\subseteq\OrbSpc$ is the full subcategory spanned by the objects $\{\BB G\}$ (this result requires a homotopical categorical context such as model categories or $\infty$-categories), and Gepner--Henriques \cite{gepnerhenriques} have defined via stacks a natural enlargement $\overline\OrbSpc$ of $\OrbSpc$ (resulting from gluing together cells $(D^k,\partial D^k)\times\BB G$ under arbitrary maps) and shown that the natural map $\overline\OrbSpc\to\PSh(\{\BB G\})$ is an equivalence.
Together this defines an equivalence $\overline\OrbSpc=\GloSpc$.

\begin{conjecture}
The restriction of the equivalence $\overline\OrbSpc=\GloSpc$ from \cite{gepnerhenriques,schwedeequivalence} to the full subcategory $\OrbSpc\subseteq\overline\OrbSpc$ coincides with the functor \eqref{orbispacetoglobalspace}.
\end{conjecture}

\subsection{Global classifying spaces}\label{globalclassifyingspaces}

We now recall various `global classifying spaces' from \cite{schwedeglobal}.

For any compact Lie group $G$, there is a `global classifying space' $\BB G\in\GloSpc$ \cite[Definition 1.1.27]{schwedeglobal} (there denoted $B_{gl}G$).
It is represented by the orthogonal space
\begin{equation}
(\BB G)(V):=\Emb(E,V)/G
\end{equation}
for any faithful $G$-representation $E$.
In particular, when $G=O(n)$, it is natural to take the defining representation $O(n)\acts\RR^n$, so we get
\begin{equation}
(\BB O(n))(V):=\Gr_n(V).
\end{equation}
Also note that when $G$ is finite, $\BB G\in\GloSpc$ is the image of $\BB G\in\OrbSpc$ under \eqref{orbispacetoglobalspace}.

Let us see that the global space $\BB G$ represents the functor of $G$-bundles on finite orbi-CW-complexes.
Maps from a finite orbi-CW-complex $X$ to $\BB G$ is the direct limit over $V/X$ of the space of embeddings of $E$ into $V$ modulo $G$ where $G\acts E$ is a faithful representation.
Denoting by $\Emb_X(E,V)$ the total space over $X$, we note that $\Emb_X(E,V)\to\Emb_X(E,V)/G$ is a principal $G$-bundle, so any section of $\Emb_X(E,V)/G\to X$ gives via pullback a principal $G$-bundle over $X$.
Conversely, given a principal $G$-bundle $P\to X$, a section of $\Emb_X(E,V)/G$ together with an isomorphism between the resulting pullback bundle and $P\to X$ is the same as an embedding $E\times_GP\hookrightarrow V$, and the space of such embeddings becomes contractible in the direct limit over $V$.

\begin{conjecture}
The functor \eqref{orbispacetoglobalspace} sends $\BB G\in\OrbSpc$ to $\BB G\in\GloSpc$.
\end{conjecture}

There are two natural global spaces $\bO$ and $\BO$ which generalize the classifying space $BO:=\varinjlim_nBO(n)$.
The global space $\bO$ \cite[Example 2.4.18]{schwedeglobal} is given by the orthogonal space
\begin{equation}
\bO(V):=\Gr_{\left|V\right|}(V\oplus\RR^\infty)
\end{equation}
in which to a map $V\hookrightarrow W$ we associate the map
\begin{equation}
\Gr_{\left|V\right|}(V\oplus\RR^\infty)\xrightarrow{\oplus(W/V)}\Gr_{\left|W\right|}(W\oplus\RR^\infty).
\end{equation}
The global space $\bO$ is the direct limit of $\BB O(n)\in\GloSpc$ \cite[Proposition 2.4.24]{schwedeglobal}.

Let us argue that $\bO$ classifies coarsely stable vector bundles of rank zero.
Maps from a finite orbi-CW-complex $X$ to $\bO$ are given by the direct limit over all vector bundles $E$ over $X$ of global sections of $\Gr_{\left|E\right|}(E\oplus\RR^\infty)$.
We may express this as the direct limit over both $n$ and $E$ of subbundles of rank $\left|E\right|$ of $E\oplus\RR^n$.
Equivalently, this is quotient bundles of $E\oplus\RR^n$ of rank $n$.
Now taking the direct limit over $E$, we realize every vector bundle has a homotopically unique surjection from $E\oplus\RR^n$ in the direct limit over $E$.
Thus what remains is the direct limit over $n$ of vector bundles over $X$, with passage from $n$ to $n+1$ acting as $\oplus\underline\RR$.
This is thus precisely rank zero coarsely stable vector bundles over $X$.

\begin{conjecture}
The functor \eqref{orbispacetoglobalspace} sends $\bO\in\OrbSpc$ to $\bO\in\GloSpc$.
\end{conjecture}

The global space $\BO$ \cite[Example 2.4.1]{schwedeglobal} is defined as the orthogonal space
\begin{equation}
\BO(V):=\Gr_{\left|V\right|}(V\oplus V)
\end{equation}
in which to a map $V\hookrightarrow W$ we associate the map
\begin{equation}
\Gr_{\left|V\right|}(V\oplus V)\to\Gr_{\left|V\right|}(W\oplus V)\xrightarrow{\oplus(W/V)}\Gr_{\left|W\right|}(W\oplus W).
\end{equation}
We argue that $\BO\in\GloSpc$ classifies stable vector bundles of rank zero (recall from Lemma \ref{noBO} that there is no $\BO\in\OrbSpc$ with this property).
Maps from a finite orbi-CW-complex $X$ to $\BO$ are the direct limit over $E/X$ of subbundles of rank $\left|E\right|$ of $E\oplus E$.
Let us choose to view this as the direct limit over pairs of vector bundles $E$ and $E'$ of subbundles of rank $\left|E\right|$ of $E\oplus E'$.
Taking the direct limit over $E'$, we see that this is just vector bundles of rank $\left|E\right|$ on $X$, and that in the remaining directed system over $E$, when going from $E_1$ to $E_2$, we add $E_2/E_1$.
Thus we get precisely rank zero stable vector bundles over $X$.

\subsection{Global spectra}

Here we relate the category $\OrbSp^f$ with the global stable homotopy category $\GloSp$ whose objects we call \emph{global spectra}.

Let $\OO$ denote the based topological category with the same objects as $\LL$ and with morphism space from $V$ to $W$ given by the Thom space of (i.e.\ the one-point compactification of the total space of) the tautological vector bundle ``$W/V$'' over $\LL(V,W)$.
An \emph{orthogonal spectrum} is a based continuous functor $F:\OO\to\kTop_*$ where $\kTop_*$ is the category of pointed $k$-spaces.
In other words, it is the assignment to each $V\in\OO$ of a based $k$-space $F(V)$ and to each pair $V,W\in\OO$ a based map $F(V)\wedge\OO(V,W)\to F(W)$, such that this rule is compatible with composition for triples $V,W,U\in\LL$.
The category of orthogonal spectra is denoted $\OrthSp$.

A map of orthogonal spectra $F\to F'$ is called a \emph{global equivalence} \cite[Equation (3.1.11) and Definition 4.1.3]{schwedeglobal} iff for every finite group $G$, every orthogonal $G$-representation $V$, every $k,\ell\geq 0$, and every diagram of based $G$-equivariant maps on the left:
\begin{equation}\label{gleqdiagsp}
\begin{tikzcd}
\partial D^k\wedge S^V\ar[r]\ar[d]&F(V\oplus\RR^\ell)\ar[d]\\
D^k\wedge S^V\ar[r]&F'(V\oplus\RR^\ell)
\end{tikzcd}
\quad
\implies
\quad
\begin{tikzcd}
\partial D^k\wedge S^W\ar[r]\ar[d]&F(W\oplus\RR^\ell)\ar[d]\\
D^k\wedge S^W\ar[r]\ar[ru,dashed]&F'(W\oplus\RR^\ell)
\end{tikzcd}
\end{equation}
there exists an orthogonal $G$-representation $W$ and an inclusion $V\hookrightarrow W$ such that after pushing forward under it, the square obtained on the right has a lift after homotoping the bottom map rel boundary (everything $G$-equivariantly).
The category of global spectra $\GloSp$ is the localization of the category of orthogonal spectra $\OrthSp$ at the global equivalences.
(There is a model structure on $\OrthSp$ whose weak equivalences are the global equivalences, giving an effective way to understand the localization $\GloSp$ \cite[\S 4.3]{schwedeglobal}.)

Given an orthogonal spectrum $F$ and a vector bundle $V\to X$ with inner product (over any stack $X$), we may define a representable map $F(V)\to X$ by applying $F$ to the fibers of $V$, just as we did for an orthogonal space.
This map $F(V)\to X$ is moreover equipped with a `basepoint' section.
A vector bundle $V\to X$ (where $X$ is still any stack) has an associated sphere bundle $S^V$ (fiberwise one-point compactification of $V$), again by defining it over $\bigsqcup_{n\geq 0}*/\GL_n(\RR)$ and pulling back under the classifying map; this is also equipped with a `basepoint' section.
We may thus consider, for any vector bundle $V\to X$ with inner product, based maps $S^V\to F(V)$ over $X$, where `based' means that the composition of the basepoint section of $S^V$ with the map $S^V\to F(V)$ is the basepoint section of $F(V)$.

Let us now argue that given an isometric inclusion of vector bundles with inner product $V\hookrightarrow W$, we may push forward a based map $S^V\to F(V)$ to obtain a based map $S^W\to F(W)$.
The structure of $F$ as an orthogonal spectrum gives a based map $F(V)\to F(W)$ over the total space of $S^{W/V}$, which over the basepoint section of $S^{W/V}$ (i.e.\ when pulled back under it) is the constant map to the basepoint section of $F(W)$.
Precomposing this map $S^{W/V}\times_XF(V)\to F(W)$ over $X$ with a map $S^V\to F(V)$ over $X$ defines a map $S^V\times_XS^{W/V}\to F(W)$, which we claim descends uniquely to a map $S^W\to F(W)$.
To prove this claim, it suffices to show that the obvious map $S^V\times_XS^{W/V}\to S^W$ (obvious since its existence over the universal classifying space is clear, so we can just pull back to $X$) pulls back under any map $Z\to X$ (where $Z$ is a topological space) to a topological quotient map.
Since vector bundles (inner product is now irrelevant) are locally trivial, this amounts to showing that $S^n\times S^m\times Z\to S^{n+m}\times Z$ is a topological quotient map for any topological space $Z$, which holds since the locus $(\{*\}\times S^m)\cup(S^n\times\{*\})\subseteq S^n\times S^m$ contracted by $S^n\times S^m\to S^{n+m}$ is compact.

We now show that global spectra give rise to cohomology theories on orbispectra.
Namely, we construct a functor
\begin{align}\label{glocoh}
(\OrbSp^f)^\op\times\GloSp&{}\to\Ab\\
W\times Z&{}\mapsto Z^0(W)
\end{align}
sending cofiber sequences to exact sequences.

We begin by defining $(W,Z)\mapsto Z^0(W)$ as a functor $(\OrbSpcPair^{f,-\Vect})^\op\times\OrthSp\to\Ab$.
For a finite orbi-CW-pair $(X,A)$ with vector bundle $\xi$ and an orthogonal spectrum $F$, we consider homotopy classes of based maps $S^V\to F(V\oplus\xi)$ over $X$, which over a neighborhood of $A$ are the constant map to the basepoint.
By the discussion in the paragraph just above, such homotopy classes of maps form a directed system over $V\in\Vect^O(X)$, and we define $F^0((X,A)^{-\xi})$ to be its direct limit (this set is naturally an abelian group by the usual argument involving $\underline\RR^2\subseteq E$).
As before, this is only a reasonable definition because of enough vector bundles.
Note that, for the purposes of computing the set of homotopy classes of based maps $S^V\to F(V\oplus\xi)$, we may consider just those which are properly supported over $X$ in the sense that they send a neighborhood of the fiberwise basepoint of $S^V$ to the fiberwise basepoint of $F(V\oplus\xi)$ (to see this, apply a map $S^V\to S^V$ which sends a neighborhood of the basepoint to the basepoint, which we may construct universally over $\bigsqcup_{n\geq 0}*/O(n)$).
That $F^0((X,A)^{-\xi})$ is a functor of $F\in\OrthSp$ is evident.
The following implies descent to a functor of $F\in\GloSp$.

\begin{lemma}
A global equivalence of orthogonal spectra $Z\to Z'$ induces an isomorphism $Z^0((X,A)^{-\xi})\to Z^{\prime 0}((X,A)^{-\xi})$.
\end{lemma}

\begin{proof}
Same as Lemma \ref{maptoorthspcdescend}.
\end{proof}

We now argue that $Z^0((X,A)^{-\xi})$ is functorial in $(X,A)^{-\xi}\in\OrbSpcPair^{f,-\Vect}$.
Suppose given a map $f:X\to Y$, an inclusion $f^*\zeta\hookrightarrow\xi$, and a section $s:X\to\xi/f^*\zeta$ such that $A$ is covered by $f^{-1}(B)$ and the locus where $\left|s\right|\geq\varepsilon$ for some $\epsilon>0$.
Now given a map $S^V\to F(V\oplus\zeta)$ over $Y$ supported away from $B$, we may pull it back to obtain a map $S^{f^*V}\to F(f^*V\oplus f^*\zeta)$ over $X$ supported away from $f^{-1}(B)$.
We then further pair with $s$, viewed as a section of $S^{\xi/f^*\zeta}$, to obtain a map $S^{f^*V}\to F(f^*V\oplus\xi)$ supported away from $A$.
To finish the construction of \eqref{glocoh}, it suffices to show that morphisms $W$ and $S$ are sent to isomorphism.
That morphisms $W$ are sent to isomorphism follows from enough vector bundles Theorem \ref{enough} (restriction of vector bundles is cofinal).
That morphisms $S$ are sent to isomorphisms is immediate from the definition.

\begin{proposition}
For any global spectrum $Z$, the functor $Z^0$ sends cofiber sequences in $\RepOrbSp^f$ to exact sequences.
\end{proposition}

\begin{proof}
We are to show that $Z^0(Y,B)\leftarrow Z^0(X,A)\leftarrow Z^0(X,A\cup_BY)$ is exact.
The composition is evidently zero.
Now suppose we have a section over $(X,A)$ whose restriction to $(Y,B)$ is null-homotopic after stabilizing by a vector bundle on $Y$.
The restriction map on vector bundles is cofinal by enough vector bundles, so without loss of generality we are in the situation of a section on $(X,A)$ whose restriction to $(Y,B)$ is null-homotopic rel $B$.
Now $(Y,B)$ has a nice neighborhood inside $(X,A)$, so we can extend this null-homotopy to a homotopy of sections over $(X,A)$ to become supported away from $A\cup_BY$.
\end{proof}

Define $Z^i(W):=Z^0(\Sigma^{-i}W)$, so the Puppe sequence now gives a bi-infinite long exact sequence of the expected form for any cofiber triple in $\RepOrbSp^f$.

\subsection{Global Thom spectra}\label{globalthomspectra}

We now recall the so-called \emph{global Thom spectra} \cite[\S 6]{schwedeglobal}, whose associated cohomology theories are called homotopical cobordism theories.
They are given by the following orthogonal spectra:
\begin{align}
\SSS(V):={}&S^V,\\
\mO(V):={}&\Gr_{\left|V\right|}(V\oplus\RR^\infty)^\tau,\\
\MO(V):={}&\Gr_{\left|V\right|}(V\oplus V)^\tau,
\end{align}
where $\tau$ denotes the tautological vector bundle.
The structure maps are induced by those of the corresponding $\bO$ and $\BO$ defined above, just passing to Thom spaces as appropriate.
In the present context of orthogonal spectra, Thom space always means the one-point compactification of the total space.
These are ring spectra in various senses, however we will not discuss this precisely, instead referring to Schwede \cite[\S 6]{schwedeglobal}.

There is a canonical `unit element' $\1\in\SSS^0(X)$ for any orbispace $X$, namely given in the definition of $\SSS^0(X)$ by taking $E=0$ and taking the unit section of $\Omega^E\SSS(E)=S^0$ over $X$.

\begin{remark}
It is natural to conjecture that $S^0\in\OrbSpc_*$ is sent to $\SSS$ and that $\varinjlim_n\BB O(n)^{\underline\RR^n-\xi_n}\in\OrbSp$ is sent to $\mO$, under natural functors to $\GloSp$.
As we have not defined an orbispace $\BO$, we cannot define an orbispectrum corresponding to $\MO$.
\end{remark}

\subsection{Pontryagin--Thom isomorphism}

Theorem \ref{ptfinal} is the combination of Propositions \ref{ptS} and \ref{ptThom} below.

\begin{proposition}\label{ptS}
There is a bijection $\SSS^0(DW)\xrightarrow\sim\Omega^\fr_0(W)$ natural in $W\in\RepOrbSp^f$.
\end{proposition}

For a compact orbifold-with-boundary $X$, this bijection sends the unit element $1\in\SSS^0(X)$ to the fundamental class $[X]\in\Omega^\fr_0((X,\partial X)^{-TX})$.

\begin{proof}
Given a compact orbifold pair $(X,A)$ and a vector bundle $\xi$ over $X$, we define a map
\begin{equation}\label{ptmapinit}
\SSS^0((X,\partial X-A^\circ)^{\xi-TX})\to\Omega^\fr_0((X,A)^{-\xi})
\end{equation}
as follows.
The value of $\SSS^0((X,\partial X-A^\circ)^{\xi-TX})$ is the direct limit over vector bundles $E/X$ of (homotopy classes of) sections of $\Omega^{E\oplus\xi}S^{E\oplus TX}$ over $X$ supported away from $\partial X-A^\circ$.
Equivalently, this is sections $s$ of $S^{E\oplus TX}$ (the fiberwise one-point compactification) over the total space of $E\oplus\xi$ over $X$ whose zero set $s^{-1}(0)$ is proper over $X$ and disjoint from the inverse image of $\partial X-A^\circ$.
Such data defines a compact derived orbifold chart with boundary $(D,E\oplus TX,s)$ ($D$ an open subset of the total space of $E\oplus\xi$), representable over $(X,A)$, with a stable isomorphism between its tangent bundle and $\xi$; this defines an element of $\Omega_0^\fr((X,A)^{-\xi})$.
This construction is compatible with enlarging $E$ and sends homotopies of sections to bordisms, hence defines the desired map \eqref{ptmapinit}.

Let us argue that \eqref{ptmapinit} defines a natural transformation of functors $\SSS^0(DW)\to\Omega^\fr_0(W)$ of $W\in\RepOrbSp^f$.
By Proposition \ref{reporbsecondlocalization} and the universal property of localization (and of direct limit), it suffices to show that this defines a natural transformation of functors out of $\RepOrbSpcPair_{N,k}^{f,-\xi}$ for every $\xi\in\Vect(R(*)_{N,k+2})$ (compatible with the functors modifying $\xi$ and $N,k$).
Compatibility with the functors modifying $\xi$ and $N,k$ is immediate; the real content is to check that the following diagram commutes
\begin{equation}
\begin{tikzcd}
\SSS^0(DW)\ar[r]\ar[d]&\Omega^\fr_0(W)\ar[d]\\
\SSS^0(DZ)\ar[r]&\Omega^\fr_0(Z)
\end{tikzcd}
\end{equation}
for any map $W\to Z$ in $\RepOrbSpcPair_{N,k}^{f,-\xi}$.
We may assume that this map $W\to Z$ is a smooth embedding of compact orbifold pairs $(X,A)\to(Y,B)$, namely $X\hookrightarrow Y$ is a smooth embedding and $A=X\cap\partial Y$ meeting transversally (so $X$ has corners at the boundary of $A$), desuspended by a vector bundle $\xi$ on $R(*)_{N,k+2}$ where the isotropy groups of $X$ and $Y$ have order $\leq N$ and $X,Y$ have dimension $\leq k$.
In this case, the map on duals is simply the evident map $(Y,\partial Y-B^\circ)\dashrightarrow(X,\partial X-A^\circ)^{TY/TX}$ desuspended by $TY$ and suspended by $\xi$.
Now commutativity of the above diagram is evident.

It remains to show that the natural transformation $\SSS^0(DW)\to\Omega^\fr_0(W)$ is a bijection for every $W\in\RepOrbSp^f$ (which we may take to be of the form $(X,A)^{-\xi}$ for a compact orbifold pair $(X,A)$ with a vector bundle $\xi$ over $X$).
To show surjectivity, let $(D,E,s)$ be a derived orbifold-with-boundary chart with a representable map $(D,\partial D)\to(X,A)$ and a stable isomorphism $TD-E=\xi$.
By Corollary \ref{embedding}, the map from $D$ to $X$ can be replaced by a smooth embedding by replacing $X$ with the unit disk bundle of a vector bundle over $X$ (and $A$ is replaced with its inverse image in this total space).
Thus we may assume $D$ is a suborbifold of $X$; choosing a nice collar near $\partial X$, we may further assume that it meets $\partial X$ transversely, precisely along $\partial D$.
Now we may stabilize our derived orbifold chart by $TX/TD$ so that $D$ is in fact an open subset of $X$.
Now we have a stable isomorphism $TX-E=\xi$, namely an isomorphism $E\oplus\xi\oplus F=TX\oplus F$ for some vector bundle $F$.
By further replacing $X$ with the total space of $F$ and stabilizing our derived orbifold chart by $F$, this becomes a true isomorphism of vector bundles $TX=E\oplus\xi$ over $D$ (which remains an open subset of $X$).
Further stabilizing by $\xi$ (thus adding $\xi$ to both $E$ and $TX$) ensures that $E$ extends to all of $X$, together with the isomorphism $TX=E\oplus\xi$.
Now the section $s$ cutting out our derived orbifold is, after extension as `infinity' to the rest of $X$, a section of $S^E$.
This gives, by definition, an element of $\SSS^0((X,\partial X-A^\circ)^{-E})=\SSS^0((X,\partial X-A^\circ)^{\xi-TX})$ which maps to our given element of $\Omega^\fr_0((X,A)^{-\xi})$.

Finally, injectivity is just a relative version of surjectivity.
We are given two elements of $\SSS^0((X,\partial X-A^\circ)^{\xi-TX})$ with the same image in $\Omega^\fr_0((X,A)^{-\xi})$.
Applying `rel boundary' the same procedure used to prove surjectivity to the derived bordism relating the images of our two given elements of $\SSS^0((X,\partial X-A^\circ)^{\xi-TX})$ produces a homotopy between them.
\end{proof}

\begin{proposition}\label{ptThom}
For $W\in\RepOrbSp^f$, there are natural bijections
\begin{align}
\label{ptmoresult}\mO^0(DW)&{}\xrightarrow\sim\Omega_0(W),\\
\label{ptMOresult}\MO^0(DW)&{}\xrightarrow\sim\Omega_0^\der(W).
\end{align}
\end{proposition}

\begin{proof}
We follow the proof of Proposition \ref{ptS}.

Given a compact orbifold pair $(X,A)$ and a vector bundle $\xi$ over $X$, we define maps
\begin{align}
\label{ptmapinitmo}\mO^0((X,\partial X-A^\circ)^{\xi-TX})&{}\to\Omega_0((X,A)^{-\xi})=\Omega_{\left|\xi\right|}^{\xi+\cst,\der}(X,A)\\
\label{ptmapinitMO}\MO^0((X,\partial X-A^\circ)^{\xi-TX})&{}\to\Omega_0^\der((X,A)^{-\xi})=\Omega_{\left|\xi\right|}^\der(X,A)
\end{align}
as follows.
The values of $\mO^0((X,\partial X-A^\circ)^{\xi-TX})$ and $\MO^0((X,\partial X-A^\circ)^{\xi-TX})$ are the direct limits over vector bundles $E/X$ of (homotopy classes of) sections of
\begin{align}
&\Omega^{E\oplus\xi}\Gr_{\left|E\right|+\left|TX\right|}(E\oplus TX\oplus\RR^{\left|E\right|+\left|TX\right|})^\tau,\\
&\Omega^{E\oplus\xi}\Gr_{\left|E\right|+\left|TX\right|}(E\oplus TX\oplus E\oplus TX)^\tau,
\end{align}
over $X$ supported away from $\partial X-A^\circ$.
Equivalently, this is open subsets $U$ of the total space of $E\oplus\xi$ over $X$ carrying a rank $\left|E\right|+\left|TX\right|$ vector bundle $V\subseteq E\oplus TX\oplus\RR^{\left|E\right|+\left|TX\right|}$ (resp.\ $\subseteq E\oplus TX\oplus E\oplus TX$) and a section $s:U\to V$ whose zero set $s^{-1}(0)$ is proper over $X$ and disjoint from the inverse image of $\partial X-A^\circ$.
Such data defines a compact derived orbifold chart with boundary $(U,V,s)$, representable over $(X,A)$, with a stable isomorphism between its tangent bundle and $TX+E+\xi-V$ (which in the case of $\mO$ is identified with $\xi+(E\oplus TX\oplus\RR^{\left|E\right|+\left|TX\right|})/V-\RR^{\left|E\right|+\left|TX\right|}$).
We thus obtain an element of $\Omega_{\left|\xi\right|}^{\xi+\cst,\der}(X,A)=\Omega_0^{\cst,\der}((X,A)^{-\xi})=\Omega_0((X,A)^{-\xi})$ (resp.\ $\Omega_0^\der((X,A)^{-\xi})$).
This construction is compatible with enlarging $E$ and sends homotopies of sections to bordisms, hence defines the desired maps \eqref{ptmapinitmo}--\eqref{ptmapinitMO}.

The proof that \eqref{ptmapinitmo}--\eqref{ptmapinitMO} define natural transformations of functors \eqref{ptmoresult}--\eqref{ptMOresult} of $W\in\RepOrbSp^f$ is exactly as in the proof of Proposition \ref{ptS}.

It remains to show that the natural transformations \eqref{ptmoresult}--\eqref{ptMOresult} are bijections for $W=(X,A)^{-\xi}$ for a compact orbifold pair $(X,A)$ with a vector bundle $\xi$ over $X$.
As before, the argument for injectivity is a relative version of that for surjectivity, so we will just explain surjectivity.
To show surjectivity, let $(D,V,s)$ be a derived orbifold-with-boundary chart with a representable map $(D,\partial D)\to(X,A)$ and, in the case of $\mO$, a vector bundle $B$ and a stable isomorphism $TD-V=\xi+B-\RR^{\left|B\right|}$ (in the case of $\MO$, with $\dim TD-\left|E\right|=\left|\xi\right|$).
As in the proof of Proposition \ref{ptS}, we may homotope and stabilize to reduce to the case that $D$ is an open subset of $X$.
Now further stabilize both $X$ and $D$ by the vector bundle $\xi$, so that we now have an isomorphism $TX=\xi\oplus E$ where $E$ is the tangent bundle before stabilizing.
We seek an element of $\mO^0((X,\partial X-A^\circ)^{-E})$ (resp.\ $\MO^0$); more specifically, we will produce a section of $\Gr_{\left|E\right|}(E\oplus E)^\tau$ (resp.\ $\Gr_{\left|E\right|}(E\oplus E)^\tau$).
We have a stable isomorphism $E\oplus\RR^{\left|B\right|}=V\oplus B$ (resp.\ an equality $\left|V\right|=\left|E\right|$); in the former case we may stabilize $X$ and $D$ to get a true isomorphism.
For $\mO^0$, we want to embed $V\hookrightarrow E\oplus\RR^{\left|E\right|}$, which we get from the isomorphism $E\oplus\RR^{\left|B\right|}=V\oplus B$ once $\left|E\right|\geq\left|B\right|$ which we can achieve by stabilizing.
For $\MO^0$, we want $V\hookrightarrow E\oplus E$.
Stabilizing to $V'$ and $E'$ allows us to embed $V\hookrightarrow E'\oplus E$ hence $V'\hookrightarrow E'\oplus E'$.
\end{proof}

\section{Bordism and stable maps}

In this final section, we apply the Pontryagin--Thom principle to describe morphism spaces in $\RepOrbSp^f$ and $\OrbSp^f$ in terms of derived orbifold bordism.

\begin{proof}[Proof of Theorem \ref{stablerepresentablemapsbordism}]
Fix a compact orbifold pair $(X,A)$ with a vector bundle $\xi$ and a finite orbi-CW-pair $(Y,B)$ with vector bundle $\zeta$.
Given a map in $\RepOrbSp^f$ (resp.\ $\OrbSp^f$)
\begin{equation}\label{maptoassociate}
D((X,A)^{-\xi})=(X,\partial X-A^\circ)^{\xi-TX}\to(Y,B)^{-\zeta},
\end{equation}
we associate as follows a bordism class of derived orbifold chart with boundary $(C,\partial C)$ with a map $(C,\partial C)\to(X,A)\times(Y,B)$ whose projections to $X$ and $Y$ (resp.\ to $X$) are representable and with a stable isomorphism between its tangent bundle and $\xi+\zeta$.
The data of a map \eqref{maptoassociate} consists of a vector bundle $E$ over $X$, an open subset $U$ of the total space of $E\oplus\xi$, a (representable) map $h:U\to Y$, an embedding $h^*\zeta\hookrightarrow TX\oplus\xi$, and a section $s$ of the quotient whose zero set is proper over $X$ such that $\partial X-A^\circ$ is contained in $f^{-1}(B)$ union the locus where $\left|s\right|\geq\varepsilon$ for some $\varepsilon>0$.
This data defines for us a compact derived orbifold chart $(U,(TX\oplus E)/h^*\zeta,s)$ which has the desired form by inspection.
Homotopies of maps evidently induce bordisms.

Let us argue that this association (of a bordism class to a stable map) is natural in $(Y,B)^{-\zeta}$.
To make sense of this statement, we should note that bordism of derived orbifolds of the requisite form is indeed a functor of $(Y,B)^{-\zeta}\in\RepOrbSpcPair^{f,-\Vect}$ (resp.\ $\OrbSpcPair^{f,-\Vect}$), where a map $(Y,B)^{-\zeta}\to(Y',B')^{-\zeta'}$ given by $q:Y\to Y'$, $q^*\zeta'\hookrightarrow\zeta$, and $s:Y\to\zeta/q^*\zeta'$ pushes forward a derived orbifold mapping $(Y,B)$ under $q$ and adds $\zeta/q^*\zeta'$ to the obstruction space and $s$ to the obstruction section.
This evidently descends to $\RepOrbSp^f$ (resp.\ $\OrbSp^f$) due to sending to isomorphisms the morphisms $W$ (obvious) and $S$ (same as Proposition \ref{thomiso}).
Now to see that the association of a bordism class to a stable map is natural, it suffices to show it is a natural transformation of functors of $(Y,B)^{-\zeta}\in\RepOrbSpcPair^{f,-\Vect}$ (resp.\ $\in\OrbSpcPair^{f,-\Vect}$) due to the universal property of localization.
This is evident by inspection.

Next, to see naturality in $(X,A)^{-\xi}\in\RepOrbSp^f$, we may argue as in the proof of Proposition \ref{ptS}: it suffices to check naturality as a functor out of $\RepOrbSpcPair_{N,k}^{f,-\xi}$ for $\xi\in\Vect(R(*))_{N,k+2}$, and this can be seen by inspection upon arranging maps to be smooth embeddings of orbifolds.

It remains to show that this association of a bordism class to a map \eqref{maptoassociate} is bijective.
As in the proof of Proposition \ref{ptS}, injectivity is simply a relative version of surjectivity, so we will just prove surjectivity.
Thus, suppose given a compact derived orbifold chart with boundary $(D,V,s)$ with a representable map $f$ to $X$, a (representable) map $g$ to $Y$ with $\partial D\subseteq f^{-1}(A)\cup g^{-1}(B)$, and a stable isomorphism between its tangent bundle $TD-V$ and $f^*\xi+g^*\zeta$.
By replacing $(X,A)$ with the total space of a vector bundle over it, we may assume the map $D\to X$ is a smooth embedding.
By stabilizing $(D,V,s)$, we may assume $D\to X$ is an open inclusion, so we have a stable isomorphism $TX=V\oplus\xi\oplus g^*\zeta$ over $D$.
Now further stabilize by $\xi$ so that we have an everywhere defined isomorphism $TX=\xi\oplus E$ (so $E$ is the tangent bundle of $X$ before stabilizing).
The resulting stable isomorphism $E=V\oplus g^*\zeta$ may be turned into a genuine isomorphism by further stabilization.
We want a map $(X,\partial X-A^\circ)^{-E}\to(Y,B)^{-\zeta}$, and this is precisely what we have: our open subset of $X$ is $D$, which has a (representable) map $g:D\to Y$, we have an embedding $g^*\zeta\hookrightarrow g^*\zeta\oplus V=E$, and we have a section of the quotient $V$ namely the obstruction section.
\end{proof}

\begin{example}\label{stablerepmapsbg}
We describe the set of stable (representable) maps $\BB G\to\BB H$ for finite groups $G$ and $H$.
Such maps (i.e.\ morphisms in $\RepOrbSp^f$ and $\OrbSp^f$) are, according to Theorem \ref{stablerepresentablemapsbordism}, in bijection with bordism classes of derived orbifolds $C$ with a representable map to $\BB G$, a (representable) map to $\BB H$, and a stable isomorphism $TC=0$.
By Wasserman's Theorem \ref{wasserman}, this is the same as bordism classes of orbifolds $C$ with the requisite (representable) maps and stable framing.
Now $C$ has dimension zero, so it must be a disjoint union of $\BB K$ for some finite groups $K$; the only bordisms between these have the form $\BB K\times[0,1]$, so bordism is just homotopy.
A homotopy class of (representable) map $\BB K\to\BB G$ is a $G$-conjugacy class of (injective) homomorphism $K\to G$.
A stable framing of $\BB K$ is, according to Example \ref{stablevboverbg}, an element of the product of $\ZZ/2$ over all irreducible real representations of $K$ with $\End(\rho)=\RR$.
We thus obtain a group theoretic description of the morphism space $\BB G\to\BB H$ in $\RepOrbSp^f$ and in $\OrbSp^f$.
\end{example}

Stated slightly differently, Theorem \ref{stablerepresentablemapsbordism} says that the category $\RepOrbSp^f$ may be described as follows.
The objects of $\RepOrbSp^f$ are denoted $(X,A)^{-\xi}$ where $(X,A)$ is a compact orbifold pair and $\xi$ is a stable vector bundle over $X$.
The morphisms $(X,A)^{-\xi}\to(Y,B)^{-\zeta}$ are bordism classes of derived orbifolds $(C,\partial C)\to(X,\partial X-A^\circ)\times(Y,B)$ whose projections to $X$ and to $Y$ are representable, equipped with a stable isomorphism $TC=TX-f^*\xi+g^*\zeta$.
Composition is given by derived fiber product.
In this description, the action of duality $D$ is obvious: it trades $(X,A)^{-\xi}$ for $(X,\partial X-A^\circ)^{\xi-TX}$ with the evident action on morphisms.

There is a notable omission in Theorem \ref{stablerepresentablemapsbordism}: we have no idea what category one gets if one allows both maps to $(X,A)$ and to $(Y,B)$ to be arbitrary (not required to be representable).
The resulting category has an apparent involution $D$, but that's all this author knows.

Using Theorem \ref{stablerepresentablemapsbordism}, we may associate to any map $X\to Y$ in $\RepOrbSp^f$ a map $X\wedge DY\to R(*)$ in $\RepOrbSp^f$ as follows.
Given a derived orbifold of the shape prescribed by Theorem \ref{stablerepresentablemapsbordism} to specify a map $X\to Y$, we simply note that the same derived orbifold also defines a map $X\wedge DY\to R(*)$ by taking the product of the two maps and appealing to the canonical map to $R(*)$.

In particular, there is a canonical pairing $X\wedge DX\to R(*)$ induced by the identity map $X\to X$ (equivalently $DX\to DX$).
It may be described concretely as follows.
Let $(X,A)$ be a compact orbifold pair carrying a vector bundle $\xi$.
The diagonal map is a map
\begin{equation}
(X,\partial X)\to(X,A)\times(X,\partial X-A^\circ).
\end{equation}
We may suspend/desuspend to define a map $(X,\partial X)^{-TX}\to(X,A)^{-\xi}\times(X,\partial X-A^\circ)^{\xi-TX}$ and then dualize to obtain
\begin{equation}
(X,\partial X-A^\circ)^{\xi-TX}\times(X,A)^{-\xi}\to X
\end{equation}
which we may then compose with the map $X\to R(*)$.
This defines a map $DZ\wedge Z\to R(*)$ for $Z=(X,A)^{-\xi}$.
Tracing through the definition of the bijection in Theorem \ref{stablerepresentablemapsbordism}, it is immediate that this is indeed the canonical pairing $Z\wedge DZ\to R(*)$ as described above.

\bibliographystyle{amsplain}
\bibliography{ptorbifold}
\addcontentsline{toc}{section}{References}

\end{document}